\documentclass[12pt,bibliography=totoc,twoside,letterpaper]{scrartcl}
\usepackage[english]{babel} 
\usepackage[latin1]{inputenc}
\usepackage[babel,german=quotes]{csquotes}
\usepackage{amsmath,amsfonts,amssymb,amsthm}
\usepackage{mathrsfs}
\usepackage{mathtools}
\usepackage[expansion, final]{microtype} 
 
\usepackage{graphicx}
\usepackage{subfigure}
\usepackage{wrapfig}	
\usepackage{color}

\usepackage{chngcntr}

\renewcommand{\Im}{\operatorname{Im}}
\renewcommand{\Re}{\operatorname{Re}}

\newcommand{\C}{\mathbb{C}}

\normalsize
\usepackage[hang,footnotesize,bf]{caption}
\usepackage[inner=1.6cm,outer=1.65cm]{geometry}
\usepackage{caption}

\usepackage[headsepline]{scrlayer-scrpage}
\automark[subsection]{section}
\automark*[subsection]{section}

\usepackage{titling}

\newtheoremstyle{break}
  {\topsep}{\topsep}
  {\itshape}{}
  {\bfseries}{}
  {\newline}{}
\theoremstyle{break}

\newtheorem{satz}{Satz} \numberwithin{satz}{section}
\newtheorem{lemma}[satz]{Lemma}
\newtheorem{korollar}[satz]{Corollary}
\newtheorem{corollary}[satz]{Corollary}
\newtheorem{proposition}[satz]{Proposition}
\newtheorem{theorem}[satz]{Theorem}
\newtheorem{conjecture}[satz]{Conjecture}

\newtheoremstyle{glatt} 
  {\topsep}{\topsep}
  {\normalfont}{}
  {\bfseries}{}
  {\newline}{}
\theoremstyle{glatt}
\newtheorem{definition}[satz]{Definition}

\newtheorem{remark}[satz]{Remark}

\mathchardef\ordinarycolon\mathcode`\: 
\mathcode`\:=\string"8000
\begingroup \catcode`\:=\active
  \gdef:{\mathrel{\mathop\ordinarycolon}}
\endgroup

\usepackage[bookmarksopen=true]{hyperref} 

\usepackage[style=alphabetic, 
backend=biber, 
isbn=false,
maxbibnames=50,
maxcitenames=3,
autocite=inline,
block=space,
backrefstyle=three+,
]{biblatex}
\bibliography{Dis_Lit} 
\DefineBibliographyStrings{ngerman}{
    bibliography={Literaturverzeichnis}
}

\title{Hyperkähler metrics on the moduli space of weakly parabolic Higgs bundles}
\author{Maximilian Holdt}
\date{}

\begin{document}

\vbadness=10000

\hypersetup{pageanchor=false}

\begin{titlingpage}
\maketitle
\begin{abstract}
	We use the theory of Gaiotto, Moore and Neitzke to construct a set of Darboux coordinates on the moduli space $\mathcal{M}$ of weakly parabolic 
	$SL(2,\mathbb{C})$-Higgs bundles. For generic Higgs bundles ($\mathcal{E},R\Phi)$ with $R\gg 0$ the coordinates are shown to be dominated by a leading term that is given by the coordinates for a corresponding simpler space of limiting configurations and we prove that the deviation from the limiting term is given by a remainder that is exponentially suppressed in $R$.
	
	We then use this result to solve an associated Riemann-Hilbert problem and construct a twistorial hyperkähler metric $g_{\text{twist}}$ on $\mathcal{M}$. Comparing this metric to the simpler semiflat metric $g_{\text{sf}}$, we show that their difference is $g_{\text{twist}}-g_{\text{sf}}=O\left(e^{-\mu R}\right)$, where $\mu$ is a minimal period of the determinant of the Higgs field.
	
\end{abstract}
\end{titlingpage}

\tableofcontents

\hypersetup{pageanchor=true}
\newpage
\section{Introduction}

	In 2009 D. Gaiotto, G. Moore and A. Neitzke published a theory which they argued would lead to a very concrete way of constructing hyperkähler metrics on 
	a large class of different manifolds using twistor theory on integrable systems. The general theory was first presented in \cite{Gaiotto.2010} and in 
	\cite{Gaiotto.2013} an explicit adaptation for the 
	moduli space $\mathcal{M}$ of weakly parabolic Higgs bundles on a compact Riemann surface $\mathcal{C}$ was developed which also introduced the notion of 
	spectral networks that were further 
	developed in \cite{Gaiotto.2013b}. While their general theory remained conjectural, a lot of progress has been made since then, proving some of their 
	conjectures 
	in certain special cases. In 2017 C. Garza was able to show in \cite{Garza.2017b} how the construction works for the so called Pentagon 
	case, an integrable system with torus fibers that exhibit some nodal structure on a certain locus, very similar to the known behavior of the Higgs bundle 
	moduli space. Notably he was able to show how the Ooguri-Vafa metric is a suitable model for the singular fibers. The most recent step in this direction 
	was 
	taken by I. Tulli, who showed in \cite{Tulli.2019} how the construction works on a specific moduli space of framed wild harmonic bundles on 
	$\mathbb{CP}^1$. Thus the progress in the general theory has mostly been made for very specific examples while the conjectures were intended for a much 
	broader context. 
	
	One of the most prominent conjectures of Gaiotto, Moore and Neitzke (GMN) concerns the comparison of the natural $L^2$-metric $g_{L^2}$ which is defined on all of $\mathcal{M}$ and the semiflat metric $g_{\text{sf}}$ for large Higgs fields, which is defined on the regular locus $\mathcal{M}'\subset\mathcal{M}$ where determinants of the Higgs field have only simple zeroes. GMN argued for the exponential decay of the difference of these to metrics on $\mathcal{M}'$ for large Higgs fields, i. e. when considering generic 
	Higgs bundles $(\mathcal{E},R\Phi)$ for large $R$ the difference of the two metrics should be $g_{L^2}-g_{\text{sf}}=O\left(e^{-\gamma R}\right)$. Additionally they 
	conjectured that $\gamma$ is given by the length of the shortest saddle connection in the metric on $\mathcal{C}$ induced by Det$\Phi$.
	This problem was later also tackled by R. Mazzeo, J. Swoboda, H. Weiß and F. Witt using different methods for the case of regular 
	$\text{SU}(2)$-Higgs bundles in \cite{Mazzeo.2012}, \cite{Mazzeo.2016} and \cite{Mazzeo.2017}. They were able to prove a polynomial decay by introducing the notion of 
	limiting configurations. For large Higgs fields these almost solve Hitchin's equations, leaving only an exponentially small remainder term which allows for 
	good approximations of the natural metric. The decay rate was later on improved to be an actual exponential decay by D. Dumas and A. Neitzke for a 
	special case in \cite{Dumas.2019} and in generality by L. Fredrickson in \cite{Fredrickson.2020b}, who also expanded the theory for the $\text{SU}(N)$ 
	case in \cite{Fredrickson.2018}.
	
	Although it seemed reasonable to suspect that the two theories (twistorial construction on the integrable system and limiting configurations) could 
	complement each other they weren't put together on a general class of Higgs bundle moduli spaces until now, as the twistorial construction needs to 
	consider Higgs bundles with singularities while the limiting configurations until recently were only developed for regular Higgs bundles. But this has 
	changed now as L. Fredrickson, R. Mazzeo, J. Swoboda and H. Weiß (FMSW) were able to develop their theory for the weakly and strongly parabolic Higgs 
	bundles in 
	\cite{Fredrickson.2020}. This finally allows to merge both theories and show that they fit together quite nicely as they give very concrete expressions 
	and good approximations for a lot of important objects on $\mathcal{M}$, namely a set of Darboux coordinates, a holomorphic symplectic form and a 
	hyperkähler metric. Thus we will argue in this article that the conjectural picture of GMN works for moduli spaces of weakly parabolic Higgs bundles of 
	rank $2$ on 
	Riemann surfaces of arbitrary genus with an arbitrary number of parabolic points. In particular a hyperkähler metric can be constructed in a 
	twistorial way whose difference to the semiflat metric decays exponentially with the conjectured optimal decay rate.
	
	Let us briefly summarize the key steps in the construction which will serve as a guide for our endeavor. One starts with the moduli space of weakly 
	parabolic Higgs bundles regarded as a space of Higgs pairs $(\Phi,A)$ consisting of a Higgs field $\Phi$ and a connection $A$ on some fixed rank $2$ vector bundle $V$ 
	over a Riemann surface $\mathcal{C}$. The basic theory of this moduli space will be explained in section \ref{parabolic_higgs} where we mostly follow 
	the set up of \cite{Fredrickson.2020}, as their results are fundamental for our work here. 
	
	To each Higgs pair one then associates a decorated triangulation of $\mathcal{C}$. This entails a triangulation of $\mathcal{C}$ that is given by certain curves that are geodesics when considered in the metric on $\mathcal{C}$ induced by Det$\Phi$. On the other hand the decoration consists of a choice of a $1$-dimensional subspace of $V$ at each vertex of the triangulation. In our setting this subspace if given by a solution $s$ to the flatness equation $\nabla(\zeta)s=0$ where $\nabla(\zeta)=\tfrac{R}{\zeta}\Phi+A+R\zeta\Phi^\ast$ (for $R$ positive and $\zeta$ a complex valued parameter) is a connection on $V$ that is flat for bundles of parabolic degree $0$. We will constrain ourselves to this case and describe the restriction later on in more detail. As V. Fock and A. Goncharov proved in \cite{Fock.2006}, one can construct a coordinate system for the moduli space of flat connections out of these decorated triangulations. Following \cite{Gaiotto.2013} section \ref{dec_and_coord} will be concerned with constructing the triangulation and explaining how the coordinates of Fock and Goncharov relate to coordinates of the Higgs bundle moduli space.
	
	The main difficulty in this construction comes from solving the flatness equation. Locally this is a non-autonomous linear ODE whose components are only partially known. For the construction to work, one has to be able to chose certain solutions in a unique way (up to scalar multiplication) such that they have the right asymptotics for $\zeta\to0$ and $\zeta\to\infty$. The existence of such sections, called \textit{small flat sections}, was conjectured in \cite{Gaiotto.2013}, but has not been rigorously proven until now. In section \ref{local_desc} we will use the results of \cite{Fredrickson.2020} to obtain a local description of Higgs pairs as the sum of a diagonal leading term and a remainder that is exponentially suppressed, the general idea being that if the connection matrix of $\nabla(\zeta)$ has this composite form, then the solutions of $\nabla(\zeta)=0$ should also be determined by an explicitly known leading term and some remainder that is exponentially suppressed. 
	
	In \cite{Gaiotto.2013} GMN already used an argument along these lines, where they used the WKB approximation to conjecture the right behaviour of the 
	solutions. Here we will not use this approach but instead develop a rigorous theory of "initial value problems at infinity" in section \ref{init_prob} which 
	expands the existing theory of Volterra integral equations. Explicitly we show the unique existence, as well as the continuous and differentiable 
	dependence on parameters of solutions to Volterra equations of the second kind on unbounded intervals for a general class of integral kernels. This 
	theory hadn't been developed before and might also be useful in different contexts. 
	
	\newpage
	
		With the tools of this theory at hand we then proceed to section \ref{const_on_mod}, where we study the solutions to the flatness equation and obtain our 
	first main result.
	
	\begin{theorem}
	At each weakly parabolic point of a Higgs bundle in $\mathcal{M}'$ there exists a solution of the flatness equation that is unique up to scalar multiplication and s. t. the Fock-Goncharov coordinates constructed out of these decorations are of the form
	\[
		\mathcal{X}=(1+r)\exp\left(\frac{R}{\zeta}\pi Z+i\theta+R\zeta\pi \overline{Z}\right),
	\]
	where $r$ is exponentially suppressed in $R$ and $Z$ and $\theta$ are coordinates of the Hitchin fibration.
	\end{theorem}
	
	We note here that the calculations necessary for this theorem show that the combination of our theory of initial value problems at infinity and the 
	theory of geodesics for quadratic differentials is a very good tool for the general problem of parallel transport.
	
	The $\mathcal{X}$-coordinates thus constructed are defined on $\mathcal{M}'$ for generic values of $\zeta\in\mathbb{C}^\ast$ and can be associated to elements of the lattice $\Gamma_q\subset H_1(\Sigma'_q;\mathbb{Z})$ on the spectral cover $\Sigma_q$ of $\mathcal{C}$ corresponding to the derminant $q$ of the Higgs field. For specific rays in $\mathbb{C}^\ast$ the underlying triangulation on $\mathcal{C}$ changes, which leads to jumps of the coordinates. For a certain generic subset $\widetilde{\mathcal{M}}'_{R,c}\subset\mathcal{M}'$ the jumps are completely understood and can be used to formulate a Riemann-Hilbert problem which is then solved by the $\mathcal{X}$-coordinates. We show how this is done in section \ref{sec_R_H}. Our main Theorem in that section then concerns a certain integral equation which was proposed by GMN for the general construction of Hyperkähler metrics on integrable system, where the jumps are determined by certain so called (generalized) Donaldson-Thomas invariants $\Omega$. 
	
	\begin{theorem}
		The $\mathcal{X}$-coordinates are solutions to the integral equation
		\[
		\mathcal{X}_\gamma(\zeta)=A_\gamma\mathcal{X}_\gamma^{\text{sf}}(\zeta)\text{exp}\left[-\frac{1}{4\pi i}\sum_{\beta\in\Gamma_q}\Omega(\beta;q)\left\langle\gamma,\beta\right\rangle\int_{l_{\beta}(q)}\frac{d\zeta'}{\zeta'}\frac{\zeta'+\zeta}{\zeta'-\zeta}\log\left(1-\mathcal{X}_\beta(x,\zeta')\right)\right],
		\]
		for some constant $A_\gamma\in\mathbb{C}$ with $\left\|A_\gamma\right\|=1$ which is exponentially close $1$ for large $R$.
	\end{theorem}
	
	In our setting the consideration of this integral equation has the advantage that we obtain additional analytic details regarding certain asymptotics of the $\mathcal{X}$-coordinates.
	
	Coming back to the general construction, the idea proposed by GMN is that the $\mathcal{X}$ coordinates fit into the picture of \cite{Gaiotto.2010}, 
	i. e. one can use them to construct a holomorphic symplectic form from which one can construct the hyperkähler metric on the moduli space by use of the twistor theorem of N. J. Hitchin, A. Karlhede, U. Lindström and M. Ro$\check{\text{c}}$ek. We will explain this general construction of GMN in the final section \ref{twist_metric_all}. Here we just note that it is the use of the twistor theorem for which the $\zeta$-asymptotics of the $\mathcal{X}$-coordinates and their derivatives have to be known. Only then one may obtain such a twistorial hyperkähler metric. As it turns out our theory is strong enough to also obtain these right asymptotics and thus obtain the final result.
	
	\begin{theorem}
		There is a hyperkähler metric on $\widetilde{\mathcal{M}}'_{R,c}$ constructed by the twistorial method whose difference to the semiflat metric decays exponentially in $R$. Furthermore the decay rate is given by the minimal period $2\pi R|Z_\mu|$, where $Z_\mu$ is the minimal period integral $\oint_\mu q^{1/2}$ for $\mu\in\Gamma_q$ with $\Omega(\mu)\neq0$.
	\end{theorem}
	
	We note here that we do not yet know that our twistorial hyperkähler metric is the natural $L^2$ metric on the moduli space. It is natural to conjecture this at this point and we will sketch an outline of a possible proof at the end of the last section, which was communicated to us by Andrew Neitzke.
	
	
	\vspace{1em}
	
	\textbf{Acknowledgments:} I thank my advisor, Harmut Weiß, for many helpful discussions and advices. I also want to thank Andrew Neitzke for clarifying important aspects of the general theory. I received reimbursement of travel expenses from the DFG SPP 2026.
	
\section{Parabolic Higgs bundles}\label{parabolic_higgs}
	
	In this section we start by developing the basic constructions for the moduli space we are interested in. The presentation here 
	mostly follows \cite{Fredrickson.2020}. For the general construction and further details see also \cite{Nakajima.1996} and \cite{BODEN.2012}. 
	We will start by giving an overview of parabolic Higgs bundles in subsection \ref{parabolic_higgs}, which can roughly be described as Higgs bundles that 
	may have singularities at certain fixed points. There are some different notions that will have to be separated there concerning the structure of the 
	poles. One idea is to allow simple poles and residues of the Higgs field which are nilpotent. Those were introduced in \cite{Konno.1993} and are called 
	strongly parabolic with regular singularities. As it turns out our construction does at this moment not seem to work on the corresponding moduli 
	space for these bundles. Nevertheless a thorough investigation of the possibilities of adapting the method presented here to that case would be highly 
	interesting. Instead we will focus on weakly parabolic bundles which also have simple poles but allow for diagonalizable residues of the Higgs 
	field. Finally there also is 
	a moduli space for bundles with singularities of order $>1$. These are called wild or irregular, were presented in \cite{Biquard.2004} and may also 
	be strongly or weakly parabolic. For us they are of importance as an application of the theory presented here to a special case of wild bundles on the  
	sphere was carried out in \cite{Tulli.2019}, which shows that the theory can actually be adapted to different settings.

	\subsection{The moduli space of parabolic Higgs bundles.}

		Let $\mathcal{C}$ be a compact Riemann surface of genus $g$ with a metric $g_C$ and Kähler form $\omega_C$ and $E$ a $\mathcal{C}^{\infty}$-vector bundle over 
		$\mathcal{C}$ of rank $r$ and degree $d$. This is the data that stays fixed throughout all of the following constructions. We also fix a holomorphic 
		structure $\overline{\partial}_{\text{Det}E}$ on the complex line bundle $\text{Det}E=\Lambda^rE$ with associated Hermitian-Einstein metric 
		$h_{\text{Det}E}$. This metric induces the Chern connection $A$ on this holomorphic Hermitian line bundle and is characterized by the curvature
		\[
			F_A=-\sqrt{-1}\text{deg}E\frac{2\pi\omega_{\mathcal{C}}}{\text{vol}_{g_{\mathcal{C}}}(\mathcal{C})}\text{Id}_{\text{Det}E}
		\]
		We can then equip $E$ with a holomorphic structure $\overline{\partial}_{E}$ demanding that $\overline{\partial}_{E}$ induces 
		$\overline{\partial}_{\text{Det}E}$ on $\text{Det}E$ and denote by $\mathcal{E}:=(E,\overline{\partial}_E)$ the emerging 
		holomorphic vector bundle. 
		
		We now introduce the parabolic data which starts with a divisor $D=p_1+\ldots+p_n$ on $\mathcal{C}$, i. e. we fix $n$ distinct points 
		on the surface. To these points we associate data that fixes the behavior of certain objects near them.
		
		\begin{definition}
			A \textit{parabolic structure} on $E$ consists of a choice of weighted flags $\mathcal{F}(p)=F_{\bullet}(p)$, i. e.
			\[
				\begin{split}
					E_p&=F_1(p)\supset\ldots\supset F_{s_p}(p)\supset 0\\
					0&\leq\alpha_1(p)<\ldots<\alpha_{s_p}(p)<1
				\end{split}
			\]
			for each $p\in D$ with $\sum_{i=1}^n\alpha_i=1$. A vector bundle with parabolic structure is called a \textsl{parabolic bundle}.
		\end{definition}
			
		Now let $K_{\mathcal{C}}$ be the canonical line bundle. We are interested Higgs fields that have simple poles at the marked points, which are defined 
		as follows:
			
		\begin{definition}
			A \textit{parabolic Higgs field} is a holomorphic bundle map $\Phi:\mathcal{E}\rightarrow\mathcal{E}\otimes K_{\mathcal{C}}(D)$ w. r. t. the 
			induced holomorphic structure of $\overline{\partial}_E$. The Higgs field is said 
			to \textit{preserve} the parabolic structure if, for all $i$ and $p$ the residue preserves the flags, i. e.:
			\[
				\text{res}_p\Phi(F_i(p))\subseteq F_i(p).
			\]
			A \textit{parabolic Higgs bundle} is a holomorphic vector bundle $\mathcal{E}$ together with a parabolic structure and a parabolic Higgs field $\Phi$ 
			preserving the 
			parabolic structure.
		\end{definition}
		
		Note that the twist by $D$ in the definition amounts to $\Phi$ having simple poles, so we will also call the Higgs field and associated objects 
		meromorphic.
		At this point there are two possibilities for the residue of $\Phi$. It may be nilpotent in which case the Higgs field is called 
		\textit{strongly parabolic} or it may be diagonalizable in which case it is called \textit{weakly parabolic}. In the case of weakly parabolic bundles 
		the residue data, i. e. the eigenvalues of the residue at each parabolic point shall also be fixed.
		
		In the following we will only consider weakly parabolic bundles of rank $2$. Furthermore we restrict ourselves to the $SL(2,\mathbb{C})$ case. For the definition 
		we state all the fixed data.
		
		\begin{definition}
			Let $E$ be a complex rank $2$ vector bundle over a compact Riemann surface of degree $d$, together with a divisor $D=p_1+\ldots+p_n$ and holomorphic 
			structure $\overline{\partial}_{\text{Det}E}$ on $\text{Det}E$. Furthermore for each $p\in D$ let a weight vector 
			$\vec{\alpha}(p)=(\alpha_1(p),\alpha_2(p))\in[0,1)^2$ with $\alpha_1(p)+\alpha_2(p)=1$ be given, as well as $\sigma_p\in\mathbb{C}^\times$.
			
			A weakly parabolic $SL(2,\mathbb{C})$-Higgs bundle over $(\mathcal{C},D)$ is a triple 
			$(\overline{\partial}_E,\left\{\mathcal{F}(p)\right\}_{p\in D},\Phi)$, where $\overline{\partial}_E$ is a holomorphic structure on $E$ inducing 
			$\overline{\partial}_{\text{Det}E}$, 
			$\mathcal{F}(p)$ is a complete flag for each $p\in D$ and $\Phi:\mathcal{E}\rightarrow\mathcal{E}\otimes K_{\mathcal{C}}(D)$ is a Higgs field 
			which is traceless with residue eigenvalues $\sigma_p$ and $-\sigma_p$ at each $p\in D$.
		\end{definition}
		
		It is useful to write $\mathcal{E}$ for the holomorphic bundle $(E,\overline{\partial}_E)$ 
		together with the flag, s. t. a parabolic Higgs bundle may be expressed as a pair $(\mathcal{E},\Phi)$ just as in the regular case.
		
		To obtain a well behaving moduli space a notion of stability is necessary. For this one defines the (weight depended) parabolic degree of $\mathcal{E}$ 
		as
		
		\[
			\text{pdeg}_{\vec{\alpha}}(\mathcal{E})=\text{deg}(\mathcal{E})+\sum_{p\in D}(\alpha_1(p)+\alpha_2(p)).
		\]
		
		A parabolic structure on $\mathcal{E}$ induces a parabolic structure on its holomorphic subbundles and so one defines $(\mathcal{E},\Phi)$ to be 
		$\vec{\alpha}$-stable iff
		\[
			\mu(\mathcal{E}):=\frac{\text{pdeg}_{\vec{\alpha}}(\mathcal{E})}{\text{rank}(\mathcal{E})}>\text{pdeg}_{\vec{\alpha}}\mathcal{L}.
		\]
		for every holomorphic line subbundle $\mathcal{L}$ preserved by $\Phi$. Here $\mu(\mathcal{E})$ is the \textit{parabolic slope} of $\mathcal{E}$. The notion of semistability follows if instead of $>$ one allows $\geq$ and the weights are called \textit{generic} if any $\vec{\alpha}$-semistable bundle is $\vec{\alpha}$-stable.
		
		For the definition of the moduli space we start by fixing the residues $\sigma=\left\{\sigma(p)\right\}_{p\in D}$ and the flags $\mathcal{F}(p)$ for each $p\in D$. We denote by SL$(E)$ the bundle of automorphisms of $E$ that induce the identity on Det$(E)$ and by $\mathfrak{sl}=\text{End}_0E$ the bundle of tracefree endomorphisms. Now denote by $\mathcal{A}_0$ the affine space of all holomorphic structures on $E$ that induce the fixed holomorphic structure $\overline{\partial}_{\text{Det}E}$ on Det$E$ and by $\mathcal{P}_{0,\sigma}\subset\Omega^{1,0}(\mathcal{C},\text{End}_0E)$ the set of sections of $\text{End}_0E\otimes K_C(D)$ (i. e. tracefree endomorphisms with simple poles), which are compatible with the flag and with residues $\sigma(p)$ and $-\sigma(p)$ at each $p\in D$.
		
		We then start with the space
		\[
			\mathcal{H}:=\left\{(\overline{\partial}_E,\Phi)\in\mathcal{A}_0\times \mathcal{P}_{0,\sigma}: \Phi\text{ is meromorphic w. r. t. to }\overline{\partial}_E\right\}.
		\]
		
		Now we fix a generic weight vector $\vec{\alpha}(p)$ at each $p\in D$ and denote by $\mathcal{H}_{\vec{\alpha}}$ the subspace of pairs $(\overline{\partial}_E,\Phi)\in\mathcal{H}$ which are $\vec{\alpha}$-stable.
		For the gauge group let ParEnd$(E)$ be the bundle of automorphisms of $E$ that preserve the flag at each $p\in D$. Our gauge group will then be
		\[
			\mathcal{G}_{\mathbb{C}}:=\Gamma(\text{SL}(E)\cap\text{ParEnd}(E)).
		\]
		
		Finally we obtain the moduli space of SL$(2,\mathbb{C})$-Higgs bundles $\mathcal{M}_{\text{Higgs}}:=\mathcal{H}_{\vec{\alpha}}/\mathcal{G}_{\mathbb{C}}$, which is a Hyperkähler manifold of dimension 
	\[
		\text{dim}_{\mathbb{C}}\mathcal{M}_{\text{Higgs}}=6(g-1)+2n
	\]
	for $g$ the genus of $\mathcal{C}$ (cf. \cite{BODEN.2012}). For further details see also \cite{YOKOGAWA.1993} and \cite{Konno.1993}.


	\subsection{Nonabelian Hodge correspondence}
	
	Although we usually talk about the Higgs bundle moduli space, the objects we are actually working with are most of the time not the Higgs bundles $(\mathcal{E},\Phi)$ but their 
	associated Higgs pairs $(\Phi,D(\overline{\partial}_E))$, which we'll introduce now.
	
	Let $(\mathcal{E},\Phi)$ be a parabolic Higgs bundle. By adding a Hermitian metric $h$ on $E$ the holomorphic structure $\overline{\partial}_E$ 
	canonically induces a connection $D(\overline{\partial}_E,h)$ on $E$, the so called Chern connection. After choosing a base connection 
	$D(\overline{\partial}_E,h)$ can be identified with an endomorphism 
	valued $1$-form $A$ and may also be written as $d_A$. Explicitly if $\overline{\partial}_E$ is locally 
	$\overline{\partial}+A_{\overline{z}}d\overline{z}$ in an $h$-unitary frame then 
	$d_A$ is simply given by $D(\overline{\partial}_E,h)=d_A=d+A=d+A_{\overline{z}}d\overline{z}-(A_{\overline{z}})^{\ast_h}dz$ where $d$ is the standard differential. Depending on the situation the connection is referred to as $D(\overline{\partial}_E,h)$, $d_A$ or simply $A$. We 
	can thus search for a hermitian metric $h$, such that the pair $(\Phi,D(\overline{\partial}_E,h))$ satisfies 
	Hitchin's equation	
	\begin{equation}\label{Hitchin_equ}
		F^{\bot}_{D(\overline{\partial}_E,h)}+\left[\Phi,\Phi^{\ast_h}\right]=0,
	\end{equation}
	where
	\[
		F^{\bot}_{D(\overline{\partial}_E,h)}-F_{D(\overline{\partial}_E,h)}=\sqrt{-1}\mu(E)\frac{2\pi\varpi_{\mathcal{C}}}{\text{vol}_{g_{\mathcal{C}}}(\mathcal{C})}\text{Id}_{\text{Det}E}.
	\]
	Note that the right hand side matches the curvature of the fixed Chern connection on $\text{Det}E$. In this way $F^{\bot}_{D(\overline{\partial}_E,h)}$ is the trace-free part of the curvature of $D(\overline{\partial}_E,h)$.\footnote{As \cite{Fredrickson.2020} is our primary source we point out an important difference in notation. Here we use the Symbol $\Phi$ for a parabolic Higgs bundle and when considering Hitchin's equation as an equation for the harmonic metric $h$. This is the holomorphic point of view and FMSW use the symbol $\varphi$ to denote the Higgs field in this picture. They, on the other hand, use the symbol $\Phi$ when considering a fixed background metric $h_0$ considering Hitchin's equation as an equation for Higgs pairs consisting of a Higgs field and an $h_0$-unitary connection. This is the unitary point of view. We will almost always work in the holomorphic picture and reserve $\varphi$ for the matrix part of the Higgs field $\Phi$.}
	
	As stated in \cite{Fredrickson.2020} in the parabolic setting one expects the Hermitian metric to match the parabolic structure in a suitable sense. For this one has to consider filtrations of the sections of 
	$\mathcal{E}\otimes\mathcal{O}_{\mathcal{C}}(\ast D)$, for $\mathcal{O}_{\mathcal{C}}(\ast D)$ the sheaf of algebras of rational functions with poles at 
	$D$. A parabolic structure determines such a filtration, as does a Hermitian metric, and if these two filtrations coincide the Hermitian metric is said to be \textit{adapted} to the parabolic structure. This can be expressed locally by demanding that near each $p\in D$ for a holomorphic basis of sections $e_1,e_2$ with $e_i\in F_i(p)$ one has $h(s,s)\sim |z|^{2\alpha_i(p)}$. For further details on this see \cite{Mochizuki.2006}.
	
	With these constraints one can now uniquely solve Hitchin's equation \eqref{Hitchin_equ} in the class of Hermitian metrics $h$, adapted to the parabolic 
	structure on $\mathcal{E}$ (cf. \cite{Simpson.1990}). Then one obtains one part of the nonabelian Hodge correspondence
	\[
		\mathcal{M}_{\text{Higgs}}\to\mathcal{M}, \enspace (\mathcal{E},\Phi)\mapsto(\mathcal{E},\Phi,h).
	\]
	
	Here $\mathcal{M}$ is the space of solutions to Hitchin's equations modulo gauge equivalence, where we consider the gauge group 
	\[
		\mathcal{G}=\Gamma(\text{SU}(E)\cap\text{ParEnd(E))}.
	\]
	of $h_0$-unitary transformations.
	We note the action of the gauge group as
	\[
		g\cdot\overline{\partial}_E=g\circ\overline{\partial}_E\circ g^{-1},
		\enspace g\cdot \Phi=g\Phi g^{-1}\enspace\text{and}\enspace
		(g\cdot h)(v,w)=h(gv,gw).
	\]
	The action on $\overline{\partial}_E$ then leads to an induced action on the Chern connection $A$.
	
	On the other side if we assume $\text{pdeg}_{\vec{\alpha}}(\mathcal{E})=0$ for simplicity then for every solution $h$ of Hitchin's equation there is a flat 
	$SL(2,\mathbb{C})$-connection 
	$A(\overline{\partial}_E,h)+\Phi+\Phi^{\ast}$. For our construction of the metric we will use a correspondence quite similar to this, with the difference 
	being a scaling in Hitchin's equation and a twist by the complex parameter $\zeta$.

\subsection{Spectral data and integrable systems}

	\begin{figure}[!htbp] 
		\centering
		\scalebox{0.4}{\includegraphics{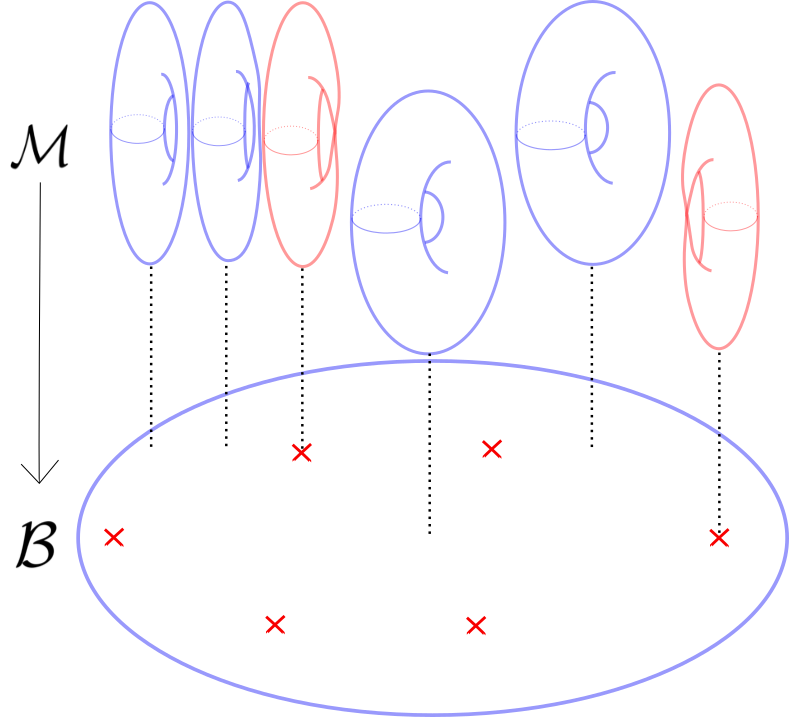}}
		\caption[A subset of the Hitchin fibration in dimension $4$.]{A subset of the Hitchin fibration in dimension $4$. The red crosses in the Hitchin base signify the elements of the discriminant locus over which the torus fibration becomes nodal.}
		\label{Hitchin_fibration}
		\end{figure}

	Before we can start with the general idea of the construction of the Darboux coordinates on $\mathcal{M}_R$ we have to introduce a final feature, again following \cite{Fredrickson.2020} where additional references are given.	
	To every Higgs bundle $(\mathcal{E},\Phi)$ one can associate the characteristic polynomial $\text{char}_{\Phi}(\lambda)$ which does not depend on the 
	parabolic structure. This gives rise to the Hitchin map
	\[
		\text{Hit}:\mathcal{M}\to\mathcal{B},\quad (\mathcal{E},\Phi)\mapsto\text{char}_{\Phi}(\lambda).
	\]
	Here $\mathcal{B}$ denotes the Hitchin base, identified with the affine space of coefficients of $\text{char}_{\Phi}(\lambda)$, so elements in $\mathcal{B}$ are collections of differential forms. In the case of 
	SL$(2,\mathbb{C})$ bundles it holds $\text{tr}\Phi=0$ and thus the characteristic polynomial is just the determinant of $\Phi$. 
	
	For weakly parabolic bundles Det$\Phi$ is a meromorphic quadratic differential on $\mathcal{C}$ with a double pole of the the form $-\sigma^2\tfrac{dz^2}{z^2}$ 
	near any weakly parabolic point $p$. One can associate a \textit{spectral} curve $\Sigma\subset\text{Tot}\left(K_{\mathcal{C}}(D)\right)$ to
	each point in $\mathcal{B}$ which is a $2:1$ ramified cover of $\mathcal{C}$ in $K_{\mathcal{C}}(D)$, where each sheet represents a different eigenvalue 
	of the Higgs field.
	
	The subset $\mathcal{B}'\subset\mathcal{B}$ of differentials in $\mathcal{B}$ that are generic, i. e. have only 
	simple zeroes, are the ones for which 
	$\Sigma$ is smooth. The following constructions will focus on the subset $\mathcal{M}':=\text{Hit}^{-1}(\mathcal{B}')$ (and the corresponding $\mathcal{M}'_R$) which is called the \textit{regular locus}. Its complement is called the \textit{discriminant locus}, cf. figure \ref{Hitchin_fibration}.
	
	Lastly we want to mention the result going back to \cite{Hitchin.1987b} for the regular case that the Hitchin map equips $\mathcal{M}$ with the 
	structure of an algebraically completely integrable system. Notably every fiber $\text{Hit}^{-1}(q)$ for each $q\in\mathcal{B}'$ forms a torus which implies the notion of action 
	angle-coordinates $Z,\theta$. Recovering these coordinates will be one of the main parts in the twistorial construction of the hyperkähler metrics.

\section{Decorated triangulations and their coordinates}\label{dec_and_coord}

	In this section we want to show how to build a (version of a) system of coordinates on the moduli space of flat connections on a Riemann surface 
	$\mathcal{C}$ that was introduced by Fock and Goncharov in \cite{Fock.2006}. These coordinates are constructed via the data of triangulations on 
	$\mathcal{C}$ together with a certain choice of decoration at each vertex. One crucial idea of GMN was, that the corresponding triangulations can be 
	obtained from quadratic differentials which appear as the determinants of the Higgs fields and that the decorations correspond to solutions of a flatness 
	equation for a certain connection build out of a solution of Hitchin's equation. We stress that the entirety of this section is a summary of the construction of \cite{Gaiotto.2013}, and no new results are presented in the first three sections.
	
	The theory presented in this section can mostly be considered independently of the theory of Higgs bundles and is interesting in its own right. So 
	ignoring their possible origin as determinants we start in subsection \ref{ca_triangulations} by introducing trajectories for arbitrary quadratic 
	differentials and how a triangulation of $\mathcal{C}$ is build out of them. Most of this theory was developed by K. Strebel in \cite{Strebel.1984}. We 
	follow \cite{Gaiotto.2013} in showing how to construct a triangulation 
	with these trajectories and which phenomena may occur that are important for us. It should also be mentionend that these constructions and phenomena may 
	be seen as an introduction to the theory of so called \textit{spectral networks} that were introduced by GMN in \cite{Gaiotto.2013b}. Since then they 
	have become widely used and generalized especially in the physics literature. The theory shown here is just their simplest example which already entails 
	quite a lot of the important structure of the general theory.
	
	In subsection \ref{sec_dec} we continue to present the work of \cite{Gaiotto.2013}, now turning to the construction of the decorations of the 
	triangulation and how to build the Fock-Goncharov coordinate system out of them. Finally we will introduce the spectral curve
	in subsection \ref{sec_hom} to obtain the full picture of the coordinate system. This is necessary to single out a certain subset of coordinates by use 
	of the homology group of the spectral curve in order to obtain a set of coordinates that can later be identified as a sort of Darboux coordinate system.
	
	Finally we have to consider what happens when we vary $\vartheta$. At certain values of $\vartheta$ the triangulation jumps which leads to corresponding jumps in the $\mathcal{X}$-coordinates. For the construction to work these jumps have to be of a certain kind and accordingly GMN carried out the analysis for the most important jumps in \cite{Gaiotto.2013}. Still a certain phenomenon may appear if $\mathcal{C}$ has genus greater then $0$ which may pose a problem. We will introduce the important notions in section \ref{spec_diff} and, using the work of T. Bridgeland and I. Smith \cite{Bridgeland.2015}, fill in some small gaps in the general theory.
	
\subsection{\texorpdfstring{$\vartheta$}{Theta}-trajectories and triangulations}\label{ca_triangulations}

	We begin this section by recalling the construction of the $\vartheta$-triangulations which were introduced in \cite{Gaiotto.2013} where they were called WKB 
	triangulations.\footnote{Here and in the following we will use the notions $\vartheta$-trajectory, triangulation, etc. instead of WKB curve, triangulation, 
	etc. which were used in \cite{Gaiotto.2013}. This is in part due to the fact that it is somewhat more in line with the mathematical literature concerning 
	quadratic 
	differentials, but also as of the fact that "WKB" is a reference to a certain approximation method used by \cite{Gaiotto.2013}. It is a rather important 
	part of the 
	results presented here that we refrain from using the WKB method. So while it may be inconvenient to differ from the language introduced in the 
	primary source for this construction it would be mathematically misleading not to do so.} In the following $\mathcal{C}$ will always denote a 
	compact Riemann surface. The foundational work for this section is \cite{Strebel.1984} where all relevant facts are already given. We present the 
	relevant results in this section and only mention those details which help getting an intuitive understanding of the phenomena that may occur. So nothing 
	new is presented here and the representation mostly follows \cite{Gaiotto.2013}.
	
	\begin{definition}
		Given a meromorphic quadratic differential $q$ on (a subset of) $\mathcal{C}$ and $\vartheta\in\mathbb{R}/2\pi\mathbb{Z}$ a 
		$\vartheta$-\textit{trajectory} of $q$ is a curve $\gamma:I\to\mathcal{C}$ for some interval $I\subset\mathbb{R}$, s. t. 
		\[
			q\left(\gamma',\gamma'\right)\in e^{2i\vartheta}\mathbb{R}_{>0}.
		\]
		Alternatively if $q$ has a well defined square root $q^{1/2}$ a $\vartheta$-\textit{trajectory} of $q$ may be defined as any curve 
		$\gamma:I\to\mathcal{C}$ for some interval $I\subset\mathbb{R}$, s. t. 
		\[
			q^{1/2}\left(\gamma'\right)\in\mathbb{R}^\times e^{i\vartheta}.
		\]
	\end{definition}
	
	Note that the second definition is more in line with calling the curve $\vartheta$-trajectory and is well defined, although the square root of $q$ is 
	only defined up to sign. This amounts to the fact, that the $\vartheta$-trajectories are not oriented, i. e. if $\gamma(t)$ is a $\vartheta$-trajectory 
	for $q$, so is $\gamma(-t)$ going in the opposite direction.
	
	As we'll explain below $\vartheta$-trajectories can be considered as straight lines with the inclination given by $\vartheta$ if considered in a suitable 
	neighborhood.  When talking about $\vartheta$-trajectories for varying $\vartheta$ it is also useful to have a notion that is independent of the specific 
	angle and just reflects that each one is of such a constant angle.
	
	\begin{definition}
		Let $q$ be a meromorphic quadratic differential. A curve $\gamma:I\to\mathcal{C}$ shall be called a $ca$-\textit{trajectory} of $q$ if there exists some 
		$\vartheta\in\mathbb{R}/2\pi\mathbb{Z}$ s. t. $\gamma$ is a $\vartheta$-trajectory.
	\end{definition}
	
	One important aspect of $ca$-trajectories is their behavior near zeroes and poles of the quadratic differential, which are called critical points in 
	\cite{Franco.2008}. In the work of GMN the zeroes of $q$ are also called turning points and we'll also use this notion at some 
	places. In the following we are interested in a certain behavior that is exhibited by the $ca$-trajectories only if they have a specific kind 
	of critical points, so we restrict our self to these differentials that are exactly the ones in the regular locus of the Hitchin base. Note however that 
	a lot is also known about the $ca$-trajectories for differentials in the discriminant locus, so even though the following structures breaks down at these 
	points, one may be able to get at least some information as to how the structure breaks down by further studies.
	
	\begin{definition}
		A \textit{generic differential} is a meromorphic quadratic differential $q$ on $\mathcal{C}$, which has only poles of order $2$ and simple 
		zeros.
	\end{definition}
	
	When talking about certain standard behaviors (e. g. as growing or decaying) along a $ca$-trajectory it is useful to fix a standard for the 
	parametrization, which we can always obtain.
	
	\begin{lemma}
		If $q$ is a generic differential there always exists a parametrization of its $\vartheta$-trajectories in which
		\[
			q^{1/2}\left(\gamma'\right)\in\pm e^{i\vartheta}.
		\]
		This shall be called the \textit{standard parametrization} of $\gamma$.
	\end{lemma}
	
	We now record the local structure of $ca$-trajectories near the different points we'll encounter.
	
	\begin{lemma}\label{log_spiral}
		Let $q$ be a generic differential, $\gamma$ a $\vartheta$-trajectory and $p$ one of the (second order) poles of $q$. Then $\gamma$ has one of the 
		following forms in a local neighborhood around $p$: 
		\begin{itemize}
			\item $\gamma$ is a logarithmic spiral running into (or away from) $p$.
			\item $\gamma$ is a circle with center $p$.
			\item $\gamma$ is a straight line running into $p$.
		\end{itemize}
	\end{lemma}
	
	\begin{proof}
		
		Since $q$ is a meromorphic quadratic differential it is known from \cite{Strebel.1984} that there exists a coordinate neighborhood 
		$z$ with $z(p)=0$ and $q=(M/z^2)dz^2$ for some complex number $M$. By shrinking our neighborhood we may thus assume, that $q^{1/2}$ is of the 
		form $\tfrac{m}{z}dz$ for some constant $m$. Then we obtain the following general form of the WKB curve in the local coordinate:
		\[
			z(t)=z_0e^{\pm\frac{e^{i\vartheta}}{m}t}.
		\]
		Here $z_0$ is some complex constant. Depending on the specific values of $m$ and $\vartheta$ one of the three asserted possibilities occurs.
	\end{proof}
	
	For more details on this see \cite{Gaiotto.2013} or, more generally, the original work of Strebel \cite{Strebel.1984}. Here we note that in such a 
	neighborhood for any given $\vartheta$ there are infinitely many $\vartheta$-trajectories running into the pole. This is contrasted by the 
	$\vartheta$-trajectories running into the other exceptional points.
	
	\begin{lemma}
		Let $q$ be a generic differential and $p$ one of its (simple) zeros, i. e. a turning point. Then for any $\vartheta\in\mathbb{R}/2\pi\mathbb{Z}$ there 
		are exactly three 
		$\vartheta$-trajectories that run into the turning point.
	\end{lemma}
	
	\begin{proof}
		This can be shown by switching to a coordinate where $q=Mzdz^2$ and integrating the defining differential equation.
	\end{proof}
		
	Finally we are interested in the behavior away from any critical point.
	
	\begin{lemma}
		Let $q$ be a generic differential and $\vartheta\in\mathbb{R}/2\pi\mathbb{Z}$. Then in a local neighborhood around any regular point $p$ of $q$ the
		$\vartheta$-trajectories are a foliation by straight lines.
	\end{lemma}
	
	\begin{proof}
		In the coordinates $w(z):=\int_{z_0}^z{q^{1/2}}$ the $\vartheta$-trajectories are a foliation by straight lines.
	\end{proof}
		
	Thus the behavior near every point of $\mathcal{C}$ is known and we obtain the following global picture.
		
	\begin{corollary}
		Let $q$ be a generic differential and $\vartheta\in\mathbb{R}/2\pi\mathbb{Z}$. Then the $\vartheta$-trajectories are the leaves of a (singular) 
		foliation of $\mathcal{C}$
	\end{corollary}
	
	With this picture in mind we now classify the different kinds of $ca$-trajectories that may occur in the following list that exhaust all the possibilities.
	
	\begin{definition}
		Let $\gamma$ be a $\vartheta$-trajectory for generic differential $q$. It is called
		\begin{itemize}
			\item \textit{generic}, iff it converges in both directions to a pole of $q$ (which may be the same for both directions).
			\item \textit{separating}, iff it converges in one direction to a zero of $q$ and in the other direction to a pole of $q$.
			\item a \textit{saddle}, iff it converges in both directions to a zero of $q$ (which may be the same for both directions).
			\item \textit{periodic}, iff it is diffeomorphic to a circle and does not converge to any pole or zero of $q$.
			\item \textit{divergent}, in all other cases. In this case $\gamma$ does in at least one direction not converge to a pole or zero of $q$.
		\end{itemize}
	\end{definition}
	
	As the name suggests almost all $ca$-trajectories in our setting will be generic and indeed most of our calculations will be along such trajectories. 
	For this we'll later on also need the following observation which we may already infer from the local behavior near their "end points".
	
	\begin{lemma}\label{infinite_range}
		A generic $\vartheta$-trajectory in standard parametrization can be extended to all of $\mathbb{R}$.
	\end{lemma}
	
	\begin{proof}
		This follows from standard methods for ODEs.
	\end{proof}
	
	This infinite domain of the curve leads to the necessity of using integral equations with initial values at infinity later on as infinity corresponds to 
	the poles of the quadratic differential.
	
	Before we continue with the construction of the triangulation note the following remarks for some intuition concerning $ca$-trajectories:
	
	First remember that $q$ was only defined up to sign which is why the $ca$-trajectories are not canonically oriented. The description of the behavior 
	around poles as spiraling into the point may still be useful as a reminder what happens. Phenomenologically one might describe the singularities of 
	$q$ as some (massive) object which "captures" all curves, which come near enough. Thus finite curves have to be "rare". Indeed for one fixed $\vartheta$ 
	their number is limited by the number of turning points and, even more, a curve coming from one turning point is more "likely" to run into a singularity 
	then another turning point. But the parentheses here are necessary when we talk about $ca$-trajectories of general quadratic differentials. For example 
	one can show that on the four punctured sphere with a quadratic differential having only simple poles, for some angles only two finite $ca$-curves and one 
	divergent curve exist which fill up the whole sphere. The angles for which such a behavior occurs even form a dense subset of 
	$\mathbb{R}/2\pi\mathbb{Z}$. However, by working with generic quadratic differentials, the generic trajectories may safely be regarded as the ones 
	which "usually" occur, while finite curves don't exist for most angles.\footnote{In fact the appearance of finite curves is the defining property for the 
	definition of the so called DT invariants in the general Riemann Hilbert problem GMN discuss.} This is why the following fact is crucial which was shown 
	by GMN and which makes precise some of the vague formulations above.
		
	\begin{proposition}
		Let $q$ be a generic differential and $\vartheta\in\mathbb{R}/2\pi\mathbb{Z}$ such that there is no finite trajectory. Then there is also no 
		divergent WKB curve.
	\end{proposition}
		
	\begin{proof}
		The proof relies on the existence of a singularity, since in that case a $ca$-trajectory can't fill up the whole complex curve, as it would "fall into 
		the singularity". For more details see \cite{Gaiotto.2013}.
	\end{proof}
	
	As the case of no finite trajectory is the standard one for the following triangulations, we also denote them as such.
	
	\begin{definition}
		Let $q$ be a generic differential. If $\vartheta\in\mathbb{R}/2\pi\mathbb{Z}$ is such that there is no finite trajectory, then $\vartheta$ is called a
		\textit{generic} angle.
	\end{definition}
	
	It follows that in this generic case the foliation by $\vartheta$-trajectories consists of a finite collection of separating trajectories and an infinite 
	amount of generic ones and the following structure emerges:
	
	The separating trajectories form boundaries of "cells" on $\mathcal{C}$ which are swept out by generic trajectories. The boundaries consist of at most 
	four separating trajectories (which may be regarded as the generic case), two of which emanate from a common turning point. This is called a "diamond" in 
	\cite{Gaiotto.2013}. There may however be cells whose boundary consists only of three separating trajectories which all emanate from the same turning 
	point. This may be regarded as a degenerate case, called a "disc" in \cite{Gaiotto.2013}, which will require a bit more attention 
	later on, as the constructions need to be aware of this special case. For understanding (and visualizing) the theory however, it is enough for now to 
	think of the foliation as forming those swept out diamonds on all of $\mathcal{C}$.
	
	With this structure in mind it is now possible to define the $\vartheta$-triangulation.
	
	\begin{definition}
		Let $q$ be a generic differential and $\vartheta\in\mathbb{R}/2\pi\mathbb{Z}$ a generic angle. A 
		$\vartheta$-\textit{triangulation} is a collection of generic $\vartheta$-trajectories, such that every two $\vartheta$-trajectories belong to 
		two different cells, given by the separating trajectories, and such that in each cell a generic trajectory is chosen.
	\end{definition}
	
	One can now look at the emerging structure and notice that, corresponding to the discs above, there may be degenerate triangles with only two distinct 
	edges. Counting these still as triangles one obtains the following result (cf. \cite{Gaiotto.2013}).
		
	\begin{proposition}
		A $\vartheta$-triangulation is a triangulation of $\mathcal{C}$, in which each face contains a single turning point.
	\end{proposition}
		
	Such a triangulation is of course not unique as there is family of generic WKB curves in each diamond or disc. When talking about these triangulation 
	one uses the fact that two different ones differ only by an isotopy, which won't be relevant for the following procedure. As GMN point out, one really 
	chooses an isotopy class of triangulations.
	
	\begin{remark}\label{no_deg_tri}
		Typically our constructions in the following sections will only consider the case of generic triangles. However, the notion of degenerate triangles has 
		to be included for the full picture as there is no way to make sure that these triangles do not occur in our cases. GMN have shown in 
		\cite{Gaiotto.2013} how all of the constructions that we consider here and depend on the triangulation can be adapted to the case of degenerate 
		triangles. We will only briefly consider this case later on as no real problems arise in their treatment.
	\end{remark}
	
\subsection{Decorations and Fock-Goncharov coordinates}\label{sec_dec}
		
	Thus far we have considered a compact Riemann surface $\mathcal{C}$ and constructed a triangulation for every generic differential and generic angle 
	on $\mathcal{C}$. This is one half of the data one needs to construct the Fock-Goncharov coordinate system, the other being a decoration at each vertex, for 
	which we now have to consider a rank $2$ complex vector bundle $V$ over $\mathcal{C}$. Though the moduli space we are interested in is the space of 
	solutions of Hitchin's equations the coordinates are constructed on the moduli space of flat $SL(2,\mathbb{C})$-connections $\mathcal{M}_{\text{flat}}$. 
	Later on we will identify the moduli spaces and thus pull back the coordinates but for explaining how the construction works we stay on 
	$\mathcal{M}_{\text{flat}}$. The construction of \cite{Gaiotto.2013} now proceeds as follows.
	
	\begin{definition}
		Let $V\to\mathcal{C}$ be a complex rank $2$ vector bundle and $\mathcal{A}$ a flat $SL(2,\mathbb{C})$-connection on $E$ with a regular singularity at a 
		point $p\in\mathcal{C}$ with monodromy operator $M$. A \textit{decoration} at $p$ is a choice of one of the two eigenspaces of $M$. The corresponding 
		eigenvalue is denoted by $\mu^T$.
		
		A \textit{decorated triangulation} is a triangulation $T$ together with a decoration for each vertex of $T$. We denote the corresponding decorated 
		triangulation by $T(q,\vartheta)$ where when no ambiguity may arise we may suppress the dependence of $T$ on $q$ and $\vartheta$. 
	\end{definition}
	
	Note that this differs slightly from the original construction by Fock and Goncharov as they considered flags instead of decorations. Furthermore, as GMN 
	point out, in the case at hand the conjugacy classes of the monodromy operators $M_i$ are fixed for the moduli space, so they are part of the discrete datum of the 
	triangulation, while Fock and Goncharov considered varying $M$ for the connections.
	
	We now have the data we need to build the coordinate system. For a given edge $E$ consider the two triangles which have this edge in common. They make up a quadrilateral $Q_E$ and their vertices $p_i$ shall be 
	numbered in counter-clockwise order (for $\mathcal{C}$ canonically oriented). For each $p_i$ there is a decoration which corresponds to a solution $v_i$ 
	of the flatness equation
	\begin{equation}\label{flatness_eq}
		\mathcal{A}v_i=0.
	\end{equation}
	
	The $v_i$ may be regarded as eigenvectors of the monodromy $M_i$ with eigenvalue $\mu_i^T$. They are only defined up to a scalar multiple, but this 
	ambiguity will cancel in the following step. Note that as $\mathcal{A}$ is flat the sections exist on any simply connected subset of $\mathcal{C}$.
	
	For two of the sections $v_i$ and $v_j$ the $\wedge$-product $v_i\wedge v_j$ is an element of the the 
	determinant line bundle 
	Det$E$. If $\mathcal{A}_{\text{Det}}$ denotes the induced connection on Det$E$, it follows that the $\wedge$-product solves the induced flatness equation
	\[
		\mathcal{A}_{\text{Det}}(v_i\wedge v_j)=0.
	\]
	When considering a frame $(e_1,e_2)$ on $Q_E$, so $v_i=s^i_1e_e+s^i_2e_2$ and $v_j=s^j_1e_1+s^j_2e_2$ the $\wedge$-product is locally given as
	\[
		v_i\wedge v_j(z)=(s^i_1s^j_2-s^i_2s^j_1)e_1\wedge e_2
	\]
	We may thus identify $v_i\wedge v_j$ with the Wronskian of the (local expression of the) two solutions of the flatness equation \eqref{flatness_eq} 
	evaluated at some point $z$ in $Q_E$. As all of our calculations later on are in a specific frame we will always use this interpretation, which we'll 
	also write as
	\[
		s^i\wedge s^j=\text{det}\left[s^i(z),s^j(z)\right].
	\]
	If $\mathcal{A}$ is represented by the matrix $A$ in the frame $F$, i. e. $\mathcal{A}=d+A$ for the standard differential $d$, the induced equation on Det
	$E$ is given by the trace of $A$: 
	\begin{equation}\label{induced_local}
		d\left(s^i\wedge s^j\right)=\text{tr}(A)(s^i\wedge s^j).
	\end{equation}
	Using this identification it is possible to consider the quotient of two $\wedge$-products of sections for $v_i$ with local expressions $s_i$ as  
	\[
		\frac{v_1\wedge v_2}{v_3\wedge v_4}(z):=\frac{s^1\wedge s^2}{s^3\wedge s^4}(z)
	\]
	evaluated at some point $z\in Q_E$. Using the fact that both $\wedge$-products are solutions to the same $1$-dimensional linear ODE \eqref{induced_local} 
	one obtains that $\frac{v_1\wedge v_2}{v_3\wedge v_4}$ is actually independent of the evaluation point $z$. Furthermore as of 
	$Gv_i\wedge Gv_j=\det(G)v_i\wedge v_j$  for any $G\in GL(2,\mathbb{C})$ the quotient is independent of transformations $G\in GL(2,\mathbb{C})$. We are now able to define the following functions for each edge.
	
	\begin{definition}\label{Darboux_coords}
		Let $V\to\mathcal{C}$ be a complex rank $2$ vector bundle and $\mathcal{A}$ a flat $SL(2,\mathbb{C})$-connection on $E$ with regular singularities and 
		$T(q,\vartheta)$ a decorated triangulation. Let $E$ be an edge of $T$ with corresponding quadrilateral $Q_E$ with vertices 
		$p_1,p_2,p_3,p_4\in\mathcal{C}$ and and corresponding flat sections $v_i$ with local expressions $s_i$ in some frame on $Q_E$. Define the \textit{canonical 
		coordinate} for $E$ as
		\[
			\mathcal{X}_E^{T(q,\vartheta)}:=-\frac{(v_1\wedge v_2)(v_3\wedge v_4)}{(v_2\wedge v_3)(v_4\wedge v_1)}
			:=-\frac{s^1\wedge s^2}{s^4\wedge s^1}\frac{s^3\wedge s^4}{s^2\wedge s^3}.
		\]
	\end{definition}
	
	As promised the ambiguity in the choice of the $v_i$ cancels in the final expression and just like the single quotients the full function is independent 
	of transformations $g\in GL(2,\mathbb{C})$ when taken for all sections simultaneously.
	
	We may now conclude this section with the main statement for this construction. The $\wedge$-product may become $0$ for two adjacent vertices, which 
	would lead to the corresponding coordinate function becoming $0$ or $\infty$, but, as GMN point out, this can only happen in a codimension $1$-subvariety 
	of $\mathcal{M}_\text{flat}$. One can see this in the local expression as the $\wedge$-product becomes $0$ iff the sections are linearly dependent. This 
	characterization will be our main tool in determining whether our choice of sections later on is appropriate. In this case we finally obtain the desired 
	coordinates, though we have to select a certain set out of all the defined functions.
		
	\begin{proposition}\label{FG_coord}
		The $\mathcal{X}_E^{T}$ contain a coordinate system on a Zariski open patch $\mathcal{U}_{T}\subset\mathcal{M}_{\text{flat}}$.
	\end{proposition}
	
	\begin{proof}
		The original proof is found in \cite{Fock.2006}. In Appendix A of \cite{Gaiotto.2013} it is shown how the data of the monodromy representation of 
		$\mathcal{A}$ is encoded in the $\mathcal{X}_E^T$.
	\end{proof}
	
	\begin{remark}
		Let us briefly remark on the word "contain" here. As said above we do need to select some coordinates as the constructed ones are to many, which can be 
		seen 
		as follows. As we have a obtained one $\mathcal{X}$ function for each edge of the triangulation we can count them using Euler's formula. We then see 
		that for a Riemann surface $\mathcal{C}$ of genus $g$ with $n$ singular points we obtain $6g-6+3n$ coordinate functions. The dimension of $\mathcal{M}$ 
		however is $\dim\mathcal{M}_{\text{flat}}=6g-6+2n$. The difference reflects the fact that we consider the moduli space with fixed monodromy eigenvalues 
		at the 
		singular points. As GMN show, $n$ combinations of the $\mathcal{X}_E$ functions give the monodromy eigenvalues $(\mu_i^T)^2$. As these are fixed the 
		information is superfluous, so that we can discard $n$ of the functions (one for each singular point) and obtain the correct number of coordinates. At 
		the moment however there is no way of distinguishing, which functions to keep and which to discard. This is done in the next section.
	\end{remark}
	
	\begin{remark}
		As is also remarked upon in Appendix A of \cite{Gaiotto.2013} the coordinates of Fock and Goncharov are at first defined on the space of PSL$(2,\mathbb{C})$ connections, so they have to be pulled back to the space of SL$(2,\mathbb{C})$ connections. In the case of $\mathcal{C}=\mathbb{CP}^1$ this can be done in a unique way, but for larger genus the moduli space of SL$(2,\mathbb{C})$ connections is a discrete cover of the moduli space of PSL$(2,\mathbb{C})$ connections. This will not be of much concern to us. On the one hand there is the possibility to work directly with PSL$(2,\mathbb{C})$ connections (cf. \cite{Teschner.2011} and references therein). But also when working with SL$(2,\mathbb{C})$ connections as is done here, we can simply use the constructed coordinates as a complete set of coordinates on each sheet of the covering which is enough for our construction.
	\end{remark}
	
\subsection{Homology and lattices}\label{sec_hom}
	
	As explained in the last section, it is necessary to select a certain subset of the constructed coordinate functions to obtain the correct dimension. 
	Additionally later on the hyperkähler metric will be obtained from a symplectic form that is constructed from these coordinates we just described. So our 
	choice of coordinates also has to ensure that the emerging form is actually symplectic. Such a selection is possible via the use of certain lattices, one 
	of which carries a symplectic pairing. We will now summaries the important facts from \cite{Gaiotto.2013} about these lattices, how the correspond to the edges and finally how 
	they select the right coordinates.
	
	All of the lattices we consider are given by the spectral data, i. e. we have to consider the spectral curve. As we are dealing with differentials with 
	singularities we have to be careful as to whether we include or exclude the poles and unfortunately the notation is not canonical throughout the 
	literature. For our use of symbols so far the following notation seems appropriate:
	
	\begin{definition}
		Let $q$ be a generic differential with $n$ singular points at $p_1,\ldots,p_n$ on the compact Riemann surface $\mathcal{C}$. Denote by 
		$\mathcal{C}':=\mathcal{C}\setminus\left\{p_1,\ldots,p_n\right\}$ the punctured surface. Then the \textit{spectral curve} is defined as
		\[
			\Sigma'_q:=\left\{(z\in\mathcal{C}',\lambda\in T^\ast_z\mathcal{C}):\enspace\lambda^2=q(z)\right\}\subset T^\ast\mathcal{C}.
		\]
	\end{definition}
	
	This is a noncompact complex curve which is smooth for generic $q$. It is a double covering $\pi':\sigma'_q\to\mathcal{C}'$ which is branched over the 
	zeroes of $q$. With this notation $\Sigma'_q$ is "punctured over the singular points $p_i$" and there is corresponding compactification $\Sigma_q$ of 
	$\Sigma'_q$ which is obtained by filling in these punctures. The compactification has a corresponding projection $\pi:\Sigma_q\to\mathcal{C}$. Depending 
	on the question at hand we might also talk about $\Sigma_q$ as the spectral curve or don't distinguish $\Sigma_q$ and $\Sigma'_q$ at all if not necessary.
	However it is necessary for the next step, where we consider their homology groups.
	
	We start by noting that $\Sigma'_q$ is equipped with an involution $i:(z,\lambda)\mapsto (z,-\lambda)$ that acts accordingly on cycles in the first 
	homology group.
	
	\begin{definition}
		Let $q$ be a generic differential with $n$ singular points at $p_1,\ldots,p_n$ on the compact Riemann surface $\mathcal{C}$. The 
		\textit{charge lattice}\footnote{The terminology in this part is mostly in accordance with the physical theory, from which the indices $e$ and $m$ below are also taken, as they refer to so called electric and magnetic charges.}
		$\Gamma_q$ is defined as the subgroup of $H_1(\Sigma'_q;\mathbb{Z})$ that is odd under the involution $i$.
	\end{definition}
	
	When varying $q$ one can think of $\Gamma_q$ as the fiber of a local system $\Gamma$ over the space of $q$. Note that $\Gamma_q$ is equipped with the 
	intersection pairing $\left\langle\cdot,\cdot\right\rangle$ which may be degenerate. 
	
	\begin{definition}
		Let $\Gamma_q$ be the charge lattice of a generic differential $q$ with intersection pairing $\left\langle\cdot,\cdot\right\rangle$. The 
		\textit{flavor lattice} 
		$\Gamma_q^f$ is defined as the radical of $\left\langle\cdot,\cdot\right\rangle$. The quotient $\Gamma_q^g:=\Gamma_q/\Gamma_q^f$ is called the \textit{
		gauge lattice}.
	\end{definition}
	
	Note that $\Gamma_f^q$ does not change with varying $q$, so it can be regarded as a fixed lattice $\Gamma^f$ instead of a local system of lattices. Most 
	important for us it that the gauge lattice $\Gamma_q^g$, which can be identified with $H_1(\Sigma;\mathbb{Z})$, also carries the intersection pairing 
	which is skew-symmetric and nondegenerate as of the quotient construction. This is important as it allows us to obtain a symplectic basis.
	
	\begin{proposition}
		Let $\Gamma_q^g$ be the gauge lattice of a generic differential $q$. Then there exists a symplectic basis 
		$\gamma^e_1,\ldots,\gamma^e_r,\gamma^m_1,\ldots,\gamma^m_r$ of $\Gamma_q^g$, i. e.
		\[
			\left\langle\gamma^e_i,\gamma^e_j\right\rangle=0,\enspace\left\langle\gamma^m_i,\gamma^m_j\right\rangle=0,\enspace
			\left\langle\gamma^e_i,\gamma^m_j\right\rangle=\delta_{ij}.
		\]
	\end{proposition}
	
	We want to use this structure, for which we need to identify charges $\gamma\in\Gamma_q$ with edges $E$ of the triangulation. We only sketch this process 
	here and refer to \cite{Gaiotto.2013} for the details. The idea is, that for ech edge the corresponding quadrilateral $Q_E$ contains two zeroes of $q$. 
	Thus we could try to lift a loop around these two points to a cycle in $\Sigma'_q$. The problem with this is that there are two ambiguities in the 
	orientation of the loop as well as the choice of sheet of the covering. These choices need to be canonically decided which is done by GMN by using a 
	canonical orientation for the lift of the edges and a pairing of $H_1(\Sigma'_q;\mathbb{Z})$ with the relative homology group 
	$H_1(\Sigma'_q,\left\{p_i\right\};\mathbb{Z})$. In this way we obtain a canonical cycle $\gamma_E$ for each edge $E$ of the triangulation $T$ and GMN 
	proceed by showing that these cycles suffice.
	
	\begin{proposition}
		Let $\Gamma_q$ be the charge lattice of a generic differential $q$ and $T(q,\vartheta)$ a triangulation for a generic angle $\vartheta$. Then the set 
		of cycles $(\gamma_{E_i})_i\subset\Gamma_q$ obtained from the edges in the way described above generate $\Gamma_q$.
	\end{proposition}
	
	With these identifications at hand we may now define the coordinate functions on the moduli space of flat $SL(2,\mathbb{C})$-connections labeled by 
	charges.
	
	\begin{definition}
		Let $q$ be a generic differential, $T(q,\vartheta)$ a decorated triangulation and $\Gamma_q$ the charge lattice for $q$ with the generating cycles 
		$(\gamma_{E_i})_i$ associated to the edges $(E_i)_i$ of the triangulation. For each such $E_i$ we define
		\[
			\mathcal{X}_{\gamma_{E_i}}^{(q,\vartheta)}:=\mathcal{X}_{E_i}^{T(q,\vartheta)}.
		\]
		For the sum of two charges $\gamma,\gamma'\in\Gamma'_q$ we define
		\[
			\mathcal{X}_{\gamma+\gamma'}^{(q,\vartheta)}:=\mathcal{X}_{\gamma}^{(q,\vartheta)}\mathcal{X}_{\gamma'}^{(q,\vartheta)}.
		\]
	\end{definition}
	
	Thus $\mathcal{X}_{\gamma}^{q,\vartheta}$ is well defined for all elements $\gamma\in\Gamma'_q$. In constructing the hyperkähler metric we won't actually 
	need the $\mathcal{X}_\gamma$ functions for charges that are not in the symplectic bases. We added the definition here as it amounts to a certain Poisson 
	structure which is relevant in the more abstract setting of integrable systems, for which our case may be thought of as a specific example. For an 
	overview of this setting we refer to \cite{Neitzke.2013}.
	
	Finally we turn to the matter of dimension. Remembering that we had $n$ more coordinate functions then need when labeled by the edges we accordingly have 
	the same redundancy when we label them by generators of $\Gamma'_q$. At this point the gauge lattice $\Gamma_q^g$ comes into play. Not only does this 
	lattice 
	carry the symplectic pairing we'll need later on. It is also of the correct rank: When we identify $\Gamma_q^g$ as the first homology group 
	$H_1(\Sigma_q;\mathbb{Z})$ of the unpunctured spectral curve $\Sigma$ we notice that by filling in the punctures of $\Sigma'_q$ we lose exactly those 
	generators in $H_1(\Sigma'_q;\mathbb{Z})$ that correspond to cycles around the preimage of the singular points $\pi^{-1}(p_i)$. Thus we lose exactly 
	those cycles which correspond to the fixed monodromy at the singular points and it remains the true coordinate system for $\mathcal{M}_{\text{flat}}$.
	
	At this point the general construction of the coordinate system is finished, although some ambiguity remains which will only be resolved when we apply it 
	to our case. While the quadratic differential $q$ has a natural interpretation in our setting as the determinant of the Higgs field, the angle $\vartheta$ 
	seems arbitrary. In our application this parameter will be related to another parameter $\zeta\in\mathbb{C}^\times$ that is build into the flat connection 
	$\mathcal{A}$. For the most part of the following constructions we will focus on a general relation between the two which allows for them both to vary in 
	relation to another. This is done as our results hold for this general setting and a very large part of the theory developed by GMN concerns the relation 
	of these two parameters. Still, when it comes to constructing the hyperkähler metric at the we will need to fix $\vartheta$ as it is not intrinsic to the 
	moduli space. This is done by setting $\vartheta=\text{arg}(\zeta)$ which will obey all relations we demand throughout the calculations.

At this point the general construction of the coordinate system is finished, although some ambiguity remains which will only be resolved when we apply it 
	to our case. While the quadratic differential $q$ has a natural interpretation in our setting as the determinant of the Higgs field, the angle $\vartheta$ 
	seems arbitrary. In our application this parameter will be related to another parameter $\zeta\in\mathbb{C}^\times$ that is built into the flat connection 
	$\mathcal{A}$ and which will encode the complex structure in an application of Hitchin's twistor theorem. For the most part of the following constructions we will focus on a general relation between the $\vartheta$ and $\zeta$ which allows for them both to vary in 
	relation to another. This is done as many of our results hold for this general setting and a very large part of the theory developed by GMN concerns the relation of these two parameters. Still, when it comes to constructing the hyperkähler metric we will need to fix $\vartheta$ as it is not intrinsic to the 
	moduli space. We will do this by setting $\vartheta=\text{arg}(\zeta)$ which will obey all relations we demand throughout the calculations. But this implies that we can't constrain ourselves to generic angles later on, because the twistor theorem demands to consider all elements $\zeta\in\mathbb{CP}^1$. So we now have to consider how the triangulation behaves when a finite trajectory appears.
	
	\subsection{The spectrum of a quadratic differential}\label{spec_diff}
	
	In this section we consider again the general setting of the triangulation, i. e. we consider some generic quadratic differential $q$, but now we vary the angles $\vartheta$ to study the changes of the associated $\vartheta$-triangulations. It turns out that the changes are mostly trivial, as for generic $\vartheta$, i. e. such that no finite trajectories occur, the edges of the triangulation only vary 
	by a continuous homotopy. Such variations don't change the isotopy class of the triangulation and the coordinates will therefore be 
	constant as functions of $\vartheta$ as long as no finite trajectories appear. If they do occur, however, there are in fact 
	discontinuous jumps, cf. the footnote in \cite{Gaiotto.2013}, p. 312.
	Understanding the behavior of the triangulations and coordinates as these jumps occur is one of the most important aspects of the work in 
	\cite{Gaiotto.2013}. From now on, an angle $\vartheta$ at which such a jump occurs shall be denoted by $\vartheta_c$ as for critical phase, cf. figure \ref{saddle_jump}.\footnote{The appearance of finite $ca$-trajectories is intricately linked to the notion of generalized DT invariants which are of utmost importance for the generalization of the construction presented here to general integrable systems. We will return to these invariants in section \ref{sec_R_H}.} GMN present three different ways in which the triangulation can change at a critical phase, which are denoted as flips, pops and juggles.
		
		\begin{figure}
		\centering
		\scalebox{0.5}{\includegraphics{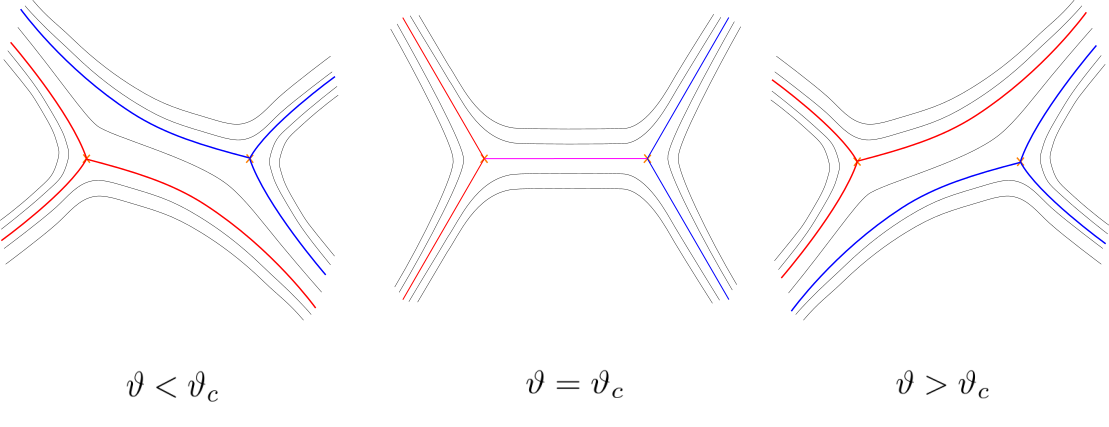}}
		\caption[A pop in the triangulation.]{A pop in the triangulation.}
		\label{saddle_jump}
		\end{figure}
		
		\newpage
		
	\begin{definition}
		Let $\vartheta_c$ be a critical phase for which a finite $\vartheta_c$-trajectory occurs.
		\begin{enumerate}
			\item If the curve connects to two different turning points, the corresponding change in the triangulation shall be called a \textit{pop}.
			\item If the curve is closed and surrounds one singular point, the corresponding change in the triangulation shall be called a \textit{flip}.
			\item If the curve is closed and does not contract to a single, point the corresponding change in the triangulation shall be called a \textit{juggle}.
		\end{enumerate}
	\end{definition}
		
	GMN show that the coordinates actually don't change when a pop occurs, so the only relevant jumps for that aspect are the flips and juggles.
	%
	The property we need now is that although the coordinates jump, a certain $2$-form $\varpi$ that is built out of the derivatives of the $\mathcal{X}$-coordinates is well defined.\footnote{We will introduce $\varpi$ properly in chapter \ref{twist_metric_all}.} This works when the jumps do not change $\varpi$, and GMN devote a great effort to proof that this is the case for the most important jumps in \cite{Gaiotto.2013}. However there may be jumps for which their argument does no work, so we have to talk about what is rigorously known now. GMN showed that the jumps which correspond to the appearance of a single saddle connection, i. e. pops, are of the right nature and one kind of jumps which corresponds to the appearance of ring domains, i. e. certain juggles, are also fine.
	
	Unfortunately, as GMN themselves note in section 7.6.2 \textit{Symplectomorphisms from a BPS vectormultiplet} of \cite{Gaiotto.2013}, a certain phenomenon of the $ca$-triangulation may occur on a Riemann surface $\mathcal{C}$ with genus $g>0$. On such a $\mathcal{C}$ a ring domain may cut off a torus without critical points, i. e. the complement of the ring domain is given by two disjoint subsurfaces, one of which has genus one and does not contain any zeros or poles of $q$. In that case their direct calculation of the jump of the coordinates does not work. Instead they refer to their work on the more general concept of spectral networks in \cite{Gaiotto.2013b} for this case. There they examine differentials which come from a degree $K$ polynomial instead of one with degree $2$ as for our generic differentials here. This leads to a lot of additional language which we can't properly introduce here, but most importantly they study the jumping behaviour of a certain \textit{formal generating function} $F$ of formal variables $X_\gamma$. The idea is that our $\mathcal{X}$-coordinates can be considered as a special case of these variables and their argument there in section 7.2 \textit{Closed loops} thus delivers the missing case. But there are some problems with this argument, one being the quite different approach which is not easily translated into the setting here. More importantly though, even if the argument for that case works there may still be a missing case, when one of the boundaries of the ring domains consists of more than one saddle connection.
	
	Now to make these problems precise it is useful to start with the work of T. Bridgeland and I. Smith. In section 5 of \cite{Bridgeland.2015} they first show what kind of ring domains can occur. It turns out that, as stated above, most of these were shown to work in \cite{Gaiotto.2013}. The only possibility that has not been rigorously proven is given by the occurrence of a ring domain that cuts off a torus without critical points. Thus we now want to concentrate on those differentials where such a phenomenon does not occur. To formulate this condition it is useful to first introduce the spectrum of a differential.
	
	\begin{definition}
		For $q$ a generic quadratic differential, the \textit{spectrum} of $q$ is the set of all ring domains and all saddle-trajectories for any angle $\vartheta\in[0,2\pi)$.
	\end{definition}
	
	The elements of the spectrum that may cause problems are ring domains that cut off a torus on which there are no critical points. This case may be pictured accordingly as a torus without any relevant structure and which has a very intricate ring structure attached to one side of it. We hope that the following term captures this picture.
	
	\begin{definition}
		Let $q$ be a generic differential. A ring domain that cuts of a torus without any singularities shall be called a \textit{signet ring}.
		If the spectrum of $q$ does not contain a signet ring, it shall be called \textit{unsealed}.
	\end{definition}
	
		Note that in particular on the sphere (with an arbitrary amount $\geq 4$ of punctures) every generic differential has an unsealed spectrum, so the theory considered here is not vacuous.
		
		For the upcoming discussion we identify the finite $ca$-trajectories with elements of the charge lattice. Following \cite{Neitzke2017} we note that the closure of the lift of a saddle connection for $q$ to the spectral cover $\Sigma'_q$ is an oriented loop, while the lift of a closed loop to the spectral cover is the union of two oriented loops. 
		
		\begin{definition}
		Let $q$ be a generic differential and $c$ a finite $ca$-trajectory for $q$. The class in $\Gamma_q$ of the corresponding loop (for a saddle connection) or the union of the two loops (for a ring domain) on $\Sigma'_q$ is called the \textit{associated charge}.
		\end{definition}
		
		Note that these associated charges correspond to the elements of $\Gamma_q$ that were identified with the edges of a $ca$-triangulation in section \ref{sec_hom}. This identification now allows us to see that saddle connections correspond to specific angles of the $ca$-triangulation. For this we consider the period map\footnote{This is sometimes also called the (central) charge function.} $Z_\gamma(q^{1/2})=\int_\gamma q^{1/2}$, for $\gamma\in\Gamma_q$ and $q^{1/2}$ denoting the tautological 1-form on $T^\ast \mathcal{C}$.
	
	\begin{lemma}\label{angle_prop}
		For a $\vartheta$-saddle connection of $q$ with associated charge $\gamma$ it holds $e^{-i\vartheta}Z_{\gamma}(q^{1/2})\in\mathbb{R}_+$.
	\end{lemma}
	
	\begin{proof}
		This follows directly from integrating the differential along the $\vartheta$-trajectory.
	\end{proof}
	
	Thus it is clear, that proportional elements of $\Gamma_q$ correspond to the same ray in the $\zeta$-plane. Many of the general properties proven in \cite{Bridgeland.2015} now only hold, when there are no non-proportional elements defining the same $\zeta$-ray.\footnote{Note that a quadratic differential with this property is called \textit{generic} in \cite{Bridgeland.2015}. But we already use \textit{generic} as the notion for a differential with only simple zeroes and we will continue to do so in this section.} Here we focus notationally on the complement of this set, which are elements of the \textit{walls of marginal stability} in the work of GMN.
	
	\begin{definition}\label{chambers}
		A generic quadratic differential with unsealed spectrum shall be called a \textit{brick} if there exist distinct saddle connections of $q$ such that for their associated elements $\gamma_1$ and $\gamma_2$ in $\Gamma_q$ it holds $\mathbb{Z}\gamma_1\neq\mathbb{Z}\gamma_2$, but $Z_{\gamma_1}(q^{1/2})/Z_{\gamma_2}(q^{1/2})\in\mathbb{R}_+$.
		
		The connected components of the complement of the set of bricks in the set of generic quadratic differentials with unsealed spectrum shall be called \textit{chambers}. 
	\end{definition}
	
	The $Z_\gamma(q^{1/2})$ for the associated charges $\gamma$ for certain saddle connections of $q$ form local coordinates on the space of quadratic differentials, cf. section 4 in \cite{Bridgeland.2015}. From this it follows that the set of bricks is a countable union of real codimension $1$ submanifolds, from which the name of a wall (and thus of a brick here) arises. From this it follows directly that the chamber-elements are dense in the set of all generic quadratic differentials with unsealed spectrum. In our construction we need to vary a differential without hitting a wall, so we need the following lemma.
	
	\begin{lemma}\label{chambers_open}
		Chambers are open in the space of all quadratic differentials.
	\end{lemma}
	
	\begin{proof}
		Let $q$ be an element of a chamber. So at any given angle $\vartheta$ at most one element of $\Gamma_q$ (up to integer multiples) is the associated charge of some $\vartheta$-trajectory. 
		
		
		For a contradiction suppose that in every neighborhood $U_t$ of $q$ there exists a brick $q_t$. Thus in every $U_t$ there are distinct saddle trajectories $c^1_t$ and $c^2_t$. They appear for angles $\vartheta_t\in[0,2\pi)$, so we obtain an infinite sequence of angles, which accumulates at some $\vartheta\in[0,2\pi)$ and we pass to a corresponding subsequence converging to $\vartheta$ for $t\to\infty$. As $q$ is not a brick, it follows from Lemma 5.1 in \cite{Bridgeland.2015} that either there is only one finite trajectory or there are two finite trajectories which form the boundary of a ring domain. 
		
		We first consider the case of a single trajectory and distinguish two cases. First we suppose that there is no $t_0$ such that for all 
		$t>t_0$ $\vartheta_t=\vartheta$. Thus near $q$ there are at least two $ca$-trajectories for angles arbitrarily close to $\vartheta$. Now since $q$ has only one finite $\vartheta$-trajectory it follows from Proposition 5.5 in \cite{Bridgeland.2015} that in a sufficiently small neighborhood of $q$ there can't be a $\vartheta+\epsilon$-trajectory for small enough $\epsilon\neq0$. So this case cannot occur.
		
		On the other hand, if such a $t_0$ exists, we obtain a sequence of differentials $q_t$ near $q$ which all have at least two distinct $\vartheta$-trajectories. But, as of Lemma 5.2. in \cite{Bridgeland.2015} the subset of differentials with only one $\vartheta$-trajectory is open in the space of all quadratic differentials. So this case also cannot occur.
		
		For the case of a ring domain the analogous statements follow from the same results in \cite{Bridgeland.2015}. Thus a sequence of wall-elements can only accumulate at wall-elements, so the subset of non-wall elements is open.
		
		
	\end{proof}
	
	Before moving on, let us consider the spectrum in some more detail. The use of some additional technique of \cite{Bridgeland.2015} will allow us to loosen the restriction of the differential in Proposition \ref{symp_is_cont} quite a bit. In the definition of the spectrum we did not count the infinitely many closed $\vartheta$-trajectories which sweep out a ring domain as separate elements, although they are finite $ca$-trajectories. The reason for this is the fact that the ring domain as a whole occurs at some specific angle $\vartheta_c$ and leads to a (collective) jump in the $ca$-triangulation at $\vartheta_c$, as was shown in \cite{Gaiotto.2013}. Additionally all these closed trajectories have the same lift on the spectral cover as they are homotopic. But the existence of a ring domain still does lead to an infinite spectrum, because near $\vartheta_c$ an infinite sequence of saddle connections may occur. These saddle connections lead to distinct jumps of the triangulation and accumulate at $\vartheta_c$. So it seems quite natural for the spectrum of a generic differential to be infinite, but the way in which a generic differential with unsealed spectrum is infinite is well understood. Indeed, infinitely many saddle connections occur only near ring domains (cf. section 5.3. in \cite{Tahar.2017}) and in the case of an unsealed spectrum their behaviour near such ring domains is completely described in \cite{Gaiotto.2013}. This also implies that ring domains do not accumulate, as they would have to accumulate near a ring domain, which is not possible as of the analysis in \cite{Gaiotto.2013}. Therefore we have the following observation:
	
	\begin{lemma}\label{unsealed_finite}
		An unsealed spectrum has at most finitely many ring domains.
	\end{lemma}
	
	The analysis in \cite{Gaiotto.2013} tells us how to calculate the jumps for ring domains and their corresponding infinite amount of saddle connections. With this lemma it is now clear, that there can only be finitely many additional jumps from single saddle connections and thus all jumps can be calculated. Otherwise, if there were an infinite sequence of ring domains, there would be technical difficulties in defining the right order for the occurring jumps. But for unsealed spectra we don't have to worry about this. 
	
	We will come back to this property of only having finitely many accumulation points later on, when considering the associated Riemann-Hilbert problem. 
	Here the only condition which really seems to shorten the applicability of our results is the existence of signet rings, but there is a way to lessen their impact, which we will describe now.
	
	%
	
	\begin{remark}
	Suppose $q$ has a spectrum that is at first not unsealed. Thus Proposition \ref{symp_is_cont} does not work as there is a ring domain cutting off a torus without singularities. In those cases, as was shown in \cite{Bridgeland.2015}, the ring domain has boundaries which are made up of some number of saddle connections and are thus \textit{strongly non-degenerate} in the terminology of Bridgeland and Smith. It is then shown that such strongly non-degenerate ring domains can be removed without altering the genus of the surface or the differential outside of the ring domain. This process is called ring-shrinking and can roughly be considered as compressing the ring domain into a single closed curve.
	
	After the ring shrinking the jump of the former ring domain becomes one of the jumps that was shown to be fine. If there are only finitely many signet rings it should then be possible to obtain an unsealed spectrum after finitely many ring shrinking moves, so there is no danger of shrinking the surfaces to a point or changing the topology in any other way, which could broaden the applicability of the theory presented here a lot.
	
	However, one has to show that the ring shrinking, which eliminates a ring domain for one specific $\vartheta_c$ does not produce another ring domain at some other angle. Additionally, we are interested in properties for the moduli space for varying differentials. Therefore we would  need the ring shrinking to work uniformly in at least some open subset of the Hitchin base. But varying the differential in the wrong direction might exactly reproduce the ring domain we just got rid of, cf. section 5.9 \textit{Juggles} in \cite{Bridgeland.2015}. Finally, it seems reasonable that there are only finitely many signet rings in the spectrum of any given generic differential, but this is not rigorously proven yet.
	\end{remark}
	
	Further explorations of these phenomenons would be very useful to obtain the most general setting for our theory. Additionally, as we explained above, the jumps corresponding to the shrunken ring domains may actually also be such that the holomorphic symplectic form $\varpi$ is well defined, and it may be possible to obtain a proof for those jumps from the arguments given in \cite{Gaiotto.2013b}. Such a proof, written in terms of the $\mathcal{X}$-coordinates, would be quite useful to obtain more information about the actual limits of this construction. But for now and for the rest of this section we will be content with the demand for $q$ to have an unsealed spectrum.

\section{Local description of the solutions to Hitchin's equations}\label{local_desc}
	
	In this section we calculate local forms of all the objects we need in a way that fits the aforementioned
	construction of the canonical coordinates. We start by explaining the fundamental result of \cite{Fredrickson.2020} in subsection 
	\ref{frame_sec} and then build a standard frame over certain subsets of the standard quadrilaterals that were introduced in the previous section. This 
	will then allow us to calculate local forms of the Harmonic metric, Higgs field and Chern connection in subsection \ref{Hitchin_pair} that are suited for 
	the construction.
	
	\subsection{Construction of the standard frame}\label{frame_sec}
	
	Our construction of the metric will mostly be done via local data of the solutions of Hitchin's equation, for which we build on the theory of limiting 
	configurations developed in \cite{Fredrickson.2020}. For the rest of this section we fix a stable $SL(2,\mathbb{C})$-Higgs bundle 
	$(\mathcal{E},\Phi)$ in the regular locus $\mathcal{M}'$ of $\mathcal{M}$ s. t. $-q:=\det\Phi$ has only simple zeroes. FMSW constructed solutions $h_R$ of 
	the $R$-rescaled Hitchin equation
	\[
		F^{\bot}_{A_{h_R}}+R^2\left[\Phi,\Phi^{\ast_{h_R}}\right]=0
	\]
	for large $R\in\mathbb{R}$.\footnote{In \cite{Fredrickson.2020} this parameter is $t$, while our notation comes from the work of GMN. As we'll need $t$ 
	mostly as the usual parameter of the curve, it seems reasonable to side with GMN on this question.} The idea behind this rescaling is, that for large 
	$R$, corresponding to large Higgs fields, Hitchin's equation would decouple leading to a much simpler set of equations that can be solved explicitly. The 
	decoupled equations are the \textit{limiting equations}
	\[\label{lim_equ}
		F^{\bot}_{A_{h_\infty}}=0,\quad\text{and}\quad\left[\Phi,\Phi^{\ast_{h_\infty}}\right]=0
	\]
	and the harmonic metric $h_\infty$ is called the limiting metric, which corresponds to the limiting pair $(A_\infty,\Phi_\infty)$. Local forms near 
	zeroes and poles of these objects are given in \cite{Fredrickson.2020} which were then glued together to create an approximate solution $h_{\text{app}}$ 
	of the true 
	Hitchin equation.\footnote{From here on out we will use "true" as signifying solutions which correspond to Hitchin's equation to separate them from 
	approximate solutions or solutions of the decoupled equations.} The foundational result for our work that they obtained is that for large enough $R$ 
	there exists a gauge transformation that transforms the approximate solution into a true solution. We paraphrase the result here which 
	is found as Theorem 6.2. in \cite{Fredrickson.2020}.
	
	\begin{theorem}
		There exists $R_0>1$ such that for every $R\geq R_0$ there exists an $h_0$-Hermitian endomorphism $\gamma_R$, s. t. $g\cdot h_{\text{app}}$ for 
		$g=\exp(\gamma)$ is a true solution of the $R$-rescaled 
		Hitchin equation. Furthermore $\gamma_R$ is unique amongst endomorphisms of small norm.
	\end{theorem}
	
	Not only did FMSW show the existence of these approximate solutions and the correction mechanism. They also showed that the transformation is 
	exponentially suppressed in $R$.  These results are foundational for the work at hand as they seem to imply that other objects that were constructed 
	with the help of these solutions might as well inherit the structure of an explicit approximate solution plus some correctional term that is 
	exponentially suppressed. In order to really make use of the results one would however need a way of implementing it in an appropriate construction of 
	e. g. the hyperkähler metric of $\mathcal{M}$. As it turns out, the proposed construction by GMN allows for precisely such an implementation as we want 
	to show here.
	
	For all of this to work we will have to make heavy use of the explicitly known local forms of all the moving parts which will later be the ingredients of 
	our flat connection for the GMN construction. So our goal here is to write the ingredients of the twistor connection as a diagonal leading term plus some 
	remainder or "error"-term which behaves in a certain way when going into a parabolic point, as well as when considering large $R$.
	
	In the following $z$ will always be a local coordinate on $\mathcal{C}$ with $r:=|z|$ in any coordinate centered around a parabolic point. We start by 
	considering the local frame $(e_1,e_2)$ in which we would like to work. We are interested in frames that are adjusted to the Higgs field 
	$\Phi=\varphi dz$, so that $\varphi$ has diagonal form
	\[
		\varphi=\begin{pmatrix}f^{1/2}&0\\0&-f^{1/2}\end{pmatrix}
	\]
	for some square root $f^{1/2}$ of $f=-\det\varphi$.
	
	\begin{lemma}\label{diag_struc}
		Let $(\mathcal{E},\Phi)$ be a parabolic Higgs bundle with local Higgs field $\Phi=\varphi dz$ with determinant $\det\Phi=\det\varphi dz^2=-fdz^2$. In 
		any gauge in which $\varphi$ is diagonal, the 
		holomorphic structure splits, i. e. is also represented by a diagonal matrix.
	\end{lemma}
	
	\begin{proof}
		Assume $\varphi$ to be diagonal, i. e. of the form
		\[
		\varphi=\begin{pmatrix}f^{1/2}&0\\0&-f^{1/2}\end{pmatrix}
		\]
		for some square root $f^{1/2}$ of $f$. We now write the holomorphic structure locally as
		\[
			\overline{\partial}_E=\overline{\partial}+A_{\overline{z}}d\overline{z}=\overline{\partial}
			+\begin{pmatrix}\overline{a_1}&\overline{a_2}\\ \overline{a_3}&
			\overline{a_4}\end{pmatrix}d\overline{z},
		\]
		where $\overline{\partial}$ is the trivial holomorphic structure induced by the coordinate $z$. As $(\Phi,\overline{\partial}_E)$ is a Higgs bundle 
		it holds $\overline{\partial}_E\Phi=0$, i. e.
		\[
			0=\partial_{\overline{z}}\varphi+[A_{\overline{z}},\varphi].
		\]
		In some more detail: Let $s$ be any section, then it holds:
		\[
			\begin{split}
			0&=(\overline{\partial}_E\varphi)s=\overline{\partial}_E(\varphi s)-\varphi\overline{\partial}_E(s)
			=\partial_{\overline{z}}(\varphi s)+A_{\overline{z}}\varphi s-\varphi\partial_{\overline{z}}s-\varphi A_{\overline{z}}\\
			&=\begin{pmatrix}\partial_{\overline{z}}(f^{1/2}s_1)\\-\partial_{\overline{z}}(f^{1/2}s_2)\end{pmatrix}
			+\begin{pmatrix}\overline{a_1}&\overline{a_2}\\ \overline{a_3}&\overline{a_4}\end{pmatrix}\begin{pmatrix}f^{1/2}&0\\0&-f^{1/2}\end{pmatrix}s\\
			&-\begin{pmatrix}f^{1/2}&0\\0&-f^{1/2}\end{pmatrix}\begin{pmatrix}\partial_{\overline{z}}s_1\\\partial_{\overline{z}}s_2\end{pmatrix}
			-\begin{pmatrix}f^{1/2}&0\\0&-f^{1/2}\end{pmatrix}\begin{pmatrix}\overline{a_1}&\overline{a_2}\\ \overline{a_3}&\overline{a_4}\end{pmatrix}s\\
			&=\begin{pmatrix}(\partial_{\overline{z}}f^{1/2})s_1\\-(\partial_{\overline{z}}f^{1/2})s_2\end{pmatrix}
			+\begin{pmatrix}-2\overline{a_2}f^{1/2}s_2\\2\overline{a_3}f^{1/2}s_1\end{pmatrix}.
			\end{split}
		\]
		Plugging in $s=(1,0)$ and $s=(0,1)$ we obtain $\overline{a_2}=\overline{a_3}=0$.
	\end{proof}
	
	We now introduce the standard neighborhood on which most of the following local calculations take place, for which we recall the $ca$-triangulations of 
	section \ref{ca_triangulations}. We take $-q=\det\Phi$ and consider any generic angle $\vartheta$ i. e.  the corresponding $\vartheta$-triangulation has 
	no finite $\vartheta$-trajectories. For any edge of the triangulation we consider the two triangles that share this edge. They form a quadrilateral $Q$
	bounded by four different $\vartheta$-trajectories inside some bigger open patch $U\subset C$. The trajectories connect four different weakly parabolic 
	points and inside the quadrilateral are two (distinct) zeroes of $\det\phi$ (cf. figure \ref{quadrilateral}).
	
	Now we may take a simple closed curve in $U$ that touches all four weakly parabolic points and whose interior $O\subset U$ contains all four $\vartheta$-
	trajectories that bound the quadrilateral. Inside $O$ we may also take a simple closed curve whose interior $\widetilde{O}\subset O$ contains both zeroes of 
	$\det\Phi$. Then $S:=O\setminus\widetilde{O}$ is homeomorphic to an annulus, does not contain the zeroes or poles of $\det\Phi$, but does contain the 
	$\vartheta$-trajectories that form the boundary of the quadrilateral. 
	\begin{definition}
		Let $q$ be a generic differential and $\vartheta$ a generic angle. Any quadrilateral obtained from two triangles that share an edge shall be called a 
		\textit{standard quadrilateral} and any domain $S\in\mathcal{C}$ constructed in the way described above out of $Q$ shall be called 
		a \textit{standard annulus}\footnote{Although such a domain is homeomorphic to an annulus it has to be noted that as the $\vartheta$-trajectories form 
		logarithmic spirals near the parabolic points the whole domain forms such a spiral structure there. The important fact is, that the boundary of the 
		annulus can be taken between any such two spirals.} of the triangulation. In particular such a domain does not contain any of the critical points of 
		$-q$ 
		and any simple closed curve in $S$ surrounds either no or two zeroes of $-q$ in $Q$.
	\end{definition}
	Note that if we vary $-q=\det\Phi$ or the the angle $\vartheta$ of the $ca$-triangulation the poles, zeroes and curves may vary. If we do this continuously 
	depending on some parameter $\epsilon$ all the parts depend continuously on $-q$ and $\vartheta$ (except for the critical values for $\vartheta$). Thus 
	it is always possible to choose a domain that is a standard annulus for all $\epsilon$ in some interval.
	
	\begin{figure}[!htbp] 
		\centering
		\scalebox{0.4}{\includegraphics{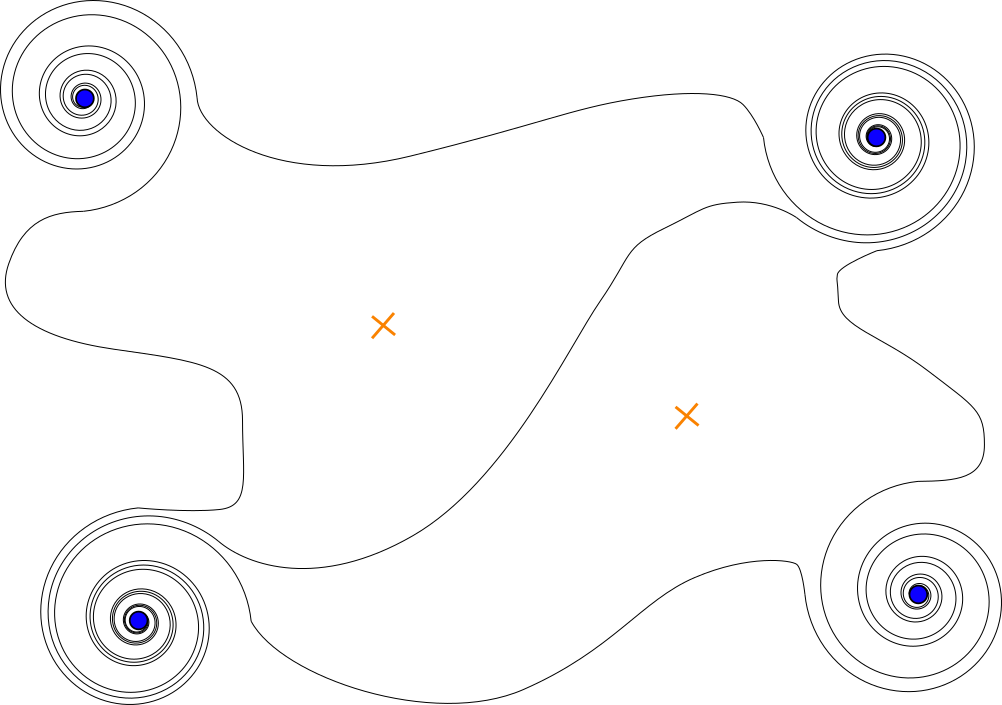}}
		\caption{A standard quadrilateral consisting of two triangles that share a common $ca$-trajectory. The blue dots represent weakly parabolic points and 
		the orange crosses represent zeroes of $-q$.}
		\label{quadrilateral}
	\end{figure}
	
	Thus defined we now establish that a standard annulus is a good working environment for our case. For the rest of this section a Higgs field $\Phi$ with 
	determinant $-q$ a generic differential and a generic angle $\vartheta$ shall be fixed.
	
	\begin{lemma}
		Let $\Phi=\varphi dz$ be a Higgs field on a standard quadrilateral $Q$ standard annulus $S$. Then there exists a (holomorphic) frame on $S$ in which 
		$\varphi$ is diagonal. 
	\end{lemma}
	
	\begin{proof}
		Let $c:[a,b]\rightarrow C$ be a simple closed curve in $S$. Then $c$ is homotopic to a point or $c$ surrounds the two zeroes of $-q=\det\Phi$ in $Q$. 
		As $c$ can't surround only one zero a standard result of complex analysis shows that a well defined square root $q^{1/2}$ on the whole annulus can be 
		chosen. (This corresponds to choosing one sheet of the double cover by the spectral curve.) In every point $x$ in $S$ the eigenvalues $\pm q^{1/2}(x)$ 
		are distinct and as the square root is well defined they don't interchange. Thus we 
		obtain two disjoint eigenspaces $E_{x_+}$, $E_{x_-}$ in every point $x$ which form two distinct line bundles $E_+$, $E_-$ on $S$. Locally we can thus 
		choose a frame in which $\varphi$ is diagonal. As $\overline{\partial}_E(\Phi)=0$, the diagonality of $\varphi$ implies as of Lemma \ref{diag_struc} that 
		the holomorphic 
		structure $\overline{\partial}_E$ is diagonal, i. e. splits into a holomorphic structure on each line bundle. Thus $E_+$ and $E_-$ are holomorphic line 
		bundles. Since they are defined on the non-compact Riemann surface $S$ they have to be holomorphically trivial (see \cite{Forster.1981} p. 229). Thus 
		we can choose some non-vanishing sections over $S$ with respect to which $\varphi$ will be diagonal on all of $S$.
	\end{proof}
	
	In the following we compute the local forms of the Higgs field, its adjoint and the Chern connection of the holomorphic structure w. r. t. the Hermitian 
	metric $h_R$ on a standard annulus or subsets thereof. For these computations we first need to collect the information on $h_{R}$ from 
	\cite{Fredrickson.2020} which for simplicity we will just call $h:=h_{R}$ for the rest of the section keeping in mind that all the constructions are considered 
	for large enough $R$. So we start by choosing the approximate metric $h_0$, 
	which was defined in \cite{Fredrickson.2020} as $h_R^{app}$, as our background metric on $E$. In \cite{Fredrickson.2020} it is shown for 
	$\text{pdeg}(\mathcal{E})=0$ that 
	$h_0=h_R^{app}$ is adapted to the parabolic structure for weakly parabolic bundles and induces the fixed flat Hermitian structure 
	on Det$E$.\footnote{While it seems likely that their arguments can be adapted to case of pdeg$(\mathcal{E})\neq0$ we will constrain ourselves to case of pdeg$(\mathcal{E})=0$ where it is useful.} The metric is 
	defined in such a way that away from zeroes of $q$ the eigenspaces of $\Phi$ are orthogonal, thus $\Phi$ is represented by a diagonal matrix in any frame given by sections of 
	the eigenspaces such as the frame constructed above in which $\varphi$ is diagonal. In \cite{Fredrickson.2020} we also find local
	forms for $h_0$ near all the different exceptional points for $q$. Here we only note the holomorphic frame above can be rescaled into a frame
	$F^\ast:=(e^\ast_1,e^\ast_2)$ on $S$ 
	s. t. the matrix $H_0$ representing $h_0$ has the following form near the (double) poles $p_i$ of $q$:
	\[
		H_0=Q\begin{pmatrix}|z|^{2\alpha_1(p)}&0\\0&|z|^{2\alpha_2(p)}\end{pmatrix}.
	\]
	Here $\alpha_1,\alpha_2$ are the parabolic weights and $Q$ is some locally defined and 
	non-vanishing smooth function, completely determined by $h_{\text{Det}\mathcal{E}}$, the choice of holomorphic section of Det$\mathcal{E}$ and coordinate 
	$z$. Note that this metric coincides with $h_\infty$ in \cite{Fredrickson.2020}. For 
	our calculations it is convenient to change to an $h_0$-unitary frame near the $p_i$, so we define
	\[
		\tilde{e_1}=\frac{e^\ast_1}{Q^{1/2}|z|^{\alpha_1}},\quad\tilde{e}_2=\frac{e^\ast_2}{Q^{1/2}|z|^{\alpha_2}}.
	\]
	Then it holds $H_0=E_2$ in the frame $(\tilde{e}_1,\tilde{e}_2)$ near the $p_i$. 
	
	Away from poles or zeros of det$\Phi$ the metric is still diagonal in any holomorphic frame and has the general form
	\[
		H_0=\begin{pmatrix}f&0\\0&g\end{pmatrix}
	\]
	for some non-vanishing functions $f,g$. Thus via some diagonal transformation we obtain the local form $H_0=E_2$ everywhere away from zeros of det$\Phi$ in this frame $\tilde{F}$. 
	Note that scaling away $Q$ is not strictly necessary but rather simplifies a lot of the following expression. It is also convenient to make another ($h_0$-unitary) transformation to obtain a certain form for the matrix-coefficients representing the holomorphic structure $\overline{\partial}_E$.
	
	\begin{lemma}\label{s_frame}
		On the standard annulus $S$ there exists an $h_0$-unitary frame $F$ in which the Higgs field is diagonal as is the matrix representing the holomorphic structure $\overline{\partial}_E$. Additionally the Chern connection w. r. t. $h_0$ is represented by a $1$-form valued matrix $C$ with $dC=0$ for the standard differential
		$d=\partial+\overline{\partial}$.
	\end{lemma}
	
	\begin{proof}
	Let $\tilde{F}$ be the aforementioned $h_0$-unitary frame. In this frame the Higgs field is diagonal and as of Lemma \ref{diag_struc} the holomorphic 
	structure is of the form
	\[
		\overline{\partial}_E=\overline{\partial}+\begin{pmatrix}\tilde{a}_1&0\\0&\tilde{a}_2\end{pmatrix}d\overline{z},
	\]
	for some differentiable functions $\tilde{a}_1$ and $\tilde{a}_2$. Again we use the fact that the standard annulus is a non-compact Riemann surface and thus there exist functions $A_1$ and $A_2$, s. t. $\overline{\partial}A_1=\tilde{a}_1d\overline{z}$ and $\overline{\partial}A_2=\tilde{a}_2d\overline{z}$ (see \cite{Forster.1981}, Theorem 25.6). With $T_1:=-i\text{Im}(A_1)$ and $T_2:=-i\text{Im}(A_2)$ define the (unitary) transition matrix
	\[
		g:=\begin{pmatrix}e^{T_1}&0\\0&e^{T_2}\end{pmatrix}.
	\]
	and let $F$ denote the frame obtained from this change of basis. It is still $h_0$-unitary, so $H_0=E_2$ while the holomorphic structure has the form
	\[
		\overline{\partial}_E=\overline{\partial}+A:=\overline{\partial}+\begin{pmatrix}\tilde{a}_1&0\\0&\tilde{a}_2\end{pmatrix}d\overline{z}+\begin{pmatrix}T_1&0\\0&T_2\end{pmatrix}d\overline{z}.
	\]
	As the frame is unitary, the matrix representing the Chern connection w. r. t. $h_0$ is simply $C:=A-\overline{A}^t$. A straight forward calculation now shows that $dC=0$.
	\end{proof}
	
	This specific choice of frame will be useful in section \ref{sec_small_sec} as it allows for a simple construction of the leading terms of the Fock-Goncharov-coordinates. In almost all of the following calculations we will work in this frame, so we denote it as our standard.
	
	\begin{definition}
		The $h_0$-unitary frame $F$ constructed in Lemma \ref{s_frame} on a standard annulus $S$ shall be called the \textit{standard frame} on $S$.
	\end{definition}
	
	\subsection{Local forms of Higgs pairs}\label{Hitchin_pair}
	
	We now want to describe the harmonic metric $h$ in the $h_0$-unitary frame $F$, which as of \cite{Fredrickson.2020} we obtain for large enough $R$ by the 
	action of a gauge transformation $g$ acting on $h_0$ via $h=g\cdot h_0(v,w)=h_0(gv,gw)$. In the chosen gauge we have $h_0(v,w)=\overline{v}^tH_0 w$ and 
	thus
		\[
			g\cdot h_0(v,w)=\overline{gv}^tH_0 gw=\overline{v}^t\overline{g}^tH_0 gw=\overline{v}^t\overline{g}^tgw.
		\]
		Let $H$ be the matrix representing $h$ in this gauge, i. e. $h(v,w)=\overline{v}^tHw$. Then, as $h=g\cdot h_0$ we obtain 
		$H=\overline{g}^tH_0g=\overline{g}^tg$. 
		Now from \cite{Fredrickson.2020} we have $g=\exp(\gamma)$ for $\gamma$ $h_0$-Hermitian and traceless, so $g$ is also $h_0$-Hermitian and has 
		determinant $1$. We write $g$ accordingly as
		\[
			g=:\begin{pmatrix}a&b\\ \overline{b}&d\end{pmatrix}
		\]
		with $a,d$ real-valued functions on $S$ and $ad-|b|^2=1$.  The behavior of the entries of $g$ is 
	determined by the behavior of the entries of $\gamma$ in the following way:
	
	\begin{lemma}\label{g_gamma}
		Let $\gamma$ be tracefree and $h_0$-hermitian with $D:=\det(\gamma)$. Then
		\[
			g=\exp(\gamma)=\cosh(\sqrt{-D})E_2+\tfrac{\sinh(\sqrt{-D})}{\sqrt{-D}}\gamma.
		\]
	\end{lemma}
	
	\begin{proof}
		This is a standard calculation: As $\gamma$ is tracefree and $h_0$-hermitian we may write it as
		\[
			\gamma=\begin{pmatrix}l&o\\ \overline{o}&-l\end{pmatrix}
		\]
		for some real valued function $l$ and complex valued function $o$. Thus $D:=\det(\gamma)=-l^2-|o|^2\leq 0$ and $\gamma^2=-DE_2$. From this we obtain
		\[
		\begin{split}
			\exp(\gamma)&=\sum_{k=0}^\infty\frac{\gamma^k}{k!}=\sum_{k=0}^\infty\frac{\gamma^{2k}}{(2k)!}+\sum_{k=0}^\infty\frac{\gamma^{2k+1}}{(2k+1)!}\\
			&=\sum_{k=0}^\infty\frac{(-D)^kE_2}{(2k)!}+\sum_{k=0}^\infty\frac{(-D)^k\gamma}{(2k+1)!}
			=\sum_{k=0}^\infty\frac{\sqrt{-D}^{2k}}{(2k)!}E_2+\sum_{k=0}^\infty\frac{\sqrt{-D}^{2k}}{(2k+1)!}\gamma\\
			&=\sum_{k=0}^\infty\frac{\sqrt{-D}^{2k}}{(2k)!}E_2+\frac{1}{\sqrt{-D}}\sum_{k=0}^\infty\frac{\sqrt{-D}^{2k+1}}{(2k+1)!}\gamma
			=\cosh(\sqrt{-D})E_2+\frac{\sinh(\sqrt{-D})}{\sqrt{-D}}\gamma.
		\end{split}
		\]
	\end{proof}
		
		This Lemma relates the properties of $\gamma$ to the properties of $g$, from which we obtain the necessary information about $H$ as we have:
		\[
			H=\overline{g}^tg=\begin{pmatrix}a^2+|b|^2&(a+d)b\\(a+d)\overline{b}&d^2+|b|^2\end{pmatrix}
			=:\begin{pmatrix}\alpha&\beta\\ \overline{\beta}&\delta\end{pmatrix}.
		\]
	
	We will use the following observation in some calculations later on:
	
	\begin{lemma}\label{H_det_1}
		In the standard frame $F$ on the standard annulus the matrix $H$ that represents the Hermitian harmonic metric $h$ has determinant $1$.
	\end{lemma}
	
	\begin{proof}
		This follows readily from the discussion of the local form above as $g$ has determinant $1$.
	\end{proof}
	
	We are interested in the decaying behavior of the entries of $H$ near weakly parabolic points and we will also need the behavior of the entries of $H^{-1}\partial H$. Using \ref{H_det_1} we obtain
	\[
		H^{-1}\partial H=\begin{pmatrix}\delta\partial\alpha-\beta\partial\overline{\beta}&\partial(\beta\delta)\\-\partial(\alpha\overline{\beta})&-(\delta\partial\alpha-\beta\partial\overline{\beta})\end{pmatrix}=:\begin{pmatrix}h_1&h_2\\h_3&-h_1\end{pmatrix}.
	\]
	We now obtain all the information from the properties of $\gamma$ which were established in \cite{Fredrickson.2020}.
	
	\begin{lemma}\label{g_asymptotics}
		There exists a local coordinate $z$ centered around any weakly parabolic point, such that the following asymptotics hold for the functions defined 
		above (with $r:=|z|$) for some $\mu>0$ and all $1\leq i\leq3$:
		\begin{align*}
			a,d,\alpha,\delta=O(1), \quad b,\beta=O(r^\mu), \quad h_i=O(r^{\mu-1}).
		\end{align*}
		Furthermore, if $\epsilon$ denotes a $C^2$-coordinate of $\mathcal{M}'$, then the derivatives w. r. t. $\epsilon$ exist and have the same asymptotics:
		\begin{align*}
			\partial_\epsilon a,\enspace\partial_\epsilon d,\enspace\partial_\epsilon\alpha,\enspace\partial_\epsilon\delta&=O(1),
			 & \quad\partial_\epsilon b,\enspace\partial_\epsilon\beta\enspace&=O(r^\mu),\\
			\partial^2_\epsilon a,\enspace\partial^2_\epsilon d,\enspace\partial^2_\epsilon\alpha,\enspace\partial^2_\epsilon\delta&=O(1),
			 & \quad\partial^2_\epsilon b,\enspace\partial^2_\epsilon\beta&=O(r^\mu),\\
			 \partial_\epsilon h_i,\partial^2_\epsilon h_i&=O(r^{\mu-1}).
		\end{align*}
		Additionally there exist constants $C,\nu>0$ s. t.
		\[
			|r^{-\mu}b|,|r^{-\mu}\beta|,|r^{1-\mu}h_i|\leq Ce^{-\nu R}\quad\text{and}\quad |r^{1-\mu}\partial_\epsilon h_i|\leq Ce^{-\nu R}
		\]
		for large enough $R$.
	\end{lemma}
	
	\begin{proof}
		One important result in \cite{Fredrickson.2020} is that the gauge transformation $\gamma$ that transforms the 
		approximate solution of the 
		Hitchin equations lies in the Friedrichs domain $\mathcal{D}_{\text{Fr}}^{0,\alpha}(\nu)$. In particular $\gamma$ has a partial expansion in 
		(non-negative) powers of $r$, where the first term is a constant diagonal matrix. This gives the first
		asymptotics for the entries of $\gamma$ by definition of the domain .
		
		The $\epsilon$-differentiability for the entries of $\gamma$ follows from an appropriate enhancement of the application of the inverse function 
		theorem that shows the existence of $\gamma$. The asymptotics for the $\epsilon$-derivatives follow readily as the $r$-dependence in 
		$\mathcal{D}_{\text{Fr}}^{0,\alpha}(\nu)$ is independent of the $\epsilon$-dependence.
		
		Finally the exponential suppression is also an explicit part of the result of \cite{Fredrickson.2020}.
		
		This shows all the conditions for the entries of $\gamma$. As $\det(\gamma)$ is well defined and non-zero on the standard annulus it follows from 
		Lemma \ref{g_gamma} that the properties also hold for the entries of $g$ and correspondingly for $H$. The properties for the $h_i$ follow in the same 
		way, using the fact, that the first two polar derivatives also lie in the corresponding domain as shown in \cite{Fredrickson.2020}.
	\end{proof}
	
	With this information we can now calculate the rest of the protagonists of our construction.
	
	\begin{corollary}
		In the standard frame $F$ on the standard annulus $S$ the adjoint Higgs field (w. r. t. the harmonic metric $h$) can be written as
		\[
			\Phi^{\ast_h}=\left(\overline{f^{1/2}}\sigma^3
			+\overline{f^{1/2}}\varphi^{er}\right)d\overline{z}
		\]
		for $\varphi^{er}$ some traceless matrix whose entries are $O(r^{\mu})$ for some $\mu>0$ near any weakly parabolic point.
	\end{corollary}
	
	\begin{proof}
		Using \ref{H_det_1} the adjoint Higgs field $\Phi^{\ast_h}=\varphi^{\ast_h}d\overline{z}$ is calculated as
		\[
			\begin{split}
			\varphi^{\ast_h}&=H^{-1}\overline{\varphi}^tH
			=\begin{pmatrix}\delta&-\beta\\-\overline{\beta}&\alpha\end{pmatrix}\begin{pmatrix}\overline{f^{1/2}}&0\\0&-\overline{f^{1/2}}\end{pmatrix}
			\begin{pmatrix}\alpha&\beta\\ \overline{\beta}&\delta\end{pmatrix}\\
			&=\overline{f^{1/2}}\begin{pmatrix}\alpha\delta+|\beta|^2&2\delta\beta\\-2\alpha\overline{\beta}&-\alpha\delta-|\beta|^2\end{pmatrix}
			=\begin{pmatrix}\overline{f^{1/2}}&0\\0&-\overline{f^{1/2}}\end{pmatrix}
			+2\overline{f^{1/2}}\begin{pmatrix}|\beta|^2&\delta\beta\\-\alpha\overline{\beta}&-|\beta|^2\end{pmatrix}.
			\end{split}
		\]
		Thus the adjoint Higgs field consists of a diagonal leading term plus a remainder whose entries are all at least linear in $|\beta|$. As of Lemma 
		\ref{g_asymptotics} the remainder behaves in specified way.
	\end{proof}
	
	The final ingredient is the Chern connection for $\overline{\partial}_E$ for the harmonic metric $h$, which we also want to describe with a "simple" leading term and a "small" remainder.
	
	\begin{corollary}\label{trace_away}
		In the frame standard frame $F$ on the standard annulus $S$ the Chern connection $A$ 
		(w. r. t. the harmonic metric $h$) of $\overline{\partial}$ can be expressed as
		\[
			A=d+\begin{pmatrix}\overline{a_1}d\overline{z}-a_1dz&0\\0&\overline{a_2}d\overline{z}-a_2dz\end{pmatrix}+(a_2-a_1)A^{er}dz+H^{-1}\partial H,
		\]
		where $d=\partial+\overline{\partial}$ is the standard differential and $A^{er}$ is some traceless matrix whose entries are $O(r^{\mu})$ 
		for some $\mu>0$ near any parabolic point.
	\end{corollary}
	
	\begin{proof}
		By construction the holomorphic structure in the standard frame $F$ has the form
		\[
			\overline{\partial}_E=\overline{\partial}+A_{\overline{z}}
			=\overline{\partial}+\begin{pmatrix}\overline{a_1}&0\\0&\overline{a_2}\end{pmatrix}d\overline{z},
		\]
		for some function $a_1,a_2$. A direct calculation using that $H$ is $h_0$-hermitian shows that the Chern connection $d_A=d+A$ with respect to $h$ is given as 
		\[
			d+A=d+A_{\overline{z}}d\overline{z}-A_{\overline{z}}^{\ast_h}dz+H^{-1}\partial H.
		\]
		So we need to calculate $A_{\overline{z}}^{\ast_h}$, which becomes
		\[
		\begin{split}
			A_{\overline{z}}^{\ast_h}&=H^{-1}\overline{A_{\overline{z}}}H
			=\begin{pmatrix}\delta&-\beta\\-\overline{\beta}&\alpha\end{pmatrix}\begin{pmatrix}a_1&0\\0&a_2\end{pmatrix}
			\begin{pmatrix}\alpha&\beta\\ \overline{\beta}&\delta\end{pmatrix}dz\\
			&=\begin{pmatrix}a_1\alpha\delta-a_2|\beta|^2&(a_1-a_2)\delta\beta
			\\-(a_1-a_2)\alpha\overline{\beta}&a_2\alpha\delta-a_1|\beta|^2\end{pmatrix}dz
			=\begin{pmatrix}a_1&0\\0&a_2\end{pmatrix}+(a_1-a_2)\begin{pmatrix}|\beta|^2&\delta\beta
			\\-\alpha\overline{\beta}&-|\beta|^2\end{pmatrix}dz,
		\end{split}
		\]
		where we used \ref{H_det_1}. Thus the whole connection becomes
		\[
		d_A=d+\begin{pmatrix}\overline{a_1}d\overline{z}-a_1dz&0\\0&\overline{a_2}d\overline{z}-a_2dz\end{pmatrix}
		+(a_2-a_1)\begin{pmatrix}|\beta|^2&\delta\beta\\-\alpha\overline{\beta}&-|\beta|^2\end{pmatrix}dz+H^{-1}\partial H.
		\]
		Again the structure is that of a diagonal leading term and an remainder term whose entries depend at least linearly on $|\beta|$ and thus have the specified 
		decaying behavior.
	\end{proof}
	
	Note that the factor $H^{-1}\partial H$ should also be considered as part of the remainder term as it is part of the deviation from the diagonal leading 
	term. It consists only of data coming from the harmonic metric and thus only depends indirectly (via the solution of Hitchin's equation) on the Higgs 
	field and holomorphic structure. Most importantly it also has the pole-structure we need: In \ref{g_asymptotics} it was shown that the $h_i$ are 
	$O(r^{\mu-1})$ while $b$ and $\beta$ are $O(r^{\mu})$. This matches in the emerging formula in section \ref{const_on_mod}, as the other remainder terms for the Higgs 
	field and the Chern connection obtain a pole of order $1$ through $\overline{f^{1/2}}$ and $a_2-a_1$.
	
	\begin{remark}\label{trace_prop}
		We have obtained the structure of a leading term plus a "small" remainder for both the data from the Higgs field $\Phi$ and from the connection $A$ which 
		will together form the connection matrix for a connection $\nabla$ in the next section. The remainder
		$\overline{f^{1/2}}\varphi^{er}d\overline{z}+(a_2-a_1)A^{er}dz+H^{-1}\partial H$ of this matrix is traceless which is seen directly for the parts 
		$\varphi^{er}$ and $A^{er}$ and follows for $H^{-1}\partial H$ by a simple calculation from the fact that $H$ has constant determinant.
		
		As $\varphi$ is traceless by definition of our moduli space the only part that isn't comes from the leading part of the Chern connection $A$. One could 
		arrange for this part to also be traceless by shifting the parabolic weights, which simplifies some of the following calculations and which is actually 
		the form \cite{Gaiotto.2013} use. We don't do this here as we follow the set-up of \cite{Fredrickson.2020}. Additionally this shows that no specific 
		choice of weights is necessary. However, thinking of $A$ as traceless leads to more convenient formulas and our treatment actually shows that one can 
		think of these objects in this way without loosing any important structure.
	\end{remark}
	
	We have now obtained all of the data that we need for the following constructions and would like to compute the Darboux coordinate system as described in 
	section \ref{sec_dec}. That is, we can construct $\nabla(\zeta)=\tfrac{R}{\zeta}\Phi+A+R\zeta\Phi^{\ast_h}$ for $\zeta\in\mathbb{C}^\times$ and try to solve the flatness equation 
	$\nabla(\zeta)v=0$, where 
	we observe the aforementioned structure of a diagonal leading term and a "small" remainder. As it turns out it is possible to solve this equation along 
	$ca$-curves but to do so in a proper way, one needs some non-standard theory for singular ODEs which we will develop in the next section.

\section{Initial value problems at infinity}\label{init_prob}

	The crucial ingredient in our construction are the so called "small flat sections" obeying the flatness equation of a connection build out of a Higgs 
	pair. This connection, when trivialized 
	along a $ca$-curve, becomes a $2\times2$ non-autonomous linear ODE. The small flat sections are supposed to be solutions of this ODE that decay in a 
	certain way as $t\rightarrow\infty$. Thus we want to solve an ODE uniquely (up to a complex multiple) be demanding a certain behaviour at $\infty$, which 
	differs substantially from the usual boundary value problems where the boundary value is given at a finite time. This becomes a problem as we want to 
	differentiate the solutions with respect to a parameter and the standard theory requires a boundary value or initial value at some finite time. It turns 
	out, that to solve this problem the theory of Volterra integral equations is helpful, as they encode the boundary value in an explicit term of the 
	equation. This allows for a rigorous definition of the small flat section and furthermore exploration of their differentiable behavior. Therefore we will 
	use this section to study these types of equations. We will completely ignore the origin of the problem in the theory of Higgs bundles in this section 
	and instead develop the theory in its own purely analytical setting, which may also be of interest in its own right and for completely different 
	applications.
	
	Before we start let us briefly summarize, what is known for this sort of equations. Generally a linear Volterra equation of the second kind is an 
	equation of the form
	\begin{equation}\label{finite}
		x(t)=a(t)-\int_b^tK(t,s)x(s)ds.
	\end{equation}
	which has to be solved for some function $x$ which may take values in any vector space for which the operations make sense. The "second kind" refers to 
	the existence of the $a(t)$ term, which takes the aforementioned role of the initial value as one can see by evaluating $x$ at $t=b$. The term $K$ is the 
	so called integral kernel, which is the main object of study to determine whether solutions exists, are smooth and so on.
	
	The equations we consider however are different in the important detail, that the integral is evaluated up to infinity:
	\begin{equation}\label{infinite}
		x(t)=a(t)-\int_t^\infty K(t,s)x(s)ds.
	\end{equation}
	Now in this equation $a$ becomes the "boundary value at infinity". We know form the work of Levinson \cite{Levinson.1948} that in some cases a specific 
	solution can be constructed, but as one might guess, the general theory becomes a bit more intricate mostly because we can't work on compact intervals. 
	One article studying this problem is \cite{Franco.2008} by D. Franco and D. O'Regan. They show the existence of solutions for specific kernels  
	but mostly consider kernels with some periodic behavior that we don't assume here. Additionally they do not consider the differentiable dependence on 
	parameters.
	
	A better guiding light is instead the article \cite{Windsor.2010} by A. Windsor. He studies equations of the finite form \eqref{finite} and proves 
	existence, continuity and smoothness of the solutions with respect to parameters. Most of our arguments are adaptations of his treatment to our case, as is the general structure of this section. But 
	quite some work has to be done to make the theory work for the unbounded intervals on which the integral lives. This becomes most apparent when it 
	comes to the differentiable dependence, where the assumptions on the integral kernel are quite strong and in fact stronger, then one might deem 
	necessary, i. e. we have to demand that $K$ and $a$ are two times differentiable to obtain one time differentiability for our solution. We will also need 
	a product structure on the kernel that exists in our special case as the ODE from which the integral equation arises has a leading part and a bounded 
	remainder term. The product structure then arises as we transform the ODE into the integral equation. This is quite a special structure but it may appear in 
	different settings and the theory we develop should be worth looking at whenever an ODE splits into a well behaved leading part and some small 
	perturbation.

	\subsection{Existence and uniqueness}
	We want to start with proving the unique existence of a solution of equations of the type \eqref{infinite}. Our main tool for this is the  
	Banach-Caccioppoli
	contraction mapping principle, which we will also need for the smooth dependence.
	\begin{theorem}[Contraction mapping principle]\label{contraction}
		Let $(X,d)$ be a complete metric space and $P:X\rightarrow X$ be a contraction mapping, i. e. there exists $0\leq\lambda<1$ s. t. 
		\[
			d(P(x),P(y))\leq\lambda d(x,y)
		\]
		for all $x,y\in X$.\\
		Then $P$ has a unique fixed point, which will be denoted by $\varpi(P)$. Moreover, for any $x\in X$ we have
		\[
			d(x,\varpi(P))<\frac{d(x,P(x)}{1-\lambda}.
		\]
	\end{theorem}
	To use this we have to start by building our function spaces. We fix a positive number $n\in\mathbb{N}$ for the dimension of our target space (in our 
	application this would be $n=2$). For any $T\in\mathbb{R}$ consider functions on the interval $[T,\infty)$, so our space $X$ will be
	\[
	B_T:=\left\{f\in C([T,\infty),\mathbb{C}^n):
	\exists M\in\mathbb{R},\forall x\in[T,\infty):\left|f(x)\right|\leq M\right\},
	\] 
	the space of continuous bounded functions on the half line beginning with $T$ with values in $\mathbb{C}^2$ (where $\left|f(x)\right|$ denotes the usual 
	euclidean norm on $\mathbb{C}^n$). It is well known that $B_T$ is complete with respect to the metric $d_\infty$ induced by the sup-norm
	\[
		\left\|f\right\|_\infty:=\sup_{x\in[T,\infty)}\left|f(x)\right|,
	\]
	so from now on $B_T$ shall be equipped with this norm. In order to obtain the contraction property we will need our kernel to decay fast enough. They 
	will live on triangles starting with $T$, i. e. we write $\Delta(T):=\left\{(t,s)\in\mathbb{R}^2:T\leq t\leq s\right\}$. Then the 'good' kernels will be 
	the ones in
	\[
		\mathcal{K}_T:=\left\{K\in C(\Delta(T),M_n(\mathbb{C})):\sup_{t\in[T,\infty)}\int_t^{\infty}\left\|K(t,s)\right\|ds<1 \right\}.
	\]
	The bound for the  norm of the integral will be the contraction factor, so we denote it from now on by $\lambda:=\sup_{t\in[T,\infty)}\int_T^{\infty}\left\|K(t,s)\right\|ds<1$. 
	The choice of matrix norm is not important at this point, so we just take the operator norm. 
	We consider $\mathcal{K}_T$ equipped with the metric 
	$d_\mathcal{K}(K_1,K_2):=\underset{t\in [T,\infty)}{\sup}\int_t^{\infty}\left\|K_1(t,s)-K_2(t,s)\right\|ds$. 
	Now we are ready to formulate our existence and uniqueness result similiar to the one in \cite{Windsor.2010}:
	\begin{theorem}[Existence and uniqueness]\label{ex}
		Let $T\in\mathbb{R}$, $a\in B_T$ and $K\in\mathcal{K}_T$. Then there exists a unique solution of the integral equation
		\begin{equation}\label{int_eq}
			x(t)=a(t)-\int_t^\infty K(t,s)x(s)ds,
		\end{equation}
		on the intervall $[T,\infty)$ and this solution is an element of $B_T$.
	\end{theorem}

	\begin{proof}
		For any $\phi\in B_T$ we know from the integral property of $K$ that $\int_t^{\infty}K(t,s)\phi(s)ds$ exists for $t\geq T$. As $a$ is also bounded it 
		follows that 
		\[
			P(\phi)(t):=a(t)-\int_t^{\infty}K(t,s)\phi(s)ds.
		\]
		defines an operator $P:B_T\rightarrow B_T$. For two elements $\phi_1,\phi_2\in B_T$ we obtain
		\[
			\begin{split}	
				d_\infty(P(\phi_1),P(\phi_2))&\leq \sup_{t\in[T,\infty)}\left|\int_t^{\infty}K(t,s)(\phi_1(s)-\phi_2(s))ds\right|\\
				&\leq \sup_{t\in[T,\infty)}\int_t^{\infty}\left\|K(t,s)\right\|\left|\phi_1(s)-\phi_2(s)\right|ds\\
				&\leq \sup_{t\in[T,\infty)}\int_t^{\infty}\left\|K(t,s)\right\|ds\sup_{t\in[T,\infty)}\left|\phi_1(s)-\phi_2(s)\right|\\
				&=\lambda d_\infty(\phi_1,\phi_2),
			\end{split}		
		\]
		where $\lambda<1$ since $K\in\mathcal{K}_T$. Thus $P$ is a contraction mapping and we obtain a unique fixed point $x\in B_T$. 
		Although this solution is unique in $B_T$ there could still exist another (non-continuous) solution of \ref{int_eq}. This possibility may be excluded 
		as follows:
		
		A non-continuous solution would still lead to a continuous function
		\[
				a(t)-\int_t^\infty K(t,s)x(s)ds
		\]
		as $a$ is continuous, $K$ is continuous in $t$ and the integral over a bounded function is continuous. But as this expression equals $x$ the solution 
		has to be continuous in the first place.
	\end{proof}
		
		\begin{remark}
			Windsor gives another argument to prove that there can't be a non-continuous bounded solution, using that the difference of two solutions also 
			satisfies an integral equation and 
			obtaining a contradiction for the case of the difference being non-zero. His argument would also work for our case but it requires more estimates. It may be more interesting to note, that he can't use our argument as he doesn't demand his kernel to be continuous 
			in 
			$t$. Thus, although the integration over $s$ leads to a continuity from the integral, there may still be a non-continuous part in $K$.
		\end{remark}
		
		Later on, we will need the aforementioned estimate for the fixed point.
		
	\begin{corollary}\label{est}
		The following estimate holds for the solution $x$ of \ref{int_eq}:
		\begin{equation}
			\left\|x\right\|_\infty=d_\infty(0,x)\leq\frac{d_\infty(0,P(0))}{1-\lambda}=\frac{\left\|a\right\|_\infty}{1-\lambda}.
		\end{equation}
	\end{corollary}

	\subsection{Continuous dependence}

	Now that we have the unique existence, we consider the case, where $a$ and $K$ may vary in order to obtain first continuous and then differentiable 
	dependence on the parameters of the solution. We start with the continuous dependence by restating Windsors theorem and proof concerning arbitrary 
	contraction maps on metric spaces. For this denote by $Ctr(X)$ the space of contraction maps on a complete metric space $(X,d)$. Every contraction map is 
	Lipschitz continuous and therefore we can regard $Ctr(X)$ as a subspace of the space $C(X)$ of continuous functions on $X$. We consider on $Ctr(X)$ the 
	supremum $\infty$-metric $d_0$, i. e. $d_0(P,Q):=\sup_{x\in X}d(P(x),Q(x))$. Note that the value of $d_0$ may indeed be $\infty$ in general settings but this has no consequences for the questions regarding continuity here. Additionally in the main Theorem \ref{cont} of this subsection the space of interest is shown to be of finite diameter, where $d_0$ becomes an honest metric. As every contraction map $P\in Ctr(X)$ has a 
	unique fixed point 
	$\varpi(P)$, we can define the map
	\[
		\varpi:Ctr(X)\rightarrow X,\quad P\mapsto\varpi(P).
	\]
	With this \cite{Windsor.2010} first obtains the following result:

	\begin{lemma}
		The function $\varpi$ is continuous.
	\end{lemma}

	\begin{proof}
		Let $P\in Ctr(X)$. Then there is some $0\leq\lambda<1$, s. t.
		\[
			d(P(x),P(y))\leq\lambda d(x,y)
		\]
		holds for all $x,y\in X$. Let $\epsilon>0$ and define $\delta:=(1-\lambda)\epsilon$. Now let $Q\in Ctr(X)$ be a contraction mapping with $d_0(P,Q)<\delta$. We apply the estimate from the contraction mapping theorem for $P$ to $\varpi(Q)$ to obtain
		\[
		\begin{split}
			d(\varpi(Q),\varpi(P))&<\frac{d(\varpi(Q),P(\varpi(Q))}{1-\lambda}=\frac{d(Q(\varpi(Q)),P(\varpi(Q)))}{1-\lambda}\\
			&<\frac{\delta}{1-\lambda}=\epsilon.
		\end{split}	
		\]
		Thus $\varpi$ is a continuous function.
	\end{proof}
	Now we come to the continuous dependence for our solutions. For this denote the unique solution of \ref{int_eq} for any $a\in B_T$ and $K\in\mathcal{K}$ 
	by $x_{a,K}\in B_T$. (Note that $x_{a,K}$ thus does depend on $T$, but as $T$ is just a parameter that is fixed from the beginning, we don't write this 
	dependence into the notation.) We now claim:

	\begin{theorem}[Continuous dependence]\label{cont}
		Fix $T\in\mathbb{R}$. Then the map
		\[
			X:B_T\times\mathcal{K}_T\rightarrow B_T
		\]
		given by $X(a,K):=x_{a,K}$ is continuous (where $B_T\times\mathcal{K}_T$ is equipped with the usual product metric.)
	\end{theorem}

	\begin{proof}
		Again the proof works similar to the one found in \cite{Windsor.2010}. Fix $a\in B_T$ and $K\in\mathcal{K}_T$. We consider $\mathcal{B}:=\left\{\phi\in B_T:d_\infty(\phi,x_{a,K})<1\right\}$. By the arguments from the 
		beginning the operator
		\[
			P_{a,k}(\phi)(t):=a(t)-\int_{t}^{\infty}K(t,s)\phi(s)ds
		\]
		is a contraction mapping on $B_T$ and for $\phi\in\mathcal{B}$ it holds
		\[
			d_{\infty}(P_{a,k}(\phi),x_{a,K})=d_{\infty}(P_{a,k}(\phi),P_{a,K}(x_{a,K}))\leq\lambda d_{\infty}(\phi,x_{a,K})
		\]
		for $\lambda=\sup_{t\in[T,\infty)}\int_t^{\infty}\left\|K(t,s)\right\|ds<1$. Thus $P_{a,k}(\phi)$ descends to a contraction mapping on (i. e. takes 
		values in) $\mathcal{B}$ with contraction factor 
		$\lambda$.\\ 
		Now take $a'\in B_T$ and $K'\in\mathcal{K}_T$. Then it holds
		\begin{equation}\label{inequ}
			\begin{split}
			\left\|P_{a',K'}(\phi)-P_{a,K}(\phi)\right\|_{\infty}&\leq \left\|a'-a\right\|+\left\|\int_t^{\infty}(K'(t,s)-K(t,s))\phi(s)ds\right\|_{\infty}\\
			&\leq \left\|a'-a\right\|_{\infty}+\gamma\left\|\phi\right\|_{\infty},
			\end{split}
		\end{equation}
		for
		\[
			\gamma:=\underset{t\in [T,\infty)}{\sup}\int_t^{\infty}\left\|K'(t,s)-K(t,s)\right\|ds
			=d_\mathcal{K}(K,K').
		\]
		Note that $\gamma$ is finite as $K'$ and $K$ are elements of $\mathcal{K}_T$. Again as before $P_{a',K'}$ is a contraction mapping on $B_T$ and we now 
		want to obtain 
		a condition under which $P_{a',K'}$ also descends to a contraction mapping on $\mathcal{B}$. For this assume that $a'$ and $K'$ are such that
		\[
			\left\|a'-a\right\|_{\infty}+\gamma(\left\|x_{a,K}\right\|_{\infty}+1)<1-\lambda
		\]
		holds. Then via triangle inequality and the fact that $x_{a,K}$ is the fixed point of $P_{a,K}$ we obtain
		\[
			\begin{split}
				d_{\infty}(P_{a',K'}(\phi),x_{a,K})&\leq d_{\infty}(P_{a',K'}(\phi),P_{a,K}(\phi))+d_{\infty}(P_{a,K}(\phi),x_{a,K})\\
				&\leq \left\|P_{a',K'}(\phi)-P_{a,K}(\phi)\right\|_{\infty}+d_{\infty}(P_{a,K}(\phi),P_{a,K}(x_{a,K}))\\
				&\leq \left\|a'-a\right\|_{\infty}+\gamma\left\|\phi\right\|_{\infty}+\lambda d_{\infty}(\phi,x_{a,K})\\
				&\leq \left\|a'-a\right\|_{\infty}+\gamma(\left\|x_{a,K}\right\|_{\infty}+1)+\lambda\\
				&< 1-\lambda+\lambda=1, 
			\end{split}
		\]
		where we used the fact $\left\|\phi\right\|_{\infty}\leq\left\|x_{a,K}\right\|_{\infty}+1$ which follows from the inverse triangle inequality and the 
		definition of $\phi\in\mathcal{B}$. Therefore in this case $P_{a',K'}$ defines a contraction mapping on $\mathcal{B}$ and we define the corresponding set
		\[
			\mathcal{P}:=\left\{(a',K')\in B_T\times\mathcal{K}_T:\left\|a'-a\right\|_{\infty}+\gamma(\left\|x_{a,K}\right\|_{\infty}+1)<1-\lambda\right\}.
		\]
		This set contains $(a,K)$ and is open in $B_T\times\mathcal{K}_T$ as the defining inequality is non-strict.
		Now define the map $P:\mathcal{P}\rightarrow Ctr(B_T)$ by $P(a',K'):=P_{a',K'}$. As of the inequality \ref{inequ} we see that 
		$P\left(\mathcal{P}\right)$ is of finite diameter in $d_0$ and we obtain
		\[
		\begin{split}
		d_0(P_{a',K'},P_{a,K})&=\sup_{\phi\in \mathcal{B}}d_\infty(P_{a',K'}(\phi),P_{a,K}(\phi))\\
		&=\sup_{\phi\in \mathcal{B}}\left\|P_{a',K'}(\phi)(t)-P_{a,K}(\phi)(t)\right\|_{\infty}\\
		&\leq \sup_{\phi\in \mathcal{B}}\left\|a'-a\right\|_{\infty}+\gamma\left\|\phi\right\|_{\infty}\\
		&\leq d_{\infty}(a,a')+d_\mathcal{K}(K,K')(\left\|x_{a,K}\right\|_{\infty}+1).
		\end{split}
		\]
		Thus for $(a',K')$ close enough to $(a,K)$ the distance $d_0(P_{a',K'},P_{a,K})$ becomes arbitrarily small and we see that $P$ is continuous at 
		$(a,K)$. From the previous lemma, the map $\varpi:\left.Ctr(B_T)\right|_{P(\mathcal{P})}\rightarrow\mathcal{B}$ is continuous and thus 
		$\varpi\circ P:\mathcal{P}\rightarrow\mathcal{B}$ is continuous at $(a,K)$. But on the open subset $\mathcal{P}$ of $B_T\times\mathcal{K}_T$ we have 
		$\varpi\circ P=\left.X\right|_{\mathcal{P}}$ and thus $X$ is continuous at $(a,K)$.	
	\end{proof}
	Next up, we would like to reformulate the preceding theorem in terms of an additional parameter on which the equation depends in a continuous way, i. e. 
	the "boundary value" $a$ and kernel $K$ shall now depend continuously on some $\epsilon$. We expand our notation accordingly for some open subset 
	$U\subset\mathbb{R}$:
	\[
		\begin{split}
	B_{T,U}&:=\left\{f\in C([T,\infty)\times U,\mathbb{C}^n):\exists M\in\mathbb{R}:
	\forall\epsilon\in U\enspace\forall x\in[T,\infty):\left|f(x)\right|\leq M\right\},\\
		\mathcal{K}_{T,U}&:=\left\{K\in C(\Delta(T)\times U,M_n(\mathbb{C})):\forall\epsilon\in U:\sup_{t\in[T,\infty)}\int_t^{\infty}\left\|K(t,s,\epsilon)\right\|ds<1\right\}.
		\end{split}
	\] 
	For $a\in B_{T,U}$ we write the corresponding element in $B_T$ as $a_{\epsilon}$ (i. e. if we regard $\epsilon$ as fixed) and in the same way for 
	$K\in\mathcal{K}_{T,U}$ the element in $\mathcal{K}_{T}$ will be $K_{\epsilon}$. Now let any $a\in B_{T,U}$ and $K\in\mathcal{K}_{T,U}$ be given. Then 
	for each $\epsilon\in U$ there is a unique continuous solution $x_{\epsilon}\in B_{T}$ of
		\[
			x_{\epsilon}(t)=a_{\epsilon}-\int_t^{\infty}K_{\epsilon}(t,s)x_{\epsilon}(s)ds.
		\]
	as of theorem \ref{ex}. We can thus define the $\epsilon$-dependent solution as
	\[
		x:[T,\infty)\times U\rightarrow\mathbb{C}^2,\quad x(t,\epsilon):=x_{\epsilon}(t).
	\]
	With this notation we can now reformulate the continuity of $x$ by a standard argument:
	\begin{korollar}\label{cont_b}
		Let $a\in B_{T,U}$ and $K\in \mathcal{K}_{T,U}$ be such that the maps $\epsilon\mapsto a_{\epsilon}$ and $\epsilon\mapsto K_{\epsilon}$ are continuous. 
		Let $x$ be the solution of the parameter-dependent integral equation. Then $x$ is continuous in $\epsilon\in U$ and $t\in[T,\infty)$.
	\end{korollar}

	\begin{proof}
	From our assumption the map $\epsilon\mapsto a_{\epsilon}$ is continuous as is the map $\epsilon\mapsto K_{\epsilon}$. Theorem \ref{cont} then tells us 
	that the concatenation with the map $X:B_T\times\mathcal{K}_T\rightarrow B_T$ is continuous, i. e. $\epsilon\mapsto x_{\epsilon}$. Thus the map 
	\[
		f:[T,\infty)\times U\rightarrow [T,\infty)\times B_T,\quad (t,\epsilon)\mapsto (t,x_{\epsilon})
	\]
	is continuous in $\epsilon$ and $t$. Finally the evaluation map
	\[
		e:[T,\infty)\times B_T\rightarrow \mathbb{C}^n,\quad (t,x_\epsilon)\mapsto x_\epsilon(t)
	\]
	is continuous. This can be derived from the fact that we are considering bounded continuous functions with the sup-norm by a standard $\epsilon-\delta$ 
	argument.
	\end{proof}

	Note that the continuity of the maps $\epsilon\mapsto a_{\epsilon}$ and $\epsilon\mapsto K_{\epsilon}$ does not follow simply from the assumption 
	$a\in B_{T,U}$ and $K\in \mathcal{K}_{T,U}$ as the functions in $B_{T,U}$ and $\mathcal{K}_{T,U}$ are defined on the open half line $[T,\infty)$. This is 
	a complication compared to Windsor, since he works with functions on a compact interval.

	\subsection{Differentiable dependence}\label{sec_diff_dep}

	We now want our initial conditions to depend differentiably on the parameter $\epsilon$. We would like to just assume that one derivative exists and is 
	continuous and obtain a differentiable solution in the same way as it was done in \cite{Windsor.2010}. But the non-compactness now demands some stronger conditions. So the greatest difficulty here is to formulate the right conditions which allow us to obtain the differentiable dependence, while still being loose enough so that they are fulfilled by the functions arising in our applications in section \ref{const_on_mod}. This amounts to demanding the right uniform behavior of our kernels $K$, for which we find the following notions to be adequate:
	
	A map $B:[T,\infty)\times U\rightarrow M_n(\mathbb{C})$ shall be called \textit{integrally bounded} iff for every $\epsilon\in U$ there exists an open 
	subset $V_\epsilon\subset U$ containing $\epsilon$ and a bounded function $g_\epsilon:[T,\infty)\rightarrow\mathbb{R}_{\geq0}$, s. t.
	\[
	\begin{split}
		\left\|B(s,\mu)\right\|\leq g_\epsilon(s)\enspace\forall\mu\in V_\epsilon,\enspace\forall(t,s)\in\Delta(T)
		\quad\text{and}\quad\int_T^{\infty}g_\epsilon(s)ds<\infty,
		\end{split}
	\]
	Note that the integral bound and the boundedness of $g$ generally don't imply one another. With this a kernel $K$ shall be called \textit{of product type} iff it can be written as 
	\[
		K(t,s,\epsilon)=A(t,s,\epsilon)B(s,\epsilon),
	\]
	for two functions $A:\Delta(T)\times U\rightarrow M_n(\mathbb{C})$ and $B:[T,\infty)\times U\rightarrow M_n(\mathbb{C})$. Furthermore $K$ shall be called 
	\textit{of product type of zeroth order}, iff it is of product type and $A$ is bounded by some constant $\alpha\in\mathbb{R}$, while $B$ is integrally 
	bounded. It shall be called of product type of order $(1,0)$, iff it is of product type of zeroth order and $\partial_\epsilon A$ exists and is bounded 
	be some constant $\beta\in\mathbb{R}$ and of order $(0,1)$, iff $\partial_\epsilon B$ exists and is integrally bounded. The general notion of product 
	type of order $(m,n)$ follows accordingly.\\
	With this notion we can define the space of kernels of any order:
	\[
		\mathcal{K}^{(m,n)}_{T,U}:=\left\{K\in\mathcal{K}_{T,U},\enspace K\enspace\text{is of product type of order $(m,n)$}\right\}.
	\]
	Note that this assumption on $K$ is not the most general one with which the following proofs work but it suffices for our set up. The important structure 
	here is that only $A$ depends on $t$ and this factor is (uniformly) bounded. This will allow us to extract an upper bound of the $t$-dependence out of 
	some of the following integrals. It is this uniform behaviour that we actually need for our results and the sufficient uniform bound will be provided in 
	our applications on the Higgs bundle moduli space later on.
	
	For our main theorem we will first need two results concerning the continuity of our ingredients, which we obtain from the following standard generalization 
	of the mean value theorem.

	\begin{proposition}
		Let $f:(a,b)\rightarrow E$ be a differentiable function from an open interval in $\mathbb{R}$ to a Banach space $E$. Then for all $x,y\in(a,b)$ it holds
		\[
			\left\|f(y)-f(x)\right\|\leq\left|y-x\right|\sup_{0\leq \tau\leq 1}\left\|f'(x+\tau(y-x))\right\|.
		\]
	\end{proposition}
 
	As a reminder for our set up the Banach spaces in question will be $\mathbb{C}^n$ and $M_n(\mathbb{C})$, the first with the standard norm and the second 
	with the operator norm. With this at hand we may now proceed. First up is the "boundary value":

	\begin{lemma}
		Let $a\in B_{T,U}$ be differentiable with respect to $\epsilon\in U$ and such that the derivative is still bounded, i. e.
		\[
			\frac{\partial a}{\partial \epsilon}\in B_{T,U}
		\]
		Then the map
		\[
		f:U\rightarrow B_T,\quad \epsilon\mapsto a_{\epsilon}
		\] 
		is continuous.
	\end{lemma}

	\begin{proof}
		We start with $a\in B_{T,U}$ and consider continuity of $f$ at some $\epsilon_0\in U$. Since $\tfrac{\partial a}{\partial \epsilon}\in B_{T,U}$ there 
		exists an $M\in\mathbb{R}$
		s. t. $\left|\tfrac{\partial a}{\partial \epsilon}(\mu,t)\right|<M$ for all $t\in[T,\infty)$ and 
		$\mu\in U$.\\
		For any $\epsilon_1\in U$ it then follows from the generalized mean value theorem:
		\[
			\left|a(\epsilon_0,t)-a(\epsilon_1,t)\right|
			\leq\sup_{0\leq\tau\leq1}\left|\tfrac{\partial a}{\partial \epsilon}((\epsilon_0+\tau(\epsilon_1-\epsilon_0)),t)\right|\left|\epsilon_0-\epsilon_1\right|
			\leq M\left|\epsilon_0-\epsilon_1\right|.
		\]
		Now for any $\nu>0$ define $\delta:=\min\left\{\tfrac{\nu}{M},r\right\}$. Then, if we have $\left|\epsilon_0-\epsilon_1\right|<\delta$ it follows
		\[
			\begin{split}
			\left\|f(\epsilon_0)-f(\epsilon_1)\right\|_{\infty}&=\sup_{t\in[T,\infty)}\left|f(\epsilon_0)(t)-f(\epsilon_1)(t)\right|
			=\sup_{t\in[T,\infty)}\left|a(\epsilon_0,t)-a(\epsilon_1,t)\right|\\
			&\leq\sup_{t\in[T,\infty)}M\left|\epsilon_0-\epsilon_1\right|\\
			&< M\delta\leq\nu.
			\end{split}
		\]
		So $f$ is continuous.
		\end{proof}
		
		Next up is the kernel:
		
		\begin{lemma}
		Let $K\in \mathcal{K}_{T,U}^{(1,1)}$. Then the map
		\[
		f:U\rightarrow \mathcal{K}_T,\quad\epsilon\mapsto K_{\epsilon}
		\] 
		is continuous.
	\end{lemma}
		
		\begin{proof}
			For the continuity of $f$ we have to work with the metric $d_\mathcal{K}$ mentioned above. We want to show the continuity at some $\epsilon_0\in U$. 
			For any other $\epsilon_1\in U$ we first obtain again by using the mean value estimate
		\[
			\begin{split}
				d_\mathcal{K}(K(t,s,\epsilon_0),K(t,s,\epsilon_1))&
				=\underset{t\in [T,\infty)}{\sup}\int_t^{\infty}\left\|K(t,s,\epsilon_0)-K(t,s,\epsilon_1)\right\|ds\\
				&\leq\underset{t\in [T,\infty)}{\sup}\int_t^{\infty}
				\sup_{0\leq\tau\leq1}\left\|\tfrac{\partial K}{\partial \epsilon}(t,s,\epsilon_0+\tau(\epsilon_1-\epsilon_0))\right\|
				\left|\epsilon_0-\epsilon_1\right|ds
			\end{split}
		\]
		Now $K$ is of product type of first order, i. e. 
		\[
			K(t,s,\epsilon)=A(t,s,\epsilon)B(s,\epsilon)\enspace\Rightarrow\enspace \frac{\partial K}{\partial\epsilon}(t,s,\epsilon)
			=\frac{\partial A}{\partial\epsilon}(t,s,\epsilon)B(s,\epsilon)+A(t,s,\epsilon)\frac{\partial B}{\partial\epsilon}(s,\epsilon),
		\]
		where $A$ and $\partial_\epsilon A$ are bounded by some constants $\alpha,\beta\in\mathbb{R}$ and $B(s,\mu)$ and $\partial_\epsilon B(s,\mu)$ are 
		integrally bounded by some functions $g_0$ and $g_1$ for all $\mu\in U$ close enough to $\epsilon$. Therefor in the integral above, we obtain for 
		$\epsilon_1$ close enough to $\epsilon_0$ 
		\[
			\sup_{0\leq\tau\leq1}\left\|\frac{\partial K}{\partial \epsilon}(t,s,\epsilon_0+\tau(\epsilon_1-\epsilon_0)\right\|\leq \alpha g_1(s)+\beta g_2(s),
		\]
		and
		\[
		\begin{split}
			d_\mathcal{K}(K(t,s,\epsilon_0),K(t,s,\epsilon_1))&\leq\underset{t\in [T,\infty)}{\sup}\int_t^{\infty}\alpha g_1(s)+\beta g_2(s)ds\left|\epsilon_0-\epsilon_1\right|\\
			&\leq\int_T^{\infty}\alpha g_1(s)+\beta g_2(s)ds\left|\epsilon_0-\epsilon_1\right|
			\leq C\left|\epsilon_0-\epsilon_1\right|.
		\end{split}
		\]
		for some constant $C\in\mathbb{R}$. Therefore $d_\mathcal{K}(K(t,s,\epsilon_0),K(t,s,\epsilon_1))$ becomes arbitrarily small and $f$ is continuous.
			
	\end{proof}

	Now for the main theorem.
	\begin{theorem}\label{different}
		Let $a\in B_{T,U}$ be bounded and such that the first and second derivative with respect to $\epsilon$ exist and are elements of $B_{T,U}$. 
		Let $K\in \mathcal{K}^{(2,1)}_{T,U}$ be of product type of order $(2,1)$.\\
		Then the solution $x(t,\epsilon)$ of the integral equation
		\[
			x(t,\epsilon)=a(t,\epsilon)-\int_t^{\infty}K(t,s,\epsilon)x(s,\epsilon)ds
		\]
		is differentiable in $\epsilon$ with $\tfrac{\partial x}{\partial\epsilon}\in B_{T,U}$ and the partial derivative satisfies the Volterra integral 
		equation
		\[
			\frac{\partial x}{\partial\epsilon}(t,\epsilon)=
			\left(\frac{\partial a}{\partial\epsilon}(\epsilon,t)-\int_t^{\infty}\frac{\partial K}{\partial\epsilon}(t,s,\epsilon)x(\epsilon,s)ds\right)
			-\int_t^{\infty}K(t,s,\epsilon)\frac{\partial x}{\partial\epsilon}(s,\epsilon) ds
		\]
		for every $\epsilon\in U$.
	\end{theorem}

	\begin{proof}
		With the notions introduced above it is now possible for the proof to follow the general idea of \cite{Windsor.2010}. Still the non-compact integration range demands a lot of additional estimates now.The parameter dependent solution $x(t,\epsilon)$ of the integral equation exists and depends continuously on $\epsilon$ as of corollary \ref{cont_b} 
		and the preceding two lemmas. In order to show the differentiability we first construct a candidate for the derivative and then show, that this 
		candidate is indeed the derivative.\\
		Formally differentiating the integral equation with respect to $\epsilon$ yields
		\begin{equation}\label{diff_int}
		\begin{split}
			\frac{\partial x}{\partial\epsilon}(t,\epsilon)&
			=\frac{\partial a}{\partial\epsilon}(\epsilon,t)-\int_t^{\infty}\left(\frac{\partial K}{\partial\epsilon}(t,s,\epsilon)x(\epsilon,s)
			-K(t,s,\epsilon)\frac{\partial x}{\partial\epsilon}(s,\epsilon)\right)ds\\		
			&=\left(\frac{\partial a}{\partial\epsilon}(\epsilon,t)-\int_t^{\infty}\frac{\partial K}{\partial\epsilon}(t,s,\epsilon)x(\epsilon,s)ds\right)
			-\int_t^{\infty}K(t,s,\epsilon)\frac{\partial x}{\partial\epsilon}(s,\epsilon)ds.
		\end{split}	
		\end{equation}
		As of our assumptions the term
		\[
			\frac{\partial a}{\partial\epsilon}(\epsilon,t)-\int_t^{\infty}\frac{\partial K}{\partial\epsilon}(t,s,\epsilon)x(\epsilon,s)ds
		\]
		is continuous in $t$ (for every $\epsilon\in U$) and bounded. Therefore theorem $\ref{ex}$ tells us, that there is a unique solution of equation 
		\ref{diff_int}, which shall be denoted by $\tfrac{\widehat{\partial x}}{\partial \epsilon}$.\\
		Note that $\tfrac{\widehat{\partial x}}{\partial \epsilon}$ only formally satisfies the equation for the derivative. In fact, in formulating equation 
		\ref{diff_int} we pulled the differential under the integral which entails the assumption that the solution is still continuous in $\epsilon$, which we 
		can't know yet. So instead we have to consider only the final equation as of itself and now show, that the solution is in fact the derivative. For this 
		consider the difference quotient
		\[
			q(\epsilon,\Delta\epsilon,t):=\frac{x(\epsilon+\Delta\epsilon)-x(\epsilon,t)}{\Delta\epsilon}.
		\]
		It satisfies the following integral equation
		\[
			q(\epsilon,\Delta\epsilon,t)=\frac{a(t,\epsilon+\Delta\epsilon)-a(t,\epsilon)}{\Delta\epsilon}
			-\int_t^{\infty}\frac{K(t,s,\epsilon+\Delta\epsilon)x(s,\epsilon+\Delta\epsilon)-K(t,s,\epsilon)x(s,\epsilon)}{\Delta\epsilon}ds.
		\]
		Now we reformulate the integrand to obtain an expression involving the difference quotient:
		\[
			\begin{split}
				&\frac{K(t,s,\epsilon+\Delta\epsilon)x(s,\epsilon+\Delta\epsilon)-K(t,s,\epsilon)x(s,\epsilon)}{\Delta\epsilon}\\
				&=\frac{\left(K(t,s,\epsilon+\Delta\epsilon)-K(t,s,\epsilon)\right)x(s,\epsilon+\Delta\epsilon)+K(t,s,\epsilon)(x(s,\epsilon+\Delta\epsilon)-x(s,\epsilon))}{\Delta\epsilon}\\
				&=\frac{K(t,s,\epsilon+\Delta\epsilon)-K(t,s,\epsilon)}{\Delta\epsilon}x(s,\epsilon+\Delta\epsilon)
				+K(t,s,\epsilon)\frac{x(s,\epsilon+\Delta\epsilon)-x(s,\epsilon)}{\Delta\epsilon}\\
				&=\frac{K(t,s,\epsilon+\Delta\epsilon)-K(t,s,\epsilon)}{\Delta\epsilon}x(s,\epsilon+\Delta\epsilon)
				+K(t,s,\epsilon)q(\epsilon,\Delta\epsilon,s).
			\end{split}
		\]
		With this, we can now look at the difference $\tfrac{\widehat{\partial x}}{\partial \epsilon}-q(\epsilon,\Delta\epsilon,t)$ and obtain another integral 
		equation:
		\[
			\begin{split}
				\frac{\widehat{\partial x}}{\partial \epsilon}(t,\epsilon)-q(\epsilon,\Delta\epsilon,t)
				=&\frac{\partial a}{\partial\epsilon}(\epsilon,t)-\frac{a(t,\epsilon+\Delta\epsilon)-a(t,\epsilon)}{\Delta\epsilon}\\
				&-\int_t^{\infty}\frac{\partial K}{\partial\epsilon}(t,s,\epsilon)x(s,\epsilon)
				-\frac{K(t,s,\epsilon+\Delta\epsilon)-K(t,s,\epsilon)}{\Delta\epsilon}x(s,\epsilon+\Delta\epsilon)ds\\
			&-\int_t^{\infty}K(t,s,\epsilon)\left(\frac{\widehat{\partial x}}{\partial\epsilon}(s,\epsilon)-q(\epsilon,\Delta\epsilon,s)\right)ds\\
			&=:\widehat{a}(t,\epsilon,\Delta\epsilon)-\int_t^{\infty}K(t,s,\epsilon)\left(\frac{\widehat{\partial x}}{\partial\epsilon}(s,\epsilon)-q(\epsilon,\Delta\epsilon,s)\right)ds
			\end{split}
		\]
		Again the first term defines a bounded continuous function on $[T,\infty)$, so we know that 
		$\tfrac{\widehat{\partial x}}{\partial \epsilon}-q(\epsilon,\Delta\epsilon,t)$ is the unique solution to this integral equation. This is helpful, as we 
		can now use the estimate from the contraction mapping principle to obtain an estimate for the difference $\tfrac{\widehat{\partial x}}{\partial \epsilon}-q(\epsilon,\Delta\epsilon,t)$. In fact in our situation corollary \ref{est} tells us
		\[
			\left\|\frac{\widehat{\partial x}}{\partial \epsilon}-q(\epsilon,\Delta\epsilon,t)\right\|_{\infty}
			\leq\frac{\left\|\widehat{a}\right\|_{\infty}}{1-\lambda}.
		\]
		Note that $\lambda=\sup_{t\in[T,\infty)}\int_t^{\infty}\left\|K(t,s,\epsilon)\right\|ds$ only depends on $\epsilon$ and not on $\Delta\epsilon$, so it 
		remains the same as we now want to consider the limit $\Delta\epsilon\rightarrow0$.\\
		The Taylor expansion of $a$ at $\epsilon$ up to first order is
		\[
			a(t,\epsilon+\Delta\epsilon)=a(t,\epsilon)+\frac{\partial a}{\partial\epsilon}(t,\epsilon)\Delta\epsilon+R(t,\Delta\epsilon),
		\]
		for the remainder $R$ that we may estimate in Lagrange form via
		\[
			\left\|R(t,\Delta\epsilon)\right\|\leq\frac{1}{2}\left\|\frac{\partial^2a}{\partial\epsilon^2}(t,\mu)(\Delta\epsilon)^2\right\|,
		\]
		for some $\mu\in[\epsilon,\epsilon+\Delta\epsilon]$. Thus the norm of the first part of the "boundary value" of the last integral equation can be 
		estimated as
		\[
		\begin{split}
			&\left\|\frac{\partial a}{\partial\epsilon}(t,\epsilon)-\frac{a(t,\epsilon+\Delta\epsilon)-a(t,\epsilon)}{\Delta\epsilon}\right\|_{\infty}
			\leq\left\|\frac{R(t,\Delta\epsilon)}{\Delta\epsilon}\right\|\\
			&\leq\frac{1}{2}\left\|\frac{\partial^2a}{\partial\epsilon^2}(t,\mu)\right\||\Delta\epsilon|.
		\end{split}
		\]
		As we assumed our $a$ to have a bounded second derivative we thus see, that the norm vanishes in the $\Delta\epsilon\rightarrow0$ limit.
		
		We would like to do the same for the integrand part, which basically works, but we have to put more work into this part, as the $\Delta\epsilon$ 
		limit , the sup-norm and the integral generally don't commute. It is mainly for this reason that we introduced the product structure for $K$ as this 
		gives us the desired uniform behavior. So now we write our kernel $K$ as $K(t,s,\epsilon)=A(t,s,\epsilon)\cdot B(s,\epsilon)$. This leads to
		\[
			\begin{split}
			&\frac{\partial K}{\partial\epsilon}(t,s,\epsilon)x(\epsilon,s)
				-\frac{K(t,s,\epsilon+\Delta\epsilon)-K(t,s,\epsilon)}{\Delta\epsilon}x(s,\epsilon+\Delta\epsilon)\\
				&=\left(\frac{\partial A}{\partial\epsilon}(t,s,\epsilon)B(s,\epsilon)+A(t,s,\epsilon)\frac{\partial B}{\partial\epsilon}(s,\epsilon)\right)x(s,\epsilon)\\
				&-\left(\frac{A(t,s,\epsilon+\Delta\epsilon)B(s,\epsilon+\Delta\epsilon)-A(t,s,\epsilon)B(s,\epsilon)}{\Delta\epsilon}\right)x(s,\epsilon+\Delta\epsilon)\\
				&=\left(\frac{\partial A}{\partial\epsilon}(t,s,\epsilon)B(s,\epsilon)+A(t,s,\epsilon)\frac{\partial B}{\partial\epsilon}(s,\epsilon)\right)x(s,\epsilon)\\
				&-\left(\frac{A(t,s,\epsilon+\Delta\epsilon)-A(t,s,\epsilon)}{\Delta\epsilon}B(s,\epsilon+\Delta\epsilon)
				+A(t,s,\epsilon)\frac{B(s,\epsilon+\Delta\epsilon)-B(s,\epsilon)}{\Delta\epsilon}\right)x(s,\epsilon+\Delta\epsilon)\\
				&=\frac{\partial A}{\partial\epsilon}(t,s,\epsilon)B(s,\epsilon)x(s,\epsilon)
				-\frac{A(t,s,\epsilon+\Delta\epsilon)-A(t,s,\epsilon)}{\Delta\epsilon}B(s,\epsilon+\Delta\epsilon)x(s,\epsilon+\Delta\epsilon)\\
				&+A(t,s,\epsilon)\frac{\partial B}{\partial\epsilon}(s,\epsilon)x(s,\epsilon)
				-A(t,s,\epsilon)\frac{B(s,\epsilon+\Delta\epsilon)-B(s,\epsilon)}{\Delta\epsilon}x(s,\epsilon+\Delta\epsilon).
			\end{split}
		\]
		Let us consider the first half of the final expression above. We write it as
		\[
		\begin{split}
			&\frac{\partial A}{\partial\epsilon}(t,s,\epsilon)B(s,\epsilon)x(s,\epsilon)
				-\frac{\partial A}{\partial\epsilon}(t,s,\epsilon)B(s,\epsilon+\Delta\epsilon)x(s,\epsilon+\Delta\epsilon)\\
				&+\frac{\partial A}{\partial\epsilon}(t,s,\epsilon)B(s,\epsilon+\Delta\epsilon)x(s,\epsilon+\Delta\epsilon)
				-\frac{A(t,\epsilon+\Delta\epsilon)-A(t,s,\epsilon)}{\Delta\epsilon}B(s,\epsilon+\Delta\epsilon)x(s,\epsilon+\Delta\epsilon)\\
				&=\frac{\partial A}{\partial\epsilon}(t,s,\epsilon)\left(B(s,\epsilon)x(s,\epsilon)-B(s,\epsilon+\Delta\epsilon)x(s,\epsilon+\Delta\epsilon)\right)\\
				&+\left(\frac{\partial A}{\partial\epsilon}(t,s,\epsilon)
				-\frac{A(t,s,\epsilon+\Delta\epsilon)-A(t,s,\epsilon)}{\Delta\epsilon}\right)B(s,\epsilon+\Delta\epsilon)x(s,\epsilon+\Delta\epsilon)
		\end{split}
		\]
		Now for the first part of this expression, we know that $B(s,\epsilon+\Delta\epsilon)x(s,\epsilon+\Delta\epsilon)$ converges point-wise to 
		$B(s,\epsilon)x(s,\epsilon)$ as $\Delta\epsilon\rightarrow0$. As of our assumption on $B$ and the boundedness of $x$ there exists some function 
		$b:[T,\infty)\rightarrow\mathbb{R}_{\geq0}$ with $\int_T^{\infty}b(s)ds<\infty$ s. t. 
		$\left\|B(s,\epsilon+\Delta\epsilon)x(s,\epsilon+\Delta\epsilon)\right\|\leq b(s)$ for all $\Delta\epsilon$ close enough to $0$. Therefore, by the 
		Dominated convergence theorem (for Bochner integrable functions) it holds
		\[
			\lim_{\Delta\epsilon\rightarrow0}\int_T^{\infty}\left\|B(s,\epsilon)x(s,\epsilon)-B(s,\epsilon+\Delta\epsilon)x(s,\epsilon+\Delta\epsilon)\right|ds=0.
		\]
		Thus, by taking into account that there is an upper bound $\beta\in\mathbb{R}$ for $\tfrac{\partial A}{\partial\epsilon}(t,s,\epsilon)$ we get for the 
		sup-norm of the whole first term:
		\[
		\begin{split}
			&\lim_{\Delta\epsilon\rightarrow0}\left\|\int_t^{\infty}\frac{\partial A}{\partial\epsilon}(t,s,\epsilon)\left(B(s,\epsilon)x(s,\epsilon)-B(s,\epsilon+\Delta\epsilon)x(s,\epsilon+\Delta\epsilon)\right)ds\right\|_{\infty}\\
			&\leq\beta\lim_{\Delta\epsilon\rightarrow0}\int_T^{\infty}\left\|B(s,\epsilon)x(s,\epsilon)-B(s,\epsilon+\Delta\epsilon)x(s,\epsilon+\Delta\epsilon)\right\|ds=0
		\end{split}
		\]
		It is for this inequality (and for following ones of the same kind) that we assumed the product structure of our kernel $K$, as it gives us the desired 
		uniform behavior.
		
		Next up is the term
		\[
			\int_t^{\infty}\left(\frac{\partial A}{\partial\epsilon}(t,s,\epsilon)-\frac{A(t,s,\epsilon+\Delta\epsilon)-A(t,s,\epsilon)}{\Delta\epsilon}\right)
			B(s,\epsilon+\Delta\epsilon)x(s,\epsilon+\Delta\epsilon).
		\]
		The $x$ part is now bounded by some constant $C$ for all $\Delta\epsilon$ small enough. For the difference quotient we consider again the Taylor 
		approximation of $A(t,s,\epsilon)$ with respect to $\epsilon$ (i. e. $t$ and $s$ are fixed):
		\[
			A(t,s,\epsilon+\Delta\epsilon)=A(t,s,\epsilon)+\frac{\partial A}{\partial\epsilon}(t,s,\epsilon)\Delta\epsilon+R(t,s,\Delta\epsilon),
		\]
		with the remainder term $R$. Thus the $A$ part in the integrand can also be regarded as
		\[
			-\frac{R(t,s,\Delta\epsilon)}{\Delta\epsilon}.
		\]
		For the remainder we again use the Lagrange estimate, i. e.:
		\[
			\left\|R(t,s,\Delta\epsilon)\right\|\leq\frac{1}{2}\left\|\frac{\partial^2A}{\partial\epsilon^2}(t,s,\mu)\left(\Delta\epsilon\right)^2\right\|
		\]
		for some $\mu\in[\epsilon,\epsilon+\Delta\epsilon]$. Now we us the assumption that in a neighbourhood of $\epsilon$ the second derivative of $A$ is 
		still bounded by some $\gamma$, so we obtain for $\Delta\epsilon$ small enough:
		\[
			\begin{split}
			&\left\|\int_t^{\infty}\left(\frac{\partial A}{\partial\epsilon}(t,s,\epsilon)-\frac{A(t,s,\epsilon+\Delta\epsilon)-A(t,s,\epsilon)}{\Delta\epsilon}\right)
			B(s,\epsilon+\Delta\epsilon)x(s,\epsilon+\Delta\epsilon)\right\|_{\infty}ds\\
			&\leq\sup_{t\in[T,\infty)}\int_t^{\infty}\left\|\left(\frac{\partial A}{\partial\epsilon}(t,s,\epsilon)-\frac{A(t,s,\epsilon+\Delta\epsilon)-A(t,s,\epsilon)}{\Delta\epsilon}\right)B(s,\epsilon+\Delta\epsilon)x(s,\epsilon+\Delta\epsilon)\right\|ds\\
			&\leq \frac{C}{2}\sup_{t\in[T,\infty)}\int_t^{\infty}\left\|\frac{\partial^2A}{\partial\epsilon^2}(t,s,\mu)\right\|\left\|B(s,\epsilon+\Delta\epsilon)\right\|ds\left(\Delta\epsilon\right)
			\leq \frac{C}{2}\gamma\int_T^{\infty}\left\|B(s,\epsilon+\Delta\epsilon)\right\|ds\left(\Delta\epsilon\right)\\
			&\leq \frac{C\gamma}{2}\int_T^{\infty}\left\|B(s,\epsilon+\Delta\epsilon)\right\|ds\left(\Delta\epsilon\right)
			\leq\frac{C\gamma}{2}\int_T^{\infty}g(s)ds\left(\Delta\epsilon\right).
			\end{split}
		\]
		As $g$ is integrable this term also vanishes in the $\Delta\epsilon\rightarrow0$ limit.\\
		Finally there is the last part 
		\[
		\int_t^{\infty}A(t,s,\epsilon)\left(\frac{\partial B}{\partial\epsilon}(s,\epsilon)x(s,\epsilon)
				-\frac{B(s,\epsilon+\Delta\epsilon)-B(s,\epsilon)}{\Delta\epsilon}x(s,\epsilon+\Delta\epsilon)\right)ds.
		\]
		Here we use a combination of the arguments above. Firstly $A$ is bounded by some $\alpha\in\mathbb{R}$ for all $(t,s)\in\Delta(T)$. From the mean value 
		estimate we obtain
		\[
			\left\|B(s,\epsilon+\Delta\epsilon)-B(s,\epsilon)\right\|\leq|\Delta\epsilon|\sup_{0\leq\tau\leq1}\left\|\frac{\partial B}{\partial\epsilon}(s,\epsilon+\tau\Delta\epsilon)\right\|.
		\]
		We assumed $\partial_\epsilon B$ to be integrally bounded, so there is some integrable function $h$ which gives a pointwise upper bound for all 
		$\Delta\epsilon$ small enough. As $x$ is also bounded the whole right side in the integrand is bounded by some general integrable function. As the 
		right side also converges pointwise to the left side, we obtain by the dominant convergence theorem that the integral vanishes in the 
		$\Delta\epsilon\rightarrow0$ limit, and thus
		\[
			\begin{split}
		&\lim_{\Delta\epsilon\rightarrow0}\left\|\int_t^{\infty}A(t,s,\epsilon)\left(\frac{\partial B}{\partial\epsilon}(s,\epsilon)x(s,\epsilon)
				-\frac{B(s,\epsilon+\Delta\epsilon)-B(s,\epsilon)}{\Delta\epsilon}x(s,\epsilon+\Delta\epsilon)\right)ds\right\|_\infty\\
				&\leq\lim_{\Delta\epsilon\rightarrow0}\alpha\sup_{t\in[T,\infty)}
				\int_t^{\infty}\left\|\left(\frac{\partial B}{\partial\epsilon}(s,\epsilon)x(s,\epsilon)
				-\frac{B(s,\epsilon+\Delta\epsilon)-B(s,\epsilon)}{\Delta\epsilon}x(s,\epsilon+\Delta\epsilon)\right)\right\|ds\\
				&\leq\alpha\lim_{\Delta\epsilon\rightarrow0}\int_T^{\infty}\left\|\left(\frac{\partial B}{\partial\epsilon}(s,\epsilon)x(s,\epsilon)
				-\frac{B(s,\epsilon+\Delta\epsilon)-B(s,\epsilon)}{\Delta\epsilon}x(s,\epsilon+\Delta\epsilon)\right)\right\|ds=0.
			\end{split}
		\]
		We have thus shown that all terms of the kernel vanish and we obtain
		\[
		\lim_{\Delta\epsilon\rightarrow0}\left\|\int_t^{\infty}\frac{\partial K}{\partial\epsilon}(t,s,\epsilon)x(\epsilon,s)
				-\frac{K(t,s,\epsilon+\Delta\epsilon)-K(t,s,\epsilon)}{\Delta\epsilon}x(s,\epsilon+\Delta\epsilon)ds\right\|_\infty=0.
		\]
		And with this \ref{est} tells us
		\[
			\lim_{\Delta\epsilon\rightarrow\infty}\left\|\frac{\widehat{\partial x}}{\partial \epsilon}-q(\epsilon,\Delta\epsilon,t)\right\|_{\infty}
			\leq\frac{\lim_{\Delta\epsilon\rightarrow\infty}\left\|\widehat{a}\right\|_{\infty}}{1-\lambda}=0.
		\]
		So the difference quotient converges towards our candidate for the derivative $\tfrac{\widehat{\partial x}}{\partial \epsilon}$ in the sup-norm, and 
		thus also point-wise. So it is in fact the true derivative. As it is a solution of the aforementioned integral equation \eqref{diff_int} we also know 
		that it is bounded and continuous in $t$. As $x$ is bounded and $\partial_\epsilon K$ is of product type and integrally bounded it is also continuous 
		in $\epsilon$.
	\end{proof}
	This finishes the proof for the continuous differentiability in $\epsilon$ and one may expect that corresponding propositions for higher derivatives 
	hold, given the corresponding differentiability of the ingredients $a$ and $K$. In fact Windsor gives a proof by induction for his setting, and the 
	arguments should translate to our set up, but as they are not needed for our applications in the second part and don't promise more important insights, 
	they are not pursued at this point.
	
	We will, however, need the extension of the arguments for holomorphic dependence. For this we first note, that the notion of kernels of product type 
	readily translates to the notion of holomorphic product types $K\in\mathcal{K}^h_{T,U}$. We don't write them here explicitly but only note that we do not 
	need to distinguish 
	product types of order greater than $1$. Instead by demanding the parts of the kernel to be complex differentiable we already have that they are 
	infinitely often complex differentiable and thanks to Cauchy's differentiation formula the upper bound and integral bound for the kernel are also 
	corresponding bounds for the derivatives. Thus the notion of order is not necessary in the holomorphic setting.
	
	\begin{theorem}\label{hol_dep}
		Let $a\in B_{T,U}$ be bounded and holomorphic in $\epsilon\in U$ where $U$ is some open subset of $\mathbb{C}$. Let $K\in \mathcal{K}^h_{T,U}$ be of 
		holomorphic product type.
		Then the solution $x(t,\epsilon)$ of the integral equation
		\[
			x(t,\epsilon)=a(t,\epsilon)-\int_t^{\infty}K(t,s,\epsilon)x(s,\epsilon)ds
		\]
		is holomorphic in $\epsilon$ with $\tfrac{\partial x}{\partial\epsilon}\in B_{T,U}$ and the partial derivative satisfies the Volterra integral equation
		\[
			\frac{\partial x}{\partial\epsilon}(t,\epsilon)=
			\left(\frac{\partial a}{\partial\epsilon}(\epsilon,t)-\int_t^{\infty}\frac{\partial K}{\partial\epsilon}(t,s,\epsilon)x(\epsilon,s)ds\right)
			-\int_t^{\infty}K(t,s,\epsilon)\frac{\partial x}{\partial\epsilon}(s,\epsilon) ds
		\]
		for every $\epsilon\in U$.
	\end{theorem}
	
	\begin{proof}
		For the proof we simply have to retrace all the steps in section \ref{sec_diff_dep}. We don't do this in detail and just note that the two main 
		estimates that were used in the proof of \ref{different} also hold in the holomorphic setting. That is firstly the mean value inequality, for which a 
		corresponding estimate in the holomorphic setting is standard. The other estimate concerns the use of the Lagrangian remainder in the Taylor expansion. 
		Similar to the explanation above this is not needed now. Instead the remainder term for all relevant parts is bounded by the some multiple of the 
		function itself thanks to Cauchy's differentiation formula. Thus all relevant estimates are given and the proof can be recast in an obvious manner.
	\end{proof}
	
	\newpage
	
\section{Construction of the coordinate system on \texorpdfstring{$\mathcal{M}$}{M}}\label{const_on_mod}

	We now come to the main part of this endeavor. While the previous section was mostly concerned with a theory that can stand on its own we now want 
	to use it for the construction of hyperkähler metrics on the moduli space of weakly parabolic Higgs bundles. Here most of the hard work is done as we use 
	the theory developed in section \ref{init_prob} to obtain the correct decorations and deduce their asymptotic behaviour. This is first done for the small 
	flat sections in subsection \ref{sec_small_sec} and then for the canonical coordinates in subsection \ref{sec_coord}. Finally in subsection 
	\ref{der_coord} the logarithmic derivatives of the coordinates are studied as they will provide us with the holomorphic symplectic form from which the 
	hyperkähler metric is obtained in the last section.

\subsection{Obtaining the small flat sections}\label{sec_small_sec}

	As was already hinted at at the beginning of section \ref{sec_dec} we actually construct coordinates on the moduli space of flat 
	connections $\mathcal{M}_{\text{flat}}$ which we identify with $\mathcal{M}$ in the following way. We start with a stable parabolic Higgs bundle $(E,\overline{\partial}_E,\Phi)$ of parabolic degree pdeg$(\mathcal{E})=0$ and a solution $h$ of the $R$-rescaled Hitchin's equation to obtain a pair
	$(\Phi,A)$ consisting of the Higgs field $\Phi$ and the Chern-connection $A=A(\overline{\partial}_E)$ on the Hermitian vector bundle $(E,h)$. Additionally we now take any 
	$\zeta\in\mathbb{C}^\times$ and construct the connection
	\[
		\nabla(\zeta):=\frac{R}{\zeta}\Phi+d_A+R\zeta\Phi^{\ast_h}.
	\]
	The $\zeta$-dependence will later on lead to the objects of interest being twisted by the line bundle $\mathcal{O}(2)$ in accordance with the twistor 
	theorem. Therefore this connection will be referred to as \textit{twistor}-connection here. The fact that $(\Phi,A)$ is a solution to Hitchin's equations 
	and pdeg$(\mathcal{E})=0$ implies that this connection is flat. Note that the flatness of $\nabla(\zeta)$ is of quite some importance in the following arguments, as we are interested in a unique way of parallel transporting sections in $E$. For different parabolic degree one would obtain projectively flat connections and it seems reasonable that the theory can be adapted accordingly. But as we also used this restriction on the degree above in section \ref{frame_sec} when we established the background metric on the Higgs bundles, we will only consider strictly flat case here.
	
	In the rest of this section want to study the decorations that belong to this family of connections. So in light of the construction in section 
	\ref{Darboux_coords} we have to study the $\nabla$-flat sections of $E$ and choose one for each vertex $p$ of any $ca$-triangulation. So we look at the 
	solutions $s$ of the flatness equation
		\[
			\nabla(\zeta)s=0.
		\]
	We first briefly describe how GMN defined the decorations and which problems arise in their setting, so that we can then state the goal our treatment 
	of the subject. For this consider a local patch around any of the poles such that the pole has coordinate 
	$z=0$. GMN consider Higgs pairs $(\Phi,A)$ of the form 
	\begin{equation}\label{GMN_way}
		\begin{split}
			\Phi&=m\sigma^3\frac{dz}{z}+\text{regular},\\
			A&=m^{(3)}\sigma^3\left(\frac{dz}{z}-\frac{d\overline{z}}{\overline{z}}\right)+\text{regular},
		\end{split}
	\end{equation}
	for $m\in\mathbb{C}$ and $m^{(3)}\in\mathbb{R}$. Note that $m^{(3)}$ corresponds to a choice of the parabolic weights as $\alpha_1=2m^{(3)}=-\alpha_2$, 
	while $m$ equals the residue of the Higgs field at the pole of at $z=0$.
	
	Then they argue that a Frobenius analysis near a regular singular point shows that there are two flat sections of $\nabla(\zeta)$ of the form
	\[
		\begin{split}
			s^{(1)}=z^{-R\zeta^{-1}m+m^{(3)}}\overline{z}^{-R\zeta\overline{m}+m^{(3)}}\begin{pmatrix}1+O(\left|z\right|\\O(\left|z\right|)\end{pmatrix},\\
			s^{(2)}=z^{R\zeta^{-1}m-m^{(3)}}\overline{z}^{R\zeta\overline{m}+m^{(3)}}\begin{pmatrix}O(\left|z\right|)\\1+O(\left|z\right|)\end{pmatrix}.
		\end{split}
	\]
	These have clockwise monodromy eigenvalues $\mu^{(1)}=e^{2\pi i\nu}$, $\mu^{(2)}=e^{-2\pi i\nu}$ with $\nu=R\zeta^{-1}m-2m^{(3)}-R\zeta\overline{m}$. To 
	single out one of the solutions thy consider these solutions along $ca$-trajectories and ask whether they decay or grow when running into the pole. Thus 
	we now fix some generic angle $\vartheta$. If not otherwise noted all angles here and in the following will always be generic. It turns out that for the 
	following constructions to work the choice of $\vartheta$ and $\zeta$ have to be tied together to gain control over how a $ca$-trajectory behaves. This is 
	done by introducing the half-plane centered around $e^{i\vartheta}\mathbb{R}_+$:
	\[
	\mathbb{H}_{\vartheta}:=\left\{\zeta\in\mathbb{C}^\times:\enspace \vartheta-\frac{\pi}{2}<\text{arg}\zeta<\vartheta+\frac{\pi}{2}\right\}.
	\]
	We will now always take $\zeta\in\mathbb{H}_{\vartheta}$. Then along any $\vartheta$-trajectory with $\zeta\in\mathbb{H}_{\vartheta}$ one of these 
	sections is exponentially decaying while the other one is exponentially growing. The space of exponentially 
	decaying sections is thus a $1$-dimensional subspace of solutions and seems to be well equipped to be the decoration for the parabolic point. 
	
	Following GMN one could now define the \textit{small flat section} at a weakly parabolic point $p$ as the (up to complex rescaling) unique 
	$\nabla(\zeta)$-flat section $s$ of $E$ that decays exponentially when following a $\vartheta$-trajectory into $p$. We will not do this here, as there 
	are some problems with this definition. A minor detail is that we do not work with a traceless connection matrix, so the argument would have to be 
	adjusted to this. Though this might be possible another problem is that from the set-up of \cite{Gaiotto.2013} it is a priori not clear, that the 
	remainder terms in the expansion of the Higgs field and Chern connection are regular enough for the asymptotic analysis to hold rigorously for all Higgs 
	pairs. One could try 
	to argue for this using the results from the previous section, but the most important difficulty would remain even if such an analysis were possible, as 
	it is a statement concerning only small neighborhoods of the parabolic points.
	Even if one were able to define these sections in this way it is not clear at all at this point that this amounts to a decoration for which the 
	construction of the 
	Fock-Goncharov coordinate system via Theorem \ref{FG_coord} works, as we do not know how these sections propagate throughout the rest of $\mathcal{C}$. 
	For example it is important that such a section that is exponentially decaying when running into some weakly parabolic point $p_1$ along a 
	$\vartheta$-trajectory is not also decaying when running in the opposite direction into another weakly parabolic point $p_2$. Only if this can be ruled out it is 
	clear that the $\mathcal{X}$-coordinates are well defined and no linear dependence appears in the $\wedge$-products. Even then, to construct the 
	$\wedge$-
	product of two such decorations we have to evaluate them at some common points on $\mathcal{C}$ for which we also need to know how the sections develop 
	away from weakly parabolic points where they are singled out. Finally we need to know the asymptotics of the sections for 
	$\zeta\to0$, $\zeta\to\infty$ and we would like to have the asymptotics for $R\to\infty$. GMN argued that all of these aspects are as desired mainly by 
	using an argument for the asymptotics build on the WKB approximation. The main goal here is to give a rigorous proof of these assertions via the limiting 
	configurations of \cite{Fredrickson.2020} and the theory of initial value problems at infinity. 
	
	With these goals in mind we now proceed to recast the problem in the formalism of section \ref{init_prob}. For the rest of this section a generic angle 
	$\vartheta$ for $-q=\det\Phi$ shall be fixed. We now also fix a standard annulus $S$ for a standard 
	quadrilateral $Q$ and corresponding frame $F$ as working environment for all our calculations of this section. Locally the flatness equation 
	for $s$ is given as
	\begin{equation}\label{flatness_equation}
		0=\nabla s=\left(d+\frac{R}{\zeta}\varphi dz+A_zdz+A_{\overline{z}}d\overline{z}+R\zeta\varphi^{\ast_h}d\overline{z}\right)s
	\end{equation}
	The form of $\nabla$ implies that the $dz$ and $d\overline{z}$ part of $ds$ have to become zero independently, so we obtain the two PDEs:
	\[
		\partial_zs=\left(-\frac{R}{\zeta}\varphi-A_z\right)s\quad\wedge\quad\partial_{\overline{z}}s=\left(-R\zeta\varphi^{\ast_h}-A_{\overline{z}}\right)s.
	\]
	Now let $\gamma(t)$ be a $\vartheta$-trajectory for $q^{1/2}$ and some generic angle $\vartheta\in\mathbb{R}/2\pi\mathbb{Z}$ that runs between two poles 
	$p_1$ and $p_2$ of $q$ and thus forms one of the boundary curves of the quadrilateral $Q$. We also fix the parametrization for the 
	following calculations such that $q^{1/2}(\gamma')=f^{1/2}(z(\gamma(t)))z'(\gamma(t))=-e^{i\vartheta}$. Note that we have thus fixed an orientation for 
	the trajectory 
	which determines whether the trajectory runs from $p_1$ to $p_2$ for increasing $t$ or in the other direction.
	
	We take this curve to be extend to all of $\mathbb{R}$ as of proposition \ref{infinite_range}. Thus trivialized the section
	$s$ can be regarded as a function $s:\mathbb{R}\rightarrow\mathbb{C}^2$. For notational simplicity let $z(t)$ denote $z(\gamma(t))$. Locally in $S$ along 
	the $\vartheta$-trajectory $\gamma(t)$ the flat section $s(t)=s(z(t))$ then satisfies the following ODE-form of the flatness equation:
	\[
		\frac{ds}{dt}(t)=\frac{\partial s}{\partial z}(z(t))\frac{dz}{dt}(t)+\frac{\partial s}{\partial\overline{z}}(\overline{z}(t))\frac{d\overline{z}}{dt}(t)
		=\left(\left(-\frac{R}{\zeta}\varphi-A_z\right)z'+\left(-R\zeta\varphi^{\ast_h}-A_{\overline{z}}\right)\overline{z}'\right)s=:B(t)s(t).
	\]
	Note that this trivialized $s$ depends on the chosen gauge. This dependency vanishes in the final expression $\mathcal{X}$, as the gauge 
	transformations cancel each other in the defining formula.
	
	To solve the ODE we start by rewriting $B$ in the acquired form of leading term plus remainder from section \ref{local_desc} in any contractible open 
	subset $U$ of $S$. For the leading part we define 
	\begin{equation}\label{leading}
		\begin{split}
		\lambda_1&:=\tfrac{R}{\zeta}e^{i\vartheta}-\overline{a_1}\overline{z'}+a_1z'+R\zeta e^{-i\vartheta}
		=\tfrac{R}{\zeta}e^{i\vartheta}-2i\text{Im}(a_1z')+R\zeta e^{-i\vartheta},\\
		\lambda_2&:=\tfrac{R}{\zeta}e^{i\vartheta}+\overline{a_2}\overline{z'}-a_2z'+R\zeta e^{-i\vartheta}
		=\tfrac{R}{\zeta}e^{i\vartheta}+2i\text{Im}(a_2z')+R\zeta e^{-i\vartheta}.
		\end{split}
	\end{equation}
	Note that $a_1$ and $a_2$ have simple poles at the weakly parabolic points, so by multiplying with $z'$ the whole expression for the leading part becomes regular. More details on this follow shortly in the next proof.
	For the error part we define
	\[
		\begin{split}
		e_1:&=(a_1-a_2)|\beta|^2z'+R\zeta2e^{-i\vartheta}|\beta|^2+h_1z'=\left((a_2-a_1)z'+R\zeta2e^{-i\vartheta}\right)|\beta|^2+h_1z',\\
		e_2:&=(a_1-a_2)\delta\beta z'+R\zeta2e^{-i\vartheta}\delta\beta+h_2z'=\left((a_2-a_1)z'+R\zeta2e^{-i\vartheta}\right)\delta\beta+h_2z',\\
		e_3:&=-(a_1-a_2)\alpha\overline{\beta}z'-R\zeta2e^{-i\vartheta}\alpha\overline{\beta}+h_3z'
		=\left(-(a_1-a_2)z'-R\zeta2e^{-i\vartheta}\right)\alpha\overline{\beta}+h_3z'.
		\end{split}
	\]
	Note that the error part only has a linear dependence on $\zeta$, while $\zeta^{-1}$ only exists in the leading term. This is a consequence of the choice of holomorphic frame in section \ref{frame_sec} and important for us, as we will consider the terms for small $\zeta$ which is in this form possible without an unbounded increase of the error part.
	
	As a corollary of Lemma \ref{g_asymptotics} we obtain the asymptotics for the entries in the error-matrix:
	
	\begin{lemma}\label{e_asymptotics}
		There exists a local coordinate $z$ centered around any weakly parabolic point, such that along a generic $\vartheta$-trajectory the following 
		asymptotics hold for the entries of the error matrix defined 
		above (with $r:=|z|$) for some $\mu>0$:
		\[
			e_1,e_2,e_3=O(r^\mu).
		\]
		Furthermore, if $\epsilon$ denotes a $C^2$-coordinate of $\mathcal{M}$, then the derivatives w. r. t. $\epsilon$ exist and have the same asymptotics:
		\begin{align*}
			\partial_\epsilon e_1\partial_\epsilon e_2,\partial_\epsilon e_3&=O(r^\mu),\\
			\partial^2_\epsilon e_1\partial^2_\epsilon e_2,\partial^2_\epsilon e_3&=O(r^\mu).
		\end{align*}
		
		Additionally there exists $C,\delta>0$ s. t.
		\[
			|r^{-\mu}e_1|,|r^{-\mu}e_2|,|r^{-\mu}e_3|\leq Ce^{-\delta R}
		\]
		for large enough $R$.
	\end{lemma}
	
	\begin{proof}
		We start with a sufficiently small neighborhood centered around a weakly parabolic point $p$, s. t. Lemma \ref{g_asymptotics} holds, which tells us that the 
		factors involving $\alpha,\delta$ and $\beta$ give exactly the asymptotics as stated in the Lemma. So for the statement at hand to hold the terms in 
		front have to be bounded. This is clear for the part $R\zeta2e^{-i\vartheta}$, so we only have to check the part $|(a_1-a_2)z'|$.  
		
		From \cite{Fredrickson.2020} $a_1-a_2$ has a simple pole at $p$, so $a_1-a_2$ is of the form $\tfrac{\kappa}{z}+O(r^{\mu-1})$ for some constants $\kappa\in\mathbb{C}$ and $\mu>0$. From Lemma 
		\ref{log_spiral} we know that $z(t)=z_0e^{-\frac{e^{i\vartheta}}{m}t}$, so
		\[
			z'(t)=-\frac{e^{i\vartheta}}{m}z_0e^{-\frac{e^{i\vartheta}}{m}t}=-\frac{e^{i\vartheta}}{m}z(t).
		\]
		It is thus clear that $|(a_2-a_1)z'|$ is bounded along the $\vartheta$-trajectory. 
		
		Finally we have to check the $h_iz'$ parts. From Lemma \ref{g_asymptotics} we see that the $h_i$ are $O(r^{\mu-1})$, so multiplying with 
		$z'=-\frac{e^{i\vartheta}}{m}z$ results in $h_iz'=O(r^{\mu})$.
		
		This proves the first part, and the argument for the derivatives follows in the same way. Finally the exponential 
		suppression also follows directly from Lemma \ref{g_asymptotics}.
	\end{proof}
	
	Note that the last statement of the Lemma depends crucially on the fact that all entries of the error matrix are at least linear in the off diagonal term 
	$\beta$ of 
	the matrix $H$ that represents the metric $h$. This in turn leads to the at least linear dependence on the off diagonal term $b$ of the gauge transformation $g$, 
	whose decaying behaviour is thus the source of these estimates. We now obtain the following form for our flatness ODE:
	\[
		\frac{ds}{dt}=\begin{pmatrix}\lambda_1&0\\0&-\lambda_2\end{pmatrix}s+\begin{pmatrix}e_1&e_2\\e_3&-e_1\end{pmatrix}s.
	\]
	The important feature is that the second matrix is integrable (which we'll show below), while choosing $\zeta$ in $\mathbb{H}_{\vartheta}$ leads to 
	the upper left entry of the leading term having positive real part while the lower right entry has negative real part. The small flat section (for the 
	pole at $+\infty$) should now be given as the unique solution of the following integral equation:
	\[
		s(t)=\begin{pmatrix}0\\e^{-\Lambda_2(t)}\end{pmatrix}-\int_{t}^{\infty}
		\begin{pmatrix}e^{\Lambda_1(t)-\Lambda_1(\tau)}&0\\0&e^{-\Lambda_2(t)+\Lambda_2(\tau)}\end{pmatrix}
		\begin{pmatrix}e_1(\tau)&e_2(\tau)\\e_3(\tau)&-e_1(\tau)\end{pmatrix}s(\tau)d\tau.
	\]
	
	Here the $\Lambda_i$ are primitives of the $\lambda_i$ along the $ca$-curve. Note that from the structure of the $\lambda_i$ and the choice of $\zeta$ it follows that 
	$e^{-\Lambda_i(t)}\stackrel{t\rightarrow\infty}{\longrightarrow}0$ for both $i\in\left\{1,2\right\}$.
	
	This equation may already be regarded as an initial value problem at infinity. The limit for $t\to\infty$ of 
	the initial value function corresponds to the limit for $t\to\infty$ of the solution $s$, i. e. the limit of the section evaluated along a 
	$\vartheta$-trajectory going into a parabolic point. In fact solutions of this equation can already be constructed via the method of successive 
	approximation which was observed by N. Levinson in \cite{Levinson.1948}. This is also what I. Tulli did for the case of wild Higgs bundles on the sphere 
	in \cite{Tulli.2019}. Although that construction works and is useful to obtain good approximations of the solution it does not seem to be adequate 
	for questions regarding the structure of the solution space as well as uniqueness of solutions and differentiable dependence of a family of solutions. This is why we had to build the general theory for these kinds of equations in 
	part \ref{init_prob} of our work. 
	
	For our theory to work in this case however we have to modify the kernel of this equation. This is done by detaching the 
	decaying leading term, which corresponds to finding the leading term in the WKB approximation. So we write $s(t)=x(t)e^{-\Lambda_2(t)}$ and the integral 
	equation for $x$ takes the form
	\[
		\begin{split}
		x(t)&=\begin{pmatrix}0\\1\end{pmatrix}-\int_{t}^{\infty}
		\begin{pmatrix}e^{\Lambda_1(t)-\Lambda_1(\tau)+\Lambda_2(t)-\Lambda_2(\tau)}&0\\0&1\end{pmatrix}
		\begin{pmatrix}e_1(\tau)&e_2(\tau)\\e_3(\tau)&-e_1(\tau)\end{pmatrix}x(\tau)d\tau\\
		&=\begin{pmatrix}0\\1\end{pmatrix}-\int_{t}^{\infty}
		\begin{pmatrix}e^{\Lambda(t)-\Lambda(\tau)}&0\\0&1\end{pmatrix}
		\begin{pmatrix}e_1(\tau)&e_2(\tau)\\e_3(\tau)&-e_1(\tau)\end{pmatrix}x(\tau)d\tau.
		\end{split}
	\]
	In the last equation we introduced $\Lambda(t):=\Lambda_1(t)+\Lambda_2(t)$ for better readabillity. Note that the leading term corresponds to the part of the solutions of 
	Hitchin's equation that comes from the limiting configurations. Thus it is this decoupling were the results of \cite{Fredrickson.2020} really come into 
	play.
	
	In order to use our theorem we have to make sure, that the kernel is now at least for large enough $R$ and small enough $\zeta$ a good one. Note that we 
	will use the notion of "small enough $\zeta$" as meaning the existence of some $\zeta_0\in\mathbb{H}_{\vartheta}$ s. t. for all 
	$\zeta\in\mathbb{H}_{\vartheta}$ with $|\zeta|<\zeta_0$ 
	the assertion holds. This corresponds to the use of "large enough $R$".
	
	We start by 
	showing the integrability locally around a parabolic point.
	\begin{lemma}\label{log_decay}
		Fix $\zeta_0\in\mathbb{H}_{\vartheta}$. Locally around any weakly parabolic point $p$ there exists a coordinate $z$, s. t. 
		\[
			\int_{t}^{\infty}\left\|\begin{pmatrix}e^{\Lambda(t)-\Lambda(\tau)}&0\\0&1\end{pmatrix}
			\begin{pmatrix}e_1(\tau)&e_2(\tau)\\e_3(\tau)&-e_1(\tau)\end{pmatrix}\right\|d\tau
			<U(t)e^{-\delta R}
		\]
		for some constant $\delta$ and bounded function $U(t)>0$ that is independent of $R$ for all $\zeta\in\mathbb{H}_{\vartheta}$ with $|\zeta|<|\zeta_0|$ 
		and all $t$ for which $z(t)$ is still in the neighborhood.
	\end{lemma}
	
	\begin{proof}
		Let
		\[
			K:=\left\|\begin{pmatrix}e^{\Lambda(t)-\Lambda(\tau)}&0\\0&1\end{pmatrix}
			\begin{pmatrix}e_1(\tau)&e_2(\tau)\\e_3(\tau)&-e_1(\tau)\end{pmatrix}\right\|
		\]
		denote the integral kernel. We are interested in the existence of $\int_t^{\infty}K(\tau)d\tau$. Clearly $\int_t^{\infty}K(\tau)d\tau$ exists if 
		$\int_{t'}^{\infty}K(\tau)d\tau$ exists for some large $t'$, so we consider $K$ more closely.
	
		First note that $\Lambda$ is a primitive of a function with positive real part, which implies that its real part is monotonically 
		increasing. Thus in this integral $\exp(\Lambda(t)-\Lambda(\tau))<1$, so the first matrix in the integrand has norm smaller 
		then $1$ in any coordinate neighborhood of $p$.
		
		We center around a weakly parabolic point $p$, s. t. Lemma \ref{e_asymptotics} holds, so $r^{-\mu}K\leq Ce^{-\delta R}$.
		As considered in Lemma \ref{log_spiral} we know that $z(\tau)=z_0e^{-\frac{e^{i\vartheta}}{m}\tau}$. Thus the integral over $K$ is bounded from 
		above:
		\[
			\int_{t'}^\infty K d\tau=\int_{t'}^\infty r^{-\mu}Kr^{\mu} d\tau\leq\int_{t'}^\infty Ce^{-\delta R}r^{\mu} d\tau
			\leq Cz_0\int_{t'}^\infty e^{-\frac{e^{i\vartheta}}{m}\mu\tau}d\tau.
		\]
		The last integral is indeed finite, i. e. $\int_{t'}^{\infty}K(\tau)d\tau$ exists and is smaller then $U(t')e^{-\delta R}$ for some function $U$ that 
		is independent of $R$. It follows that 
		$\int_t^{\infty}K d\tau$ exists for all $t$ for which the curve is in this neighborhood and is bounded by some $U(t)$ where the $t$-dependence 
		denotes that decreasing $t$ increases the integration range and thus the upper bound. We will obtain a uniform bound from this shortly in Lemma 
		\ref{small_kernel}. 
		
		As we consider the asymptotic behavior for $\zeta\rightarrow0$ we note that decreasing $|\zeta|$ may only decrease the norm of the error matrix, so 
		while $U(t)$ may depend on $\zeta_0$ the bound holds for all $\zeta$ with smaller absolute value.
		
	\end{proof}
	
	Note that this is the main use of the special form of a $ca$-trajectory in our analysis. We use the fact that the local form of the curve is known and 
	the 
	distance of the curve towards the parabolic point decays fast enough. This is enough to guarantee the existence of the solution and together with the 
	exponential suppression of the remainder term later on also suffices to obtain the correct asymptotics and decaying behavior of the hyperkähler metric. This is how we circumvent the use of the WKB approximation.
	
	Now that we have the desired statement locally around the parabolic points we can expand the region for which it holds along the whole 
	$\vartheta$-trajectory. For this we now use contractible neighborhoods.
	
	\begin{lemma}\label{extend}
		Fix $\zeta_0\in\mathbb{H}_{\vartheta}$. In any contractible open set $W$ inside the standard annulus it holds  
		\[
			\int_{t}^{\infty}\left\|\begin{pmatrix}e^{\Lambda(t)-\Lambda(\tau)}&0\\0&1\end{pmatrix}
			\begin{pmatrix}e_1(\tau)&e_2(\tau)\\e_3(\tau)&-e_1(\tau)\end{pmatrix}\right\|d\tau
			<U(t)e^{-\delta R}
		\]
		for some constant $\delta$ and bounded function $U(t)>0$ that is independent of $R$ for all $\zeta\in\mathbb{H}_{\vartheta}$ with $|\zeta|<|\zeta_0|$ 
		along the generic $\vartheta$-trajectory inside $W$.
	\end{lemma}
	
	\begin{proof}
		The standard neighborhood runs along the $\vartheta$-trajectory up until $p$, so all of the curve (excluding the end point $p$) lies in the 
		intersection of this 
		neighborhood and the local one of the preceding Lemma \ref{log_decay}. Inside this intersection the integrands only differ by the derivative of the 
		change of coordinates map, i. e. if the coordinates are $z$ and $\tilde{z}$ we have
		\[
		\begin{split}
			\int_{\tilde{t}}^{\infty}&\left\|\begin{pmatrix}e^{\Lambda(\tilde{t})-\Lambda(\tau)}&0\\0&1\end{pmatrix}
			\begin{pmatrix}e_1(\tau)&e_2(\tau)\\e_3(\tau)&-e_1(\tau)\end{pmatrix}\right\|d\tau\\
			&=\int_{t}^{\infty}\left\|\begin{pmatrix}e^{\Lambda(t)-\Lambda(\tau)}&0\\0&1\end{pmatrix}
			\begin{pmatrix}e_1(\tau)&e_2(\tau)\\e_3(\tau)&-e_1(\tau)\end{pmatrix}\right\|D(\tilde{z}\circ z^{-1})d\tau.
		\end{split}
		\]
		As $\tilde{z}\circ z^{-1}$ is a diffeomorphic map from a bounded region into a bounded region of $\mathbb{C}$ it follows 
		(from Cauchys integral formula), that $|D(\tilde{z}\circ z^{-1})|$ is bounded. 
		
		Thus the integral is also finite in the standard neighborhood along the intersection. But for smaller values of $t$ in the standard neighborhood only a 
		finite term is added. Thus the integral exists along the whole $\vartheta$-trajectory. As the entries of the error matrix are also exponentially 
		suppressed everywhere, it follows the claim.
	\end{proof}
	
	This would now already be enough to infer the existence of the solution on any interval $[t,\infty)$, but the upper bound involves a $t$-dependence in 
	$U(t)$ which could a priori increase unbounded as $t\to-\infty$ which would then prohibit us from finding a uniform $R$ for which the solution exists 
	all along the curve. Fortunately this can't happen in our setting, as the generic $\vartheta$-trajectory starts and ends at a parabolic point, at each of 
	which the consideration of Lemma \ref{log_decay} applies. So we may take another step to achieve some more uniform behavior.
	
	\begin{lemma}\label{small_kernel}
		There exist uniform $\delta,\widetilde{C}>0$, i. e. independent of $|\zeta|<|\zeta_0|$ and $t$ such that 
		\[
			\int_{t}^{\infty}
			\left\|\begin{pmatrix}e^{\Lambda(t)-\Lambda(\tau)}&0\\0&1\end{pmatrix}
			\begin{pmatrix}e_1(\tau)&e_2(\tau)\\e_3(\tau)&-e_1(\tau)\end{pmatrix}\right\|d\tau<\widetilde{C}e^{-\delta R}
		\]
		for all $R$ large enough and $\zeta$ small enough.
	\end{lemma}
	
	\begin{proof}
		As both ends of the generic $\vartheta$-trajectory run into a weakly parabolic point the arguments we used in the proof of Lemma \ref{log_decay} 
		apply for both ends of the curve and as mentioned in Lemma \ref{extend} between the ends only a finite part is added. Thus
		\[
			\int_{-\infty}^{\infty}
			\left\|\begin{pmatrix}e_1(\tau)&e_2(\tau)\\e_3(\tau)&-e_1(\tau)\end{pmatrix}\right\|d\tau
		\]
		exists. Furthermore it is exponentially suppressed by $Ce^{-\delta R}$ for some constants $C$ and $\delta$ which are now both independent of $t$. With 
		the same arguments as in Lemma \ref{log_decay} they can also be chosen independent of $|\zeta|<|\zeta_0|$. As $\exp(\Lambda(t)-\Lambda(\tau))\leq 1$ 
		still holds everywhere along the $\vartheta$-curve we have for all $t\in\mathbb{R}$ and $|\zeta|<|\zeta_0|$ the exponential suppression
		\[
			\int_{t}^{\infty}
			\left\|\begin{pmatrix}e^{\Lambda(t)-\Lambda(\tau)}&0\\0&1\end{pmatrix}
			\begin{pmatrix}e_1(\tau)&e_2(\tau)\\e_3(\tau)&-e_1(\tau)\end{pmatrix}\right\|d\tau
			\leq \int_{t}^{\infty}\left\|\begin{pmatrix}e_1(\tau)&e_2(\tau)\\e_3(\tau)&-e_1(\tau)\end{pmatrix}\right\|d\tau
			<\widetilde{C}e^{-\delta R}.
		\]
	\end{proof}
	
	We have thus achieved the main estimate for obtaining the solution of the Volterra equation which shall be our basic ingredient for constructing the 
	metric.
	
	\begin{theorem}
		For large enough $R$ there exists a unique solution of the integral equation
		\begin{equation}\label{error_equation}
			x(t)=\begin{pmatrix}0\\1\end{pmatrix}-\int_{t}^{\infty}
					\begin{pmatrix}e^{\Lambda(t)-\Lambda(\tau)}&0\\0&1\end{pmatrix}
					\begin{pmatrix}e_1(\tau)&e_2(\tau)\\e_3(\tau)&-e_1(\tau)\end{pmatrix}x(\tau)d\tau
		\end{equation}
		all along any generic $\vartheta$-trajectory.
	\end{theorem}
	
	\begin{proof}
		From the Lemma \ref{small_kernel} above we infer that the integral over the kernel is $<1$ for large enough $R$ and the initial value function is 
		bounded. Thus via Theorem \ref{ex} a unique solution to the equation exists all along the $\vartheta$-trajectory.
	\end{proof}
	
		With the unique existence of a solution $x$ to the equation we obtained by detaching the leading part we have now a good characterization of the 
		small flat section along a generic $\vartheta$-trajectory $\gamma$ as $s(t)=x(t)e^{-\Lambda_2(t)}$. 
		
	\begin{definition}\label{def_small_flat}
		The \textit{small flat section} for a weakly parabolic point $p$ is the solution of the corresponding initial value problem at infinity that becomes 
		finite when going into $p$ along a generic $ca$-trajectory after detaching the leading term.
	\end{definition}
	
	In this definition the choice of $(0,1)^t$ as leading term is arbitrary. Other complex multiples would work just as well and cancel in the final expression of $\mathcal{X}$. Additionally the choice for $p$ now depends on the specific $ca$-curve we use to approach $p$. Taking the boundary of a standard quadrilateral works, as we simply need a way to choose a section. One might worry, that the section that is small along one curve may become large for a curve in the same isotopy class which was used in defining the triangulation, but as we will show in the next section, this can actually not occur.
	
\subsection{Constructing the coordinates}\label{sec_coord}

	Now that we have the solution we can say more about its asymptotic behavior and its derivative. For this we keep the notions that we fixed before 
	and also fix a $\zeta_0\in\mathbb{H}_{\vartheta}$. We now talk about "small" $\zeta\in\mathbb{H}_{\vartheta}$ i. e. $|\zeta|<|\zeta_0|$, so the 
	statements above lead to uniform constants. 
	
	\begin{proposition}\label{rest_form}
		For small $\zeta$ the small flat section along a $ca$-trajectory $\gamma$ is of the form
		\[
			s(t)=e^{-\Lambda_2(t)}\left(\begin{pmatrix}0\\1\end{pmatrix}+x^r\right),
		\]
		with $x^r$ some $\mathbb{C}^2$-valued function with $|x^r|\leq \tilde{K}e^{-\delta R}$ for some constant $\tilde{K}$.
	\end{proposition}
	
	\begin{proof}
		First we note, that the previous 
		considerations lead for small enough $\zeta$ to a (uniform) constant $K$, s. t. 
		\[
			\int_{t}^{\infty}
			\left\|\begin{pmatrix}e^{\Lambda(t)-\Lambda(\tau)}&0\\0&1\end{pmatrix}
			\begin{pmatrix}e_1(\tau)&e_2(\tau)\\e_3(\tau)&-e_1(\tau)\end{pmatrix}\right\|d\tau<Ke^{-\delta R}.
		\]
		The estimate from the integral equation theory \ref{est} thus leads to 
		\[
			|x|\leq\frac{1}{1-Ke^{-\delta R}}.
		\]
		Plugging this back into the right hand side of the integral equation \eqref{error_equation}, we obtain the estimate
		\[
			\left\|\int_{t}^{\infty}
			\begin{pmatrix}e^{\Lambda(t)-\Lambda(\tau)}&0\\0&1\end{pmatrix}
			\begin{pmatrix}e_1(\tau)&e_2(\tau)\\e_3(\tau)&-e_1(\tau)\end{pmatrix}x(\tau)d\tau\right\|
			\leq Ke^{-\delta R}\cdot\frac{1}{1-Ke^{-\delta R}}
			=:\tilde{K}e^{-\delta R}.
		\]
		Here we use that for $R$ large enough $\tfrac{K}{1-Ke^{-\delta R}}$ is bounded by some (uniform) constant $\tilde{K}$. Thus the integral equation leads 
		to 
		\[
			x=\begin{pmatrix}0\\1\end{pmatrix}+x^r
		\]
		where $x^r$ is some function with $|x^r|\leq \tilde{K}e^{-\delta R}$. The form for $s$ follows readily.
	\end{proof}
	
	Thus we see that for large enough $R$ the small flat section $s$ is dominated by the leading part $(0,1)^t$ of $x$.
	
	Now we want to build the product $s^1\wedge s^2$ where $s^1$ and $s^2$ are the small flat sections for two weakly parabolic points $p_1$ and $p_2$ joined 
	by the $\vartheta$-trajectory. We recall that we fixed the orientation of the $\vartheta$-trajectory $\gamma$ in section \ref{sec_small_sec}. So if 
	$s^1$ is the small flat section constructed there, then for constructing $s^2$ we have to change the orientation of the $\vartheta$-curve, i. e. 
	consider the flatness equation \eqref{flatness_equation} along $\widetilde{\gamma}(t):=\gamma(-t)$ which leads to a change of sign in the ODE, which in 
	turn exchanges the role of the first and second coordinate in the local expression of the solution. The calculations then work just as before and we 
	obtain a solution $s^2$ that decays in the other direction. Now evaluating this section again along $\gamma$ we obtain
	\[
		s^2(t)=e^{\Lambda_1(t)}\left(\begin{pmatrix}1\\0\end{pmatrix}+y^r\right),
	\]
	for some exponentially suppressed $y^r$. We have thus obtained that the small flat section for one direction grows in the other direction and can thus 
	infer the next important result.
	
	\begin{proposition}\label{indep_dec}
		The small flat sections for two weakly parabolic points that are connect by a generic $ca$-trajectory are linearly independent decorations.
	\end{proposition}
	
	Note that as of the complex $2$-dimensional nature of the space of solutions of the linear ODE \eqref{flatness_equation} we have thus obtained all solutions 
	there are.
	
	\begin{corollary}\label{fund_sys}
		The small flat sections for two weakly parabolic points connected by a generic $ca$-curve form a fundamental system for the ODE \eqref{flatness_equation}.
	\end{corollary}
	
	Before proceeding, we have to take care that the definitions make sense everywhere in a quadrilateral containing two zeroes of $q$. The important objects to 
	take care of are the primitive functions $\Lambda_j$ of $\lambda_j$ for $j\in\left\{1,2\right\}$. These can't be chosen globally, so we take one primitive $\Lambda_j^i$ for each puncture, 
	each defined on a neighborhood of the two $\vartheta$-trajectories running into the parabolic point $p_i$ (and excluding $p_i$ itself). Now we can 
	calculate the $\wedge$-product which is locally evaluated as det$\left[s^1,s^2\right]$.
	
	\begin{corollary}\label{wedge_eval}
		The wedge-product of two small flat sections $s^1$ and $s^2$ along a $\vartheta$-trajectory with remainder terms $x^r$ and $y^r$ is
		\[
			s^1\wedge s^2=e^{-\Lambda_2^1}\left(\begin{pmatrix}0\\1\end{pmatrix}+x^r\right)\wedge e^{\Lambda_1^2}\left(\begin{pmatrix}1\\0\end{pmatrix}+y^r\right)
			=e^{\Lambda_1^2-\Lambda_2^1}(1+r),
		\]
		for some function $r$ which is exponentially suppressed in $R$ (independently of $\zeta$).
	\end{corollary}

	Until now we considered the sections evaluated along some $\vartheta$-trajectory $\gamma$. For the formula of the $\mathcal{X}$-coordinates the section 
	needs to be 
	evaluated at points along two $\vartheta$-trajectories (that are the two sides of the quadrilateral running into the same vertex). For this we consider that the 
	section is well defined on the contractible neighborhood of the two $\vartheta$-curves and obtained from the solution along $\gamma$ via parallel 
	transport. Thus we my obtain the value on the second curve by connecting the two curves along any other curve inside the neighborhood and solving the 
	ODE along this connecting curve. But we may use the fact that the solutions along the other curve are also obtained as corresponding small flat 
	sections, as they give us the full space of solutions of the ODE. Thus we only need to examine whether the curve may actually become unbounded along the 
	second curve when running into the pole. Otherwise it will still be a small flat section along this second curve and thus a solution to the ODE we already 
	solved, albeit it with maybe another initial value at infinity which will be important later as we discuss the derivative.
	
	To get a good control over the parallel transport we use the neighborhood around a pole as the $\vartheta$-trajectories are well known there and they allow 
	for good approximation of the solution as we've already seen. The following observation is needed later on and follows rather directly from the definition:
	
	\begin{lemma}\label{const_angle}
		Let $\alpha$ be a $\vartheta$-curve and $\beta$ be a $\widetilde{\vartheta}$-curve for some quadratic differential $q$ in any local chart. Then at 
		every intersection the (euclidean) angle between $\alpha$ and $\beta$ is independent of the point of intersection and the specific curves $\alpha$ and 
		$\beta$.
	\end{lemma}
	
	\begin{proof}
		Consider $\alpha$ and $\beta$ in standard parametrization. Let $x$ be a point of intersection, i. e. $\alpha(t_1)=x=\beta(t_2)$. From the definition it 
		follows
		\[
			q(\alpha(t))\alpha'(t)=e^{i\vartheta}=e^{i(\vartheta-\widetilde{\vartheta})}e^{i\widetilde{\vartheta}}
			=e^{i(\vartheta-\widetilde{\vartheta})}q(\beta(t))\beta'(t)
		\]
		for all $t$. Thus at the point of intersection we obtain
		\[
			\alpha'(t_1)=e^{i(\vartheta-\widetilde{\vartheta})}\beta'(t_2).
		\]
		Thus we see that the tangents are related by a rotation of $\vartheta-\widetilde{\vartheta}$ which is independent of the choice of curves or 
		intersection point.
	\end{proof}
	
	Now let $\gamma$ and $\nu$ be the two sides of $Q$ that share some vertex $p$ and $s$ the small flat section along $\gamma$. The idea now is to 
	use the same techniques as before to obtain an estimate along $\nu$ by estimating a corresponding integral equation along a curve that connects 
	$\gamma$ and $\nu$. For this we first have to define a curve along which we parallel transport the section. We need the curve to lie inside the 
	quadrilateral and it has to 
	be of a form which allows for the same kinds of calculations as before, i. e. the growth of the section along the curve has to be known in some way. 
	It is thus natural in our setting to look for $ca$-curves connecting the sides of the quadrilateral.
	
	\begin{proposition}\label{conn_curve}
		Let $U$ be the standard neighborhood around a weakly parabolic point and $\gamma$ and $\nu$ the parts of two sides of the standard quadrilateral 
		for $\vartheta$ in this neighborhood. Let $V$ be a neighborhood of $\vartheta$ and $x=\gamma(t_x)$ any Element in the image of $\gamma$.
		
		Then there exists $\widetilde{\vartheta}\in V$, and a $\widetilde{\vartheta}$-trajectory $\widetilde{\gamma}$ that intersects $\gamma$ in 
		$x$ and $\nu$ in $y:=\nu(t_y)$ that lies in the quadrilateral for all $t_x<t<t_y$. 
		
		Furthermore $t_y-t_x=d$ for some constant $d_{\widetilde{\vartheta}}$ that does depend on $\widetilde{\vartheta}$ and the choice of $\gamma$ and 
		$\nu$, but not on the choice of $x$ on $\gamma$.
	\end{proposition}
	
	\begin{proof}
	
		\begin{figure}[!htbp] 
		\centering
		\scalebox{0.25}{\includegraphics{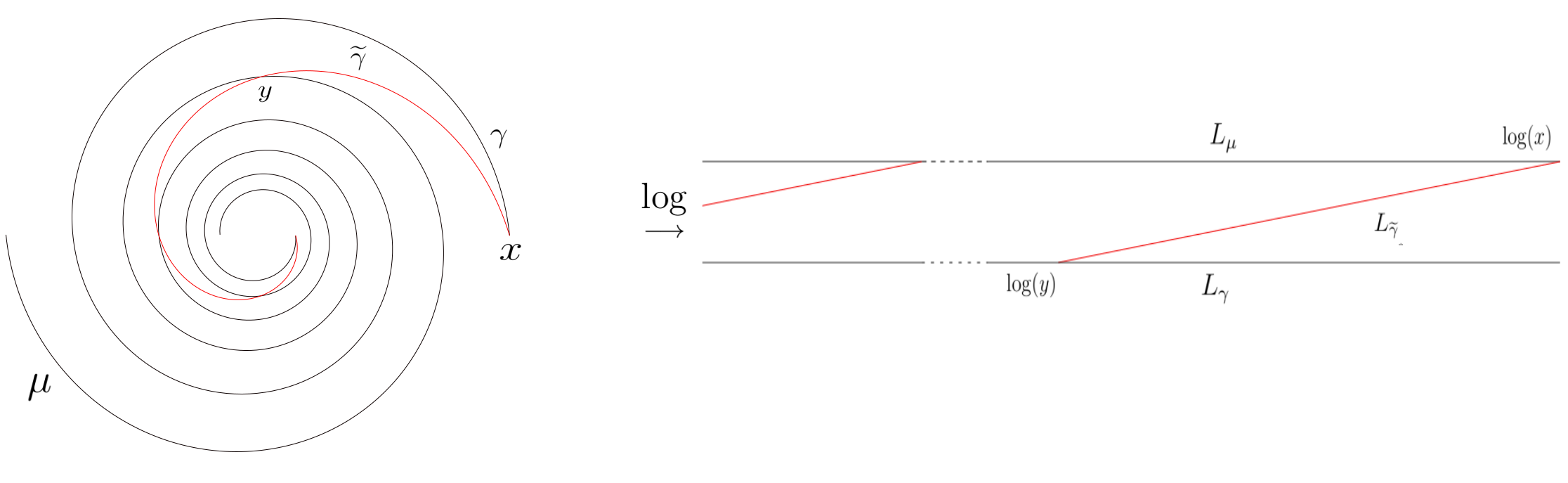}}
		\caption{The $\vartheta$-trajectories are mapped by the logarithm to a horizontal strip, whereas the $\widetilde{\vartheta}$-trajectory is mapped to a 
		straight line intersecting the strip at constant angles. }
		\label{log_wirkung}
		\end{figure}
		
		We consider the standard neighborhood $U$ around a weakly parabolic point, where $-qdz=\det(\Phi)=\frac{m}{z}dz$. In this neighborhood the 
		$\vartheta$-trajectories are known to be logarithmic spirals (or straight lines) except for two exceptional values of $\vartheta_c$ where they become 
		concentric 
		circles around the parabolic point. We may  parametrize all these spirals as $\mu(t)=\mu_{0}e^{i\varpi t}$ for differing values of $\mu_0$ and 
		with $\varpi$ only depending on $\vartheta$ with $\Re(\varpi)<0$ as seen before. Let $\gamma$ and $\nu$ be the parts of two sides of the 
		standard quadrilateral in this neighborhood. Their images together with $0$ 
		separates the neighborhood into two regions, one of which corresponds to the inside of the quadrilateral, which shall be called $I$. As the $\vartheta$-
		trajectories don't intersect, there are $\vartheta$-trajectories that lie in the complement $I^c$. We may take any one of them, say $\alpha$, to be a 
		branch cut on 
		$U$ and define a biholomorphic logarithm on $U\setminus\left(\Im(\alpha)\cup\left\{0\right\}\right)$. The domain of $\log$ then includes $I$. The 
		logarithm now maps all the $\vartheta$-trajectories $\mu$ (excluding $\alpha$) to parallel lines $L_\mu$ which we rotate to be horizontal, so 
		$L_{\gamma}$ and $L_{\nu}$ define a horizontal strip $S$. In the following $\log$ shall denote the composition of the logarithm and the rotation.
		In $\Im(\log)$ the left endpoints of the $L_{\mu}$ correspond to the values $\mu_0$ of $\mu$ at $t=0$ in $U$ while $0\in U$ is mapped to 
		$-\infty$ to the left.
		
		Now let $\widetilde{\gamma}$ be a $ca$-trajectory in for some $\widetilde{\vartheta}\neq\pm\vartheta$ that intersects $\gamma$ at some point $x$. The 
		part 
		of the curve starting at $x$ is then also mapped to a straight line at least until some point, where it intersects $\alpha$, which happens somewhere, as 
		$\widetilde{\vartheta}\neq\pm\vartheta$ (cf. figure \ref{log_wirkung}).
		
		As the mapping is biholomorphic $L_{\widetilde{\gamma}}$ intersects $L_{\gamma}$ in the point $\log(x)$ at some non-zero angle $\beta$ 
		which, as of Lemma \ref{const_angle} is constant for all intersection points. We recall 
		that the $ca$-trajectories are oriented as running into $0$. For $\widetilde{\vartheta}$ close enough to $\vartheta$ we thus have that depending on 
		wether 
		$\widetilde{\vartheta}$ is bigger or smaller then $\vartheta$ the straight line $L_{\widetilde{\gamma}}$ continues in the strip $S$ or outside 
		of $S$. We may thus take $\widetilde{\vartheta}$ arbitrarily close to $\vartheta$ and such that $L_{\widetilde{\gamma}}$ continues in the 
		strip $S$. Being a straight line, $L_{\widetilde{\gamma}}$ will intersect $L_{\mu}$ at some point $\log(y)$ and the thus defined segment 
		is the image of a part of the $ca$-trajectory $\widetilde{\gamma}$ inside $I$ under the logarithm. Thus $\widetilde{\gamma}$ intersects $\gamma$ and 
		$\mu$ in $x=\gamma(t_x)$ and $y=\mu(t_y)$ stays inside $I$ for all $t_x<t<t_y$.
		
		To calculate $t_x-t_y$ note that $L_{\gamma}=a_1+t\Re(\varpi)$ and $L_{\mu}=a_2+t\Re(\varpi)$ for $a_1:=\log(\gamma(0))$ and 
		$a_2:=\log(\mu(0))$. Now 
		$L_{\widetilde{\gamma}}$ has the form $L_{\widetilde{\gamma}}(t)=L_0+te^{i\beta}$ and intersects $L_{\gamma}$ in $t_x$ and $L_{\mu}$ in  $t_y$. We 
		obtain, that
		\[
			\begin{split}
			a_2-a_1+\Re(\varpi)(t_y-t_x)&=a_2+t_y\Re(\varpi)-(a_1+t_x\Re(\varpi))=L_{\mu}(t_y)-L_{\gamma}(t_x)\\
			&=L_{\widetilde{\gamma}}(t_y)-L_{\widetilde{\gamma}}(t_x)
			=t_ye^{i\beta}-t_xe^{i\beta}=(t_y-t_x)e^{i\beta}\\
			\Rightarrow |t_y-t_x|&\leq \left|\frac{a_2-a_1}{e^{i\beta}-\Re(\varpi)}\right|=:d.
			\end{split}
		\]
		Note that $a_2$ and $a_1$ are constants that only depend on the sides of the quadrilateral, as is $\varpi$. Also $\beta$ is constant for all 
		intersections of $\widetilde{\gamma}$ and any part of any $\vartheta$-trajectory everywhere in $I$ as of Lemma \ref{const_angle} and the 
		biholomorphicity of $\log$. Thus we conclude that $|t_y-t_x|=d$ is constant for all $\widetilde{\gamma}$-trajectories connecting $\gamma$ and $\mu$ 
		in $U$.
	\end{proof}
		
	We have thus obtained the prerequisites for transporting the small flat section that was obtained along some side $\gamma$ of the quadrilateral $Q$ to 
	another side $\mu$. For the proof we can now use ideas similar to the ones that helped befored. 
		
	\begin{proposition}\label{connecting_coord}
		Let $\gamma$ and $\mu$ be two sides of the quadrilateral $Q$ sharing a vertex $p$ and $s$ the small flat section along $\gamma$. If $\zeta$ is small enough 
		then $s$ 
		evaluated along $\mu$ is of the form
		\[
			s(t)=e^{-\Lambda_2(t)}\left(\begin{pmatrix}0\\1\end{pmatrix}+x^r\right),
		\]
		with $x^r$ some $\mathbb{C}^2$-valued function with $|x^r|\leq \tilde{K}e^{-\delta R}$ for some constant $\tilde{K}$.
	\end{proposition}
		
	\begin{proof}
		To parallel transport the small flat section from one side of the quadrilateral to another we use the formulation via an integral equation as we have 
		done before, only now we do not need to consider the integral up to infinity. We recall that in our set-up $\vartheta$ is fixed and we consider only 
		those $\zeta$ which lie in the open half-plane $\mathbb{H}_\vartheta$. Thus $\zeta$ also lies in the half-plane $\mathbb{H}_{\widetilde{\vartheta}}$ for 
		$\widetilde{\vartheta}$ close enough to $\vartheta$. Now we chose any point $p_s$ on $\mu$ close enough to the parabolic point, s. t. it lies in 
		the neighborhood where the $ca$-trajectories may be expressed as logarithmic spirals. As of the previous Proposition \ref{conn_curve} there exists a 
		$\widetilde{\vartheta}$-trajectory $\alpha:[a,b]\to C$ with starting point $\alpha(a)=p_s$ and first intersection with $\gamma$ at $\alpha(b)=p_e$, 
		s. t. the image of the curve is contained in the quadrilateral and $b-a\leq d$ for some constant $d$ that does not depend on the starting point $p_s$.
		
		We now consider the solution of our ODE via variation of constants along $\alpha$, i. e. the parallel transport along $\alpha$ is the unique solution 
		$x:[a,b]\to\mathbb{R}^2$ to the integral equation
		\begin{equation}\label{connecting_equ}
			x(t)=\begin{pmatrix}e^{\Lambda(t)-\Lambda(b)}&0\\0&1\end{pmatrix}x(b)
			-\int_{t}^b\begin{pmatrix}e^{\Lambda(t)-\Lambda(\tau)}&0\\0&1\end{pmatrix}\begin{pmatrix}e_1&e_2\\e_3&-e_1\end{pmatrix}x(\tau)d\tau.
		\end{equation}
		Note that $x(b)$ is now the "initial value" of $x$ at the end point $p_e$ instead of the value at infinity.
		Here $\Lambda$ is again a primitive of $\lambda$ as above but now evaluated along $\alpha$. The theory developed before also holds for these usual Volterra equations where the integral does not go to infinity \cite{Windsor.2010}.\footnote{More accurately the theory 
		developed here tried to broaden the already existing theory for these equations to the case of an initial value at infinity.} 
		
		Thus we have a unique solution from which we would like to gather information about $x(a)=x(p_s)$ on $\mu$ by estimating it against $x(b)=x(p_e)$ 
		on $\gamma$, using the techniques we already developed for the infinite case. As $\zeta$ lies in the half-plane 
		$\mathbb{H}_{\widetilde{\vartheta}}$ we still have that $\Lambda(t)-\Lambda(\tau)$ has negative real part, which holds in particular for $\tau=b$. Thus the norm 
		of the first matrix in the integrand 
		is smaller then $1$ while the norm of the error matrix stays bounded by some $Ce^{-\delta R}$ as before for $\zeta$ small enough. As the domain of the 
		integral is bounded by $d$ we obtain for large enough $R$ that the integral over the kernel is smaller then $1$, so a unique solution exists. As the 
		matrix in the initial value function is also bounded by $1$, we obtain:
		\[
			|x(t)|\leq\frac{\left|x(b)\right|}{1-dCe^{\delta R}}.
		\]
		Plugging this estimate back into the the right hand side of the equation \eqref{connecting_equ} we obtain for the integral
		\begin{equation}\label{con_est}
			\begin{split}
			\left|\int_{t}^b\begin{pmatrix}e^{\Lambda(t)-\Lambda(\tau)}&0\\0&1\end{pmatrix}\begin{pmatrix}e_1&e_2\\e_3&-e_1\end{pmatrix}x(\tau)d\tau\right|
			&\leq \int_{t}^bCe^{-\delta R}\frac{\left|x(b)\right|}{1-dCe^{-\delta R}}d\tau\\
			&\leq|x(b)|\widetilde{C}e^{-\delta R}
			\end{split}
		\end{equation}
		for some uniform constant $\widetilde{C}$. The full equation \eqref{connecting_equ} thus shows us, that we have:
		\begin{equation}\label{conn_sol}
			x(t)=\begin{pmatrix}e^{\Lambda(t)-\Lambda(b)}&0\\0&1\end{pmatrix}x(b)+v(t)
		\end{equation}
		where $|v(t)|\leq|x(b)|\widetilde{C}e^{-\delta R}$ along the whole segment. As $\Lambda(t)-\Lambda(\tau)\leq0$ along the segment, we obtain the estimate
		\begin{equation}\label{conn_est}
			|x(t)|\leq |x(b)|(1+Ce^{-\delta R}).
		\end{equation}
		In particular we have the (in)equalities \ref{conn_sol} and \ref{conn_est} for $t=a$, i. e. for large enough $R$ the endpoints $x(a)$ and $x(b)$ on the 
		two 
		sides of the quadrilateral differ by some (bounded) multiplicative factor in the first component (that vanishes for $\zeta\to0$ or $R\to\infty$) and 
		an added vector which is bounded by the norm of the function on the other side multiplied with some exponentially suppressed constant. Note that this 
		argument is valid for all $\widetilde{\vartheta}$ segments arbitrarily close to the weakly parabolic point $p$ and the upper bounds can be chosen to be 
		valid for all arbitrarily close segments. Now $x(b)$ is the value of the section on the side where the small flat 
		section was defined which is bounded near $p$, so the value on the other side is also bounded. This already shows that the transported section along 
		the second side must also be a small flat section as the big sections would have to surpass every upper bound somewhere close to the parabolic 
		point. Thus we know that the section has a limit along the second side going into the parabolic point and the first component of this limit must 
		vanish. As $x(b)$ has limit $(0,1)^t$ going into $p$ along $\gamma$ the limit along the second side $\mu$ is $(0,1+\epsilon)$ for some constant 
		$\epsilon$ with norm 
		$|\epsilon|\leq Ce^{-\delta R}$ for some uniform constant $C$ for $\zeta$ small enough.
		
		Therefore the section along $\mu$ is the solution of 
		\[
			\begin{split}
			x(t)=\begin{pmatrix}0\\1+\epsilon\end{pmatrix}-\int_{t}^{\infty}
			\begin{pmatrix}e^{\Lambda(t)-\Lambda(\tau)}&0\\0&1\end{pmatrix}
			\begin{pmatrix}e_1(\tau)&e_2(\tau)\\e_3(\tau)&-e_1(\tau)\end{pmatrix}x(\tau)d\tau,
			\end{split}
		\]
		where $\Lambda$ and the integral are now taken along $\mu$. As $\epsilon$ is exponentially bounded it follows with the same arguments as in 
		Proposition \ref{rest_form} that $x$ is of the form 
		\[
			x(t)=\begin{pmatrix}0\\1\end{pmatrix}+x^r,
		\]
		with $x^r$ some $\mathbb{C}^2$-valued function with $|x^r|\leq \tilde{K}e^{-\delta R}$ for some constant $\tilde{K}$. Again the form for $s$ follows 
		readily.
	\end{proof}
	
	We can now proceed to build the $\mathcal{X}_E$ coordinate (where we mostly suppress the explicit dependence on the triangulation, $q$ and $\vartheta$ in 
	the notation). 
	For this consider again the standard quadrilateral $Q$ associated to an edge $E$ of the triangulation with vertices given by four weakly parabolic 
	points 
	$p_i$.\footnote{In the following all of the expressions involving the indices of the vertices, like $p_{i+2}$, are to be considered as mod $4$ as we talk 
	about a quadrilateral with four vertices.} We want to 
	build the $\wedge$-product according to the definition \ref{Darboux_coords} of $\mathcal{X}_E$. But to do so we have to consider that our calculations up 
	until now was always validated on contractible subsets of the standard annulus $S$ in $Q$ and we already introduced different primitives for $\lambda$ in 
	the different patches belonging to two different $\vartheta$-curves. This now has to be done for all of the four sections of the quadrilateral. 
	Additionally for each vertex $p_i$ there are two sides joining $p_i$ and thus two small flat sections of $\nabla(\zeta)$ which are given by the solution 
	of the integral equation along these two curves, from which one has to be chosen. Our way for doing this is the following: 
	
	First we remember that we fixed the orientation for the $\vartheta$-curves depending on $q^{1/2}$ (cf. the discussion of the flatness equation 
	\eqref{flatness_equation}). This implies that if one side of the quadrilateral $Q$ is given by a curve oriented as running into a vertex $p_i$ 
	(for $t\to\infty$) the other 
	side of $Q$ sharing the vertex $p_i$ is also oriented as running into $p_i$. Thus two vertices of $Q$ are points into which the $\vartheta$-curves run, 
	and the other two vertices are points from which the curves emerge. We start with labeling one of the points into which the curves run as $p_1$ and label 
	the following vertices clockwise. Now we label the side joining $p_4$ and $p_1$ as $\gamma_1$ with is thus oriented in such a way that $\gamma_1$ 
	converges to $p_1$ for $t\to\infty$. The other side joining $p_1$ connects $p_1$ and $p_2$ and shall be labeled as $\gamma_2$ and is also oriented in 
	such a way that $\gamma_2$ converges to $p_1$ for $t\to\infty$. Now $p_2$ will be a vertex from which $\vartheta$-curves emerge, while $p_3$ will be of 
	the same type as $p_1$. So we proceed for $\gamma_3$ and $\gamma_4$ in the same way as for $\gamma_1$ and $\gamma_2$, just now for the vertex $p_3$. 
	In this way we obtain four oriented curves as sides of $Q$. For each of the vertices 
	$p_i$ let $\Lambda^i_1$ be a primitive of $\lambda_1:=-\frac{R}{\zeta}q^{1/2}-R\zeta q^{1/2}-\overline{a_1}d\overline{z}+a_1dz$ and $\Lambda^i_2$ a 
	primitive of 
	$\lambda_2:=-\frac{R}{\zeta}q^{1/2}-R\zeta q^{1/2}+\overline{a_2}d\overline{z}-a_2dz$ defined on a simply connected neighborhood $U_i$ of $\gamma_i$ and 
	$\gamma_{i+1}$ in $S$\footnote{This neighborhood can be taken to be $S$ minus some branch cut connecting $p_{i+2}$ and the inner 
	boundary of the annulus $S$.}. Note that these primitives exist as $p$ is holomorphic and the frame was chosen in such a way that 
	$d(\overline{a_j}-a_j)=0$. Evaluated along a $\vartheta$-curve these primitives become the primitives that were used in the calculations above 
	(cf. equation \eqref{flatness_equation}). Finally 
	for each vertex we choose $s^i$ to be the small flat section along 
	$\gamma_i$ defined using the primitives $\Lambda_j^i$ and extended to sections of $\nabla(\zeta)$ on $U_i$ (cf. figure \ref{coords_quad}). This 
	extension is possible as the $U_i$ are simply connected and $\nabla(\zeta)$ is flat.
	We note that this way of choosing the sections is not necessary. Other 
	choices lead to the same $\mathcal{X}$ coordinates as the small flat sections only differ by a (constant) complex scalar, which cancels in the final 
	$\wedge$-product. This also justifies the arbitrary choice of the "initial value at infinity" as $(0,1)^t$. Thus we may now define the $\mathcal{X}$-
	coordinate.
	
	\begin{figure}[!htbp]
		\centering
		\scalebox{0.3}{\includegraphics{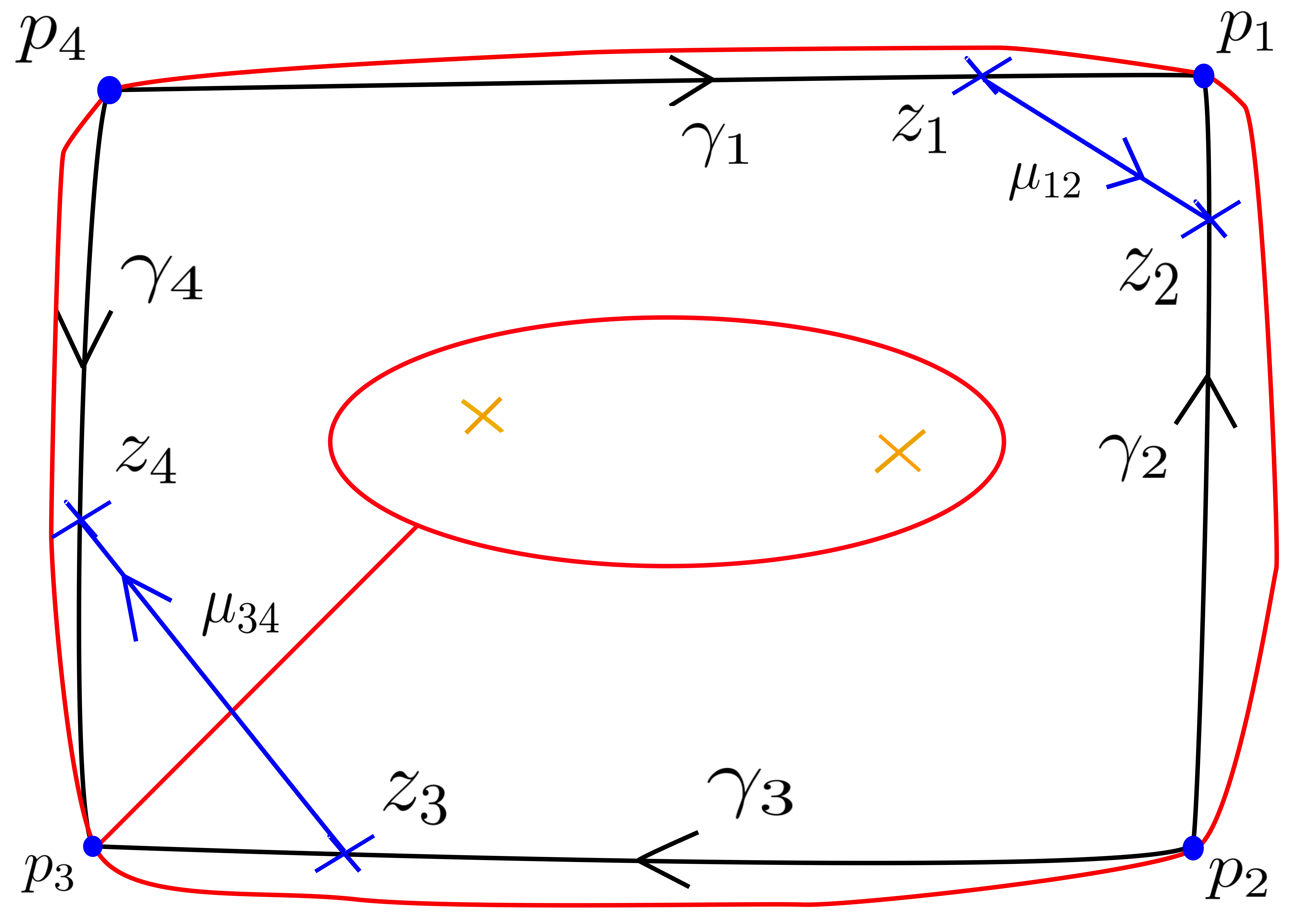}}
		\caption{A standard quadrilateral given by four oriented $\vartheta$-curves $\gamma_i$. The orange crosses denote the zeroes of $-q$ and the blue circles the weakly parabolic points. The blue crosses denote the evaluation points for the small flat sections and the blue lines are curves along which two of the $\wedge$-products are transported. The red outer and inner "circles" are the boundary of a standard annulus. Together with the red straight arc they form the boundary of the simply connected neighborhood $U_1$ for $\gamma_1$.}
		\label{coords_quad}
		\end{figure}
	
	\begin{proposition}\label{x_well}
		Let $Q$ be a standard quadrilateral with vertices $p_i$ and small flat sections chosen as described above. Then for $z_1\in U_1\cap U_4$ and 
		$z_3\in U_2\cap U_3$ the product
		\[
			-\frac{\left(s^{1}\wedge s^{2}\right)(z_1)}{\left(s^{4}\wedge s^{1}\right)(z_1)}
			\frac{\left(s^{3}\wedge s^{4}\right)(z_3)}{\left(s^{2}\wedge s^{3}\right)(z_3)}
			=:-\frac{\left(s^{1}\wedge s^{2}\right)\left(s^{3}\wedge s^{4}\right)}{\left(s^{4}\wedge s^{1}\right)\left(s^{2}\wedge s^{3}\right)}
			=:\mathcal{X}_E
		\]
		is well defined, non-zero and independent of the evaluation points $z_1$, $z_3$.
	\end{proposition}
	
	\begin{proof}
		By construction the $\wedge$-products are defined on $z_1$ and $z_3$. As of Propositions \ref{indep_dec} and \ref{connecting_coord} it follows that all 
		of the $\wedge$-products are non-zero, so $\mathcal{X}$ is well defined and non-zero. The independence of the evaluation points is also true by 
		construction, cf. the discussion preceding the definition \ref{Darboux_coords}.
	\end{proof}
	
	So $\mathcal{X}_E$ only depends on the parameters of the Higgs pair, $\zeta$, $\vartheta$ and the 
	choice of quadrilateral. Thus if we fix $\zeta$ and $\vartheta$ we obtain functions on $\mathcal{M}$. As described in section \ref{sec_dec} it follows 
	from the work of Fock and Goncharov \cite{Fock.2006} and the treatment in \cite{Gaiotto.2013} that the collection of these functions for all edges are 
	actually coordinates of $\mathcal{M}$, called the \textit{canonical coordinates.}
	
	Now consider the cycle $\gamma_E\in H_1(\Sigma'_q;\mathbb{Z})$ encircling the preimage of the two zeroes inside of $Q$ in the (punctured) spectral 
	curve $\Sigma'_q$. One of the key claims in \cite{Gaiotto.2013} is that in the $\zeta\rightarrow0$ asymptotic the period integral $\oint_\gamma 
	\left(q^{1/2}\right)^\ast$ emerges as the leading term for the $\mathcal{X}$-coordinates, where $\left(q^{1/2}\right)^\ast$ denotes the Liouville 
	(tautological) $1$-form on $T^\ast C$
	That is by pushing $\gamma_E$ down to a cycle on $C$ (which we'll also denote by $\gamma_E$) one has 
	$\mathcal{X}_{\gamma_E}\sim c_{\gamma_E} \exp(\frac{R}{\zeta}\oint_{\gamma_E} q^{1/2}dz)$ 
	for some constant $c_{\gamma_E}$. Their conjecture was based on the WKB method but had not been rigorously proven until now. With the calculations done 
	above we can 
	now achieve this, but one technical detail has to be accounted for. As noted above \cite{Gaiotto.2013} used the convention that the matrix representing 
	$\nabla(\zeta)$ is traceless, which leads to a $z$-invariance of the $\wedge$-products even before taking the quotient in the $\mathcal{X}$-definition. 
	Thus they can evaluate the sections on each curve so that the primitives combine in just the right way to create the period integral. In our set-up the 
	the connection matrix is not traceless, so this independence does not hold. Additionally from Propositions \ref{indep_dec} and \ref{connecting_coord} we 
	only have the detailed description of $s^i\wedge s^k$ along the curve connecting the vertices $p_i$ and $p_k$. But the definition of $\mathcal{X}$ 
	implies that at least two of the $\wedge$-products have to be evaluated in a point on some other side. We have to account for this by parallel transporting 
	two of the $\wedge$-products to other sides. As the $s^i\wedge s^k$ are sections of a line bundle this is done rather easily.
	
	\begin{lemma}\label{transport_wedge}
		Let $\gamma_i$ and $\gamma_{i+1}$ be the sides of a standard quadrilateral $Q$ sharing the vertex $p_i$, with $z_i\in\text{Im}(\gamma_i)$ 
		and $z_{i+1}\in\text{Im}(\gamma_{i+1})$ two points on these sides of $Q$. Let $s^i$ and $s^{i+1}$ be the corresponding small flat sections defined on 
		the neighborhoods $U_i$ and $U_{i+1}$.
		
		Then for every curve $\mu\in U_i\cap U_{i+1}$ connecting $z_i$ and $z_{i+1}$ it holds:
		\[
			s^i\wedge s^{i+1}(w_2)=e^{\int_\mu (a_1+a_2)dz-\overline{a_1+a_2}d\overline{z}}s^i\wedge s^{i+1}(w_1),
		\]
		
	\end{lemma}
	
	\begin{proof}
		From Proposition \ref{connecting_coord} and Corollary \ref{fund_sys} $s^i$ and $s^j$ define a fundamental system for the linear ODE $s'(t)=B(t)s(t)$ 
		given by the flatness equation \eqref{flatness_equation} along $\mu$. The equation is now simply an application of Liouville's formula where we note that 
		the only part of the connection matrix of $\nabla(\zeta)$ that is not traceless comes from the leading part of the Chern connection 
		(cf. Remark \ref{trace_prop}), which is just $(a_1+a_2)dz-\overline{a_1+a_2}d\overline{z}$ (cf. equation \eqref{leading}).
	\end{proof}
	
	Finally we can see the emergence of the period integrals. 
	
	\begin{theorem}\label{x_int_form}
		Let $Q$ be the standard quadrilateral for an edge $E$ of a $\vartheta$-triangulation of $C$. The canonical coordinate for $Q$ can be expressed as
		\[
			\mathcal{X}_E=-(1+r_q)
			\exp\left(\frac{R}{\zeta}\oint_{\gamma_E} q^{1/2}+R\zeta\oint_{\gamma_E}\overline{q^{1/2}}+\oint_{\gamma_E}\overline{a_1}d\overline{z}-a_1dz\right),
		\]
		where $r_q$ is exponentially suppressed in $R$ for small enough $\zeta$, and $\gamma_E$ is the projection of a cycle in $H_1(\Sigma'_q;\mathbb{Z})$ 
		corresponding to the edge $E$, that is a simple closed curve in $Q$ encircling the two zeros of $q$ once.
	\end{theorem}
	
	\begin{proof}
	With the same notations as before we choose a point $z_i\in\text{Im}(\gamma_i)$ on each of 
	the sides of the standard quadrilateral $Q$ and curves $\mu_{kl}$ connecting the points $z_{k}$ and $z_l$ in $U_k\cap U_l$ 
	(cf. figure \ref{coords_quad}). For abbreviation we write 
	$\text{tr}\left(\nabla(\zeta)\right):=\overline{a_1+a_2}d\overline{z}-(a_1+a_2)dz$. Then 
	using Lemma \ref{transport_wedge} to parallel transport the 
	$\wedge$-products to the sides where they can be evaluated as in corollary \ref{wedge_eval} we obtain the following expression:
	\footnote{The right choice of lower indices of $\Lambda^i_{j}$ can be a bit confusing here. Proposition \ref{rest_form} and the following discussion 
	imply that a small flat section 
	$s^i$ for the vertex $p_i$ along the curve $\gamma_i$ running into $p_i$ has the exponent $-\Lambda^i_2$ while a small flat section $s^{i+1}$ for the 
	vertex $p_{i+1}$ along the curve $\gamma_i$ running into $p_i$ has the exponent $\Lambda^{i+1}_1$.}
	\[
		\begin{split}
			\mathcal{X}_E&=-\frac{\left(s^{1}\wedge s^{2}\right)(z_2)\left(s^{3}\wedge s^{4}\right)(z_4)}
			{\left(s^{4}\wedge s^{1}\right)(z_2)\left(s^{2}\wedge s^{3}\right)(z_4)}
			=-\frac{\left(s^{1}\wedge s^{2}\right)(z_2)\left(s^{3}\wedge s^{4}\right)(z_4)}
			{\left(s^{4}\wedge s^{1}\right)(z_1)\left(s^{2}\wedge s^{3}\right)(z_3)}
			e^{-\int_{\mu_2} \text{tr}\left(\nabla(\zeta)\right)-\int_{\mu_4} \text{tr}\left(\nabla(\zeta)\right)}\\
			&=-\frac{e^{-\Lambda^1_2(z_{2})+\Lambda^2_1(z_{2})}(1+r_{12})e^{-\Lambda^3_2(z_{4})+\Lambda^4_1(z_{4})}(1+r_{34})}
			{e^{\Lambda^4_1(z_{1})-\Lambda^1_2(z_{1})}(1+r_{41})e^{\Lambda^2_1(z_{3})-\Lambda^3_2(z_{3})}(1+r_{23})}
			e^{-\int_{\mu_2} \text{tr}\left(\nabla(\zeta)\right)-\int_{\mu_4} \text{tr}\left(\nabla(\zeta)\right)}\\
			&=-(1+r_q)
			\cdot e^{\Lambda^1_2(z_{1})-\Lambda^1_2(z_{2})+\Lambda^2_1(z_{2})-\Lambda^2_1(z_{3})
			+\Lambda^3_2(z_{3})-\Lambda^3_2(z_{4})+\Lambda^4_1(z_{4})-\Lambda^4_1(z_{1})}
			e^{-\int_{\mu_2} \text{tr}\left(\nabla(\zeta)\right)-\int_{\mu_4} \text{tr}\left(\nabla(\zeta)\right)},
		\end{split}
	\]
	where the $r_{kl}$ denote the corresponding remainder terms and $r_q$ is the term emerging from all of the $r_{kl}$, which is thus, as of 
	Corollary \ref{wedge_eval}, exponentially suppressed in $R$. 
	
	The key step is now noticing that $\Lambda^i_j(z_k)-\Lambda^i_j(z_l)=\int_{\mu_{lk}}\lambda_j$ leads to the desired period integrals, which is a bit 
	concealed here as we work with the different $\lambda_j$, emerging from tr$\nabla(\zeta)\neq0$. So at first we obtain the following:
	\[
		\mathcal{X}_E=-(1+r_q)\exp\left(\int_{\mu_{21}}\lambda_2+\int_{\mu_{32}}\lambda_1+\int_{\mu_{43}}\lambda_2+\int_{\mu_{14}}\lambda_1\right)
		\exp\left(-\int_{\mu_{12}} \text{tr}\left(\nabla(\zeta)\right)-\int_{\mu_{34}} \text{tr}\left(\nabla(\zeta)\right)\right).
	\]
	To put these integrals together we note that $\lambda_1$ and $\lambda_2$ differ only by containing expressions of $a_1$ or $a_2$. So for the parts 
	involving $q^{1/2}$ we already obtain a period integral along a simple closed curve $\mu$ encircling the two zeroes of $q$, while we observe the 
	following cancellation for the rest
	\begin{equation}\label{exp_exp}
	\begin{split}
		\mathcal{X}_E&=-(1+r_q)\exp\left(-\frac{R}{\zeta}\oint_\mu q^{1/2}-R\zeta\oint_\mu \overline{q^{1/2}}\right)\\
		&\exp\left(\int_{\mu_{21}}\overline{a_2}d\overline{z}-a_2dz-\int_{\mu_{32}}\overline{a_1}d\overline{z}-a_1dz+\int_{\mu_{43}}\overline{a_2}d\overline{z}-a_2dz-\int_{\mu_{14}}\overline{a_1}d\overline{z}-a_1dz\right)\\
		&\exp\left(\int_{\mu_{12}} \overline{a_1}d\overline{z}+\overline{a_2}d\overline{z}-\int_{\mu_{12}}a_1dz+a_2dz
		+\int_{\mu_{34}} \overline{a_1}d\overline{z}-\overline{a_2}d\overline{z}-\int_{\mu_{34}}a_1dz+a_2dz\right)\\
		&=-(1+r_q)\exp\left(-\frac{R}{\zeta}\oint_\mu q^{1/2}-R\zeta\oint_\mu \overline{q^{1/2}}+\oint_{\mu}a_1dz-\overline{a_1}d\overline{z}\right).
	\end{split}
	\end{equation}
	
	As $q^{1/2}$ is holomorphic on $\mathcal{C}$, the integral does not depend on the specific curve but only on its homotopy class. For the term $a:=\overline{a_1}d\overline{z}-a_1dz$ we note that this corresponds to the entries of the Chern connection $A_0$ of $\overline{\partial}_E$ w. r. t. the limiting metric $h_0$, and this solves the decoupled Hitchin equations \ref{lim_equ}. Thus $A_0$ is a flat connection and this implies $da=0$, so this integral is also defined on the corresponding homotopy class and we obtain the period integrals on $\mathcal{C}$. The construction as outlined in section \ref{sec_hom} and described in more detail in \cite{Gaiotto.2013} guarantees that the signs match in the right way when comparing $\mu$ with $\gamma_E$ coming from a cycle in $H_1(\Sigma'_q;\mathbb{Z})$, leading to the desired result.
	\end{proof}

	As noted before the $\oint_{\gamma_E} q^{1/2}$ should be thought of as living on the spectral cover, where the tautological $1$-form is integrated along 
	$\gamma_E\in H_1(\Sigma'_q;\mathbb{Z})$. As \cite{Gaiotto.2013} already noted the period integrals would vanish if $p^{1/2}$ could be well defined on the 
	whole quadrilateral. But as we had to cut out a disc around the zeroes of $p$, i. e. the branch points of the spectral curve, we obtain some nonzero 
	value. Following the notation of GMN the periods are denoted $Z_\gamma:=\tfrac{1}{\pi}\oint_{\gamma}q^{1/2}$. They give "special coordinates" on the Hitchin Base 
	$\mathcal{B}'$ corresponding to its special Kähler metric. Similarly note that $\oint_{\gamma}a_1dz-\overline{a_1}d\overline{z}$ is imaginary-valued and 
	one can show 
	that the functions $\theta$ with $i\theta:=\oint_{\gamma}\overline{a_1}d\overline{z}-a_1dz$ for 
	$\gamma_E\in H_1(\Sigma'_q;\mathbb{Z})$ give coordinates of the torus fibers of the Hitchin fibration, thus completing the set of coordinates for the 
	semiflat metric in $\mathcal{M}'$.
	
	One might wonder at this point that the coordinates don't depend on $a_2$ anymore which was one of the two entries of the holomorphic structure from which the Chern connection of the Higgs pair was constructed. The vanishing of this term does not lead to a loss of information as the connections for two Higgs pairs only differ by a traceless matrix so all necessary information is encoded in one of the diagonal entries.

	Let us now proceed to show a property of $\mathcal{X}_E$, which will be important for the twistorial construction later on. 
	First we want to expand the definition of $\mathcal{X}_E$ to large $\zeta$, for which we first have to check whether the small flat sections can be 
	treated 
	in the same way as before for large $\zeta$. We do this by considering $\nabla\left(-1/\overline{\zeta}\right)$ for small $\zeta$. It is useful in this 
	context to switch the frame and consider the flatness equation $0=\nabla(\zeta) s$ in an $h$-unitary frame. So in the following the indices $h$ and $h_0$ 
	denote the local expression in the $h$-unitary and $h_0$-unitary frame (where the second one is the frame $F$ we used up until now).
	
		In this new frame $H_h=E_2$ locally and so $\varphi_h^{\ast_h}=\overline{\varphi_h}^t$ and 
		$A_{h_{\overline{z}}}^{\ast_h}=\overline{A_{h_{\overline{z}}}}^t$. 
		By construction the change of basis matrix from the $h$-unitary frame to the $h_0$-unitary frame is $g^{-1}$. So if $s_{h_0}$ is 
		the local expression of a solution of the flatness equation in the $h_0$-unitary frame, $s_h:=gs_{h_0}$ is the local expression in the $h$-unitary 
		frame. Now from
		\[
			g\circ\left(\nabla_{h_0} \right)\circ g^{-1}s_h=g\circ\left(\nabla_{h_0} s_{h_0}\right)=0
		\]
		it follows, that $s_h$ is a solution of the flatness equation in the $h$-unitary frame.
		
		We take the $\zeta$-dependence of $B_h$ into account and calculate $\overline{B_h}^{t}$ using the fact that $A$ is $h$-unitary:
		\[
			\begin{split}
			\overline{B_h}^{t}(\zeta,t)
			&=\left(-\frac{R}{\overline{\zeta}}\overline{\varphi}^t-\overline{A_{h_z}}^t\right)\overline{z}'
			+\left((-R\overline{\zeta})\overline{\varphi^{\ast_h}}^t-\overline{A_{h_{\overline{z}}}}^t\right)z'\\
			&=\left(-\frac{R}{\overline{\zeta}}\varphi^{\ast_h}+A_{h_{\overline{z}}}\right)\overline{z}'
			+\left((-R\overline{\zeta})\varphi+A_{h_{z}}\right)z'\\
			&=-B_u(-\tfrac{1}{\overline{\zeta}},t).
			\end{split}
		\]
		Let $s$ and $b$ denote the small flat sections for two parabolic points $p_1$ and $p_2$ joined by $\gamma$. From the Proposition \ref{indep_dec} 
		we know that they are linearly independent and thus form a fundamental system $X$ for the ODE given by 
		$B_h(\zeta)$:
		\[
			X=\begin{pmatrix}s^1_h&b^1_h\\s^2_h&b^2_h\end{pmatrix}.
		\]
		Combining the result above with a standard argument we obtain
		\[
		\begin{split}
			0&=\frac{d(XX^{-1})}{dt}=X'X^{-1}-X\left(X^{-1}\right)'=B_h(\zeta)XX^{-1}-X\left(X^{-1}\right)'\\
			&\Rightarrow \left(X^{-1}\right)'=-X^{-1}B_h(\zeta)\\
			&\Rightarrow \left(\overline{X^{-1}}^t\right)'=-\overline{B_h(\zeta)}^t\overline{X^{-1}}^t\\
			&\Rightarrow \left(\overline{X^{-1}}^t\right)'=B_h(-\frac{1}{\overline{\zeta}})\overline{X^{-1}}^t.
		\end{split}
		\]
		So we see that a fundamental system for $-\frac{1}{\overline{\zeta}}$ is given by
		\[
			\overline{X^{-1}}^t=\frac{1}{s^1_hb^2_h-s^2_hb^1_h}\begin{pmatrix}\overline{b^2_h}&-\overline{s^2_h}\\ -\overline{b^1_h}&\overline{s^1_h}\end{pmatrix}.
		\]
		Thus we see that the sections given by these rows of this matrix are flat sections for $-\tfrac{1}{\overline{\zeta}}$. We denote them by 
		$\tilde{s}:=\frac{1}{s^1_hb^2_h-s^2_hb^1_h}(-\overline{s^2_h},\overline{s^1_h})$ and 
		$\tilde{b}:=\frac{1}{s^1_hb^2_h-s^2_hb^1_h}(\overline{b^2_h},-\overline{b^1_h})$. We would like these sections to be the corresponding small flat 
		sections, so we have to check how they behave as $t\to\infty$. But to be able to compare them to the sections considered before we first 
		have to transform back to the $h_0$-unitary frame. Using the fact that the transform is given by $g^{-1}$ and using Lemma \ref{g_asymptotics} a short 
		calculation shows that $\tilde{s}$ decays exponentially for $\to\infty$, while $\tilde{b}$ grows. Thus we have the same structure as before and can 
		expand our choice of small flat sections to $-\tfrac{1}{\overline{\zeta}}$ for small $\zeta$. As before there is a complex $1$-dimensional subspace of 
		small flat sections and we may chose a unique one by demanding the a certain limit for $t\to\infty$ after extracting the leading term. Again this 
		choice is of no real importance as complex factors cancel in the final expression for $\mathcal{X}_E$.
		
		Now we have to remember that $s$ and $b$ are the small flat section for the 
		vertices $p_1$ and $p_2$ for some small $\zeta$ and a $\vartheta$-curve running from the $p_2$ into $p_1$. 
		We want $\tilde{s}$ to be the small flat section for some vertex. As explained above the behaviour is just right but we have to take care of the 
		correct domain in the $\zeta$-plane. By construction the 
		decoration at any parabolic point is only defined if $\zeta\in\mathbb{H}_\vartheta$. Now $-\frac{1}{\overline{\zeta}}$ is not in $\mathbb{H}_\vartheta$ 
		as it has exactly the opposite phase. But it is thus in the half space $\mathbb{H}_{\vartheta+\pi}$. Now it remains to be checked
		whether $\tilde{s}$ is small along a $ca$-trajectory with angle $\vartheta+\pi$. At first $ca$-trajectories are only defined up to 
		orientation, i. e. 
		$\gamma$ is also a $ca$-trajectory for $\vartheta+\pi$ as $e^{i(\vartheta+\pi)}=-e^{i\vartheta}$. But now the orientation is the opposite one if we 
		demand $q^{1/2}(\gamma')=-e^{i\vartheta} $ as before. Thus $\tilde{s}$ is now the small flat section for $p_2$ and $\tilde{b}$ for $p_1$, which corresponds to a change 
		of the decoration at both vertices, which is called \textit{pop} in \cite{Gaiotto.2013}. In the full picture for a 
		standard quadrilateral this amounts to switching all of the decorations, which is called an $\textit{omnipop}$. For our calculations we note that the 
		omnipop leads to switching all sections in the wedge products of $\mathcal{X}_E$, leading to four changes of the sign, which thus overall 
		cancel. We note that GMN use the equality under the omnipop as a quite powerful tool to determine the BPS spectrum, but here it will only be used now 
		to obtain the so called reality condition.
		
	\begin{theorem}\label{reality}
		For small enough $\zeta$ the $\mathcal{X}$ coordinates obey the reality condition
		\[
			\mathcal{X}_\gamma^{\vartheta}(\zeta)=\overline{\mathcal{X}_\gamma^{\vartheta+\pi}(-1/\overline{\zeta})}.
		\]
	\end{theorem}
	
	\begin{proof}
		
		With the considerations from above and the notation as for Theorem \ref{x_int_form} we obtain the following equality:
		\[
		\begin{split}
			&\mathcal{X}_E^{T(\vartheta+\pi)}(-1/\overline{\zeta})\\
			&=\frac{\begin{pmatrix}-\overline{s^1_2}\\ \overline{s^1_1}\end{pmatrix}\wedge \begin{pmatrix}-\overline{s^2_2}\\ \overline{s^2_1}\end{pmatrix}}{\overline{s^1\wedge s^2}^2}(z_2)
			\frac{\begin{pmatrix}-\overline{s^3_2}\\ \overline{s^3_1}\end{pmatrix}\wedge \begin{pmatrix}-\overline{s^4_2}\\ \overline{s^4_1}\end{pmatrix}}{\overline{s^3\wedge s^4}^2}(z_4)
			\left(\frac{\begin{pmatrix}-\overline{s^4_2}\\ \overline{s^4_1}\end{pmatrix}\wedge \begin{pmatrix}-\overline{s^1_2}\\ \overline{s^1_1}\end{pmatrix}}{\overline{s^1\wedge s^2}\overline{s^3\wedge s^4}}(z_2)
			\frac{\begin{pmatrix}-\overline{s^2_2}\\ \overline{s^2_1}\end{pmatrix}\wedge \begin{pmatrix}-\overline{s^3_2}\\ \overline{s^3_1}\end{pmatrix}}{\overline{s^1\wedge s^2}\overline{s^3\wedge s^4}}(z_4)\right)^{-1}\\
			&=\frac{\overline{s^1\wedge s^2}(z_4)\overline{s^3\wedge s^4}(z_2)}{\overline{s^4\wedge s^1}(z_2)\overline{s^2\wedge s^3}(z_4)}.
		\end{split}
		\]
		Now we can again use Liouville's formula to transport $s^1\wedge s^2$ from $z_4$ to $z_2$ and $s^3\wedge s^4$ from $z_2$ to $z_4$. The emerging 
		exponential factors then cancel an we obtain $\mathcal{X}_E^{T(\vartheta+\pi)}(-1/\overline{\zeta})=\overline{\mathcal{X}^{T(\vartheta)}(\zeta)}$.
		By construction
		\[
			\gamma_E^{\vartheta}=-\gamma_E^{\vartheta+\pi},
		\]
		and thus combined we obtain
		\[
			\mathcal{X}_\gamma^{(\vartheta)}(\zeta)=\overline{\mathcal{X}_\gamma^{\vartheta+\pi}(-1/\overline{\zeta})}.
		\]
	\end{proof}
	
	\begin{remark}\label{deg_case}
	Before concluding this section we may shortly consider the case of degenerate edges. As stated in remark \ref{no_deg_tri} we do not work this out in 
	detail here as we can mostly adapt the way this is done in \cite{Gaiotto.2013}. The basic idea is that in the case of a degenerate edge $E$ corresponding 
	to a degenerate triangle the corresponding quadrilateral also degenerates and one has to consider a covering surface on which the necessary cycles can be 
	defined. After appropriately matching the definitions one obtains cycles consisting of two loops running in opposite directions around the two lifts of a 
	vertex on the spectral cover. The 
	period integrals are then given by the residues at the vertex which are fixed for the moduli space.
	
	But GMN also show that in this case $\mathcal{X}_E=\mu^2$ where $\mu$ is the clockwise monodromy around the degenerate vertex, which they have calculated 
	from their Frobenius analysis (cf. the discussion following equation \eqref{GMN_way}). This is the only part in which we differ, as we have not obtained 
	the small flat section in this way. But our way works just as good, as we may also evaluate our small flat sections along a circle around the degenerate 
	vertex. Indeed, there is always some $\vartheta$ such that the $\vartheta$-curves are circles around the poles of $q$. We can thus simply use our 
	analysis just like before and obtain the monodromy data, which matches the one from GMN. One then sees directly by comparing with the residue data that 
	$\mathcal{X}$ again becomes dominated by the period integral without any further calculations.
	\end{remark}
	
	\subsection{Concerning the derivatives of the coordinates}\label{der_coord}
	
	In the last section of this part we want to check the logarithmic derivative of the $\wedge$-product as the hyperkähler metric will be build out of this. 
	After considering the possibility of differentiating at all the important result here will be, that the logarithmic derivative has only simple poles for 
	$\zeta\to0$ or $\zeta\to\infty$.
	To show this we first take a look at the derivative 
	$\partial_\epsilon$ of $x\wedge y$, where we differentiate with respect to some coordinate of the moduli space which is a parameter $\epsilon$ of the 
	ODE. Generally, 
	if the sections are differentiable we have the usual product rule
	\[
		\partial_\epsilon (x\wedge y)=(\partial_\epsilon x)\wedge y+x\wedge(\partial_\epsilon y).
	\]
	Note that we do not need to consider the variation of the frame as this can be taken to be locally constant in $\epsilon$ and taking if necessary a 
	global change of basis which cancels in the definition of $\mathcal{X}$. Thus we need to differentiate the solutions of the integral equation, but before we can do carry out the explicit calculations there is one rather technical point we have to address.
	
	\begin{remark}\label{diff_quotient}
		In this chapter we will assume that the coefficients of the local expressions of our objects, like the Higgs field, depend differentiably on differentiable coordinates of the moduli space. As long as we only considered the functions $\mathcal{X}$ at some point on the moduli space this was not important, but as we begin to describe the derivatives we are moving in the moduli space and one has to lift families of elements of the moduli spaces (i. e. the equivalence classes) to corresponding families of representatives in such a way, that the differentiable structure is preserved. The possibility of such a differentiable choice of representatives is basically built directly into the construction of the moduli space as a Hyperkähler quotient, which is why we won't go into the technical details here. Instead we refer to the thorough explanation of the process of building the quotient in \cite{Mayrand.2016}.
		
		But this choice of representatives only addresses the part of the story where the objects as global functions dependd differentiable on the coordinates. As we are going to work with the local expressions, we also have to consider how the differentiability survives, the process of evaluation in a local frame. As long as the frame is fixed for a curve of Higgs fields $\phi_t$ one obtains the differentiable dependence of the matrix coefficients as the evaluation at a point in a local frame defines a bounded linear operator. But in our construction in section \ref{frame_sec} we adapted the frame to the Higgs field by taking a frame consisting of eigenvectors of $\phi$. The $\mathcal{X}$-coordinates themselves are independent of any frame, but we have to make sure that our construction of the small flat sections from the corresponding integral equations works. Therefore it needs to be possible to chose the eigenvectors of $\phi_t$ in such a way that they vary differentiably with the entries of the local expressions of $\phi_t$ in the fixed frame. But this follows from a short calculation in linear algebra, so everything works out and we have the desired structure of a diagonal leading term plus some small error matrix. However, at this point we have only diagonalized the Higgs field $\phi_t$ and not scaled the frame to a standard form, as was also done in section \ref{frame_sec}. This should also be possible as all the transformations are differentiable, but we actually don't have to adapt the frames in this way. Indeed, by not adapting the frame to the limiting metric and holomorphic structure, we only obtain a slightly different expression of the leading term in the flatness equation and some additional finite multiplicative factors in the entries of the error matrix. As they are finite, they don't change the analytical properties of our integral equation at all. Therefore we can safely ignore these changes and continue to work with the form of our equation we used up until now.
		\end{remark}

	\begin{lemma}\label{derivative_bounded}
		The small flat sections are differentiable with respect to any differentiable coordinate $\epsilon$ of $\mathcal{M}_R$ and the derivative  of the remainder-term is bounded by some uniform constant for small enough $\zeta$.\footnote{For this statement to be precise one has to consider $R$-independent coordinates. Such coordinates can be obtained by considering $\mathcal{M}_R$ as the moduli space of Higgs bundles and such a choice will always be implicit when considering this property.}
	\end{lemma}
	
	\begin{proof}
		The leading term $e^{-\Lambda_2}$ is differentiable by construction. So what we have to consider is the remainder term. We want to use 
		Theorem \ref{different}, so we have to check that the kernel is of product type of second order w. r. t. a coordinate $\epsilon$ of the 
		moduli space $\mathcal{M}_R$. The kernel is of product type as was shown above, i. e. it is of the form $AB$ for
		\[
			A:=\begin{pmatrix}e^{\Lambda(t)-\Lambda(\tau)}&0\\0&1\end{pmatrix},\quad
			B:=\begin{pmatrix}e_1(\tau)&e_2(\tau)\\e_3(\tau)&-e_1(\tau)\end{pmatrix}.
		\]
		The derivative of $A$ is 
		\[
			\partial_\epsilon\begin{pmatrix}e^{\Lambda(t)-\Lambda(\tau)}&0\\0&1\end{pmatrix}
			=\begin{pmatrix}\partial_\epsilon(\Lambda(t)-\Lambda(\tau))e^{\Lambda(t)-\Lambda(\tau)}&0\\0&0\end{pmatrix}.
		\]
		
		So we have to consider $\partial_\epsilon\Lambda(t)$ where $\Lambda$ is a primitive function for $\lambda_1+\lambda_2$, i. e. for
		\[
			\lambda=\left(2\tfrac{R}{\zeta}e^{i\vartheta}-2i(\text{Im}(a_2z')+\text{Im}(a_1z'))+2R\zeta e^{-i\vartheta}\right).
		\]
		
		Now $a_1$ and $a_2$ are obtained directly from the connection $A$ of the Higgs pair and of the form $\tfrac{c_i(\epsilon)}{z}+f(\epsilon)$ where $c_i$ and $f_i$ 
		depend differentiably on 
		$\epsilon$. Thus the pole structure does not change and $(\partial_\epsilon a)z'$ is finite everywhere. The other part of the derivative is 
		$a_i\partial_\epsilon(z')$, so we need to consider $\partial_\epsilon(z')$. For this we need recall that $z(t)$ denoting the $\vartheta$-trajectory for 
		$\det\Phi=-q=fdz^2$ depends on $\epsilon$ in the way $q$ depends on $\epsilon$. We may now consider any initial value problem for the $\vartheta$-trajectories 
		to obtain the differentiability of $z'(t,\epsilon)$ w. r. t. $\epsilon$. For example when considering the derivative in $\epsilon_0$ we can consider 
		the initial value problem with $z'(0,\epsilon)=z'(0,\epsilon_0)$ for all $\epsilon$ near $\epsilon_0$ and thus obtain a well defined initial value 
		problem whose solution thus depends differentiably on $\epsilon$. For calculations we note that by definition $f^{1/2}z'\equiv$const. Thus
		\[
			\partial_\epsilon(-f^{1/2}z')=0,\quad\Rightarrow\quad \partial_\epsilon{z'}=\frac{\partial_\epsilon f^{1/2}}{f^{1/2}}z'.
		\]
		By a local description in a coordinate centered around a weakly parabolic point it is clear 
		that the fraction is finite and thus $\partial_\epsilon z'$ behaves like $z'$. 
		Consequently $a_i\partial_\epsilon(z')$ is finite and so is $\partial_\epsilon(\Im(a_iz'))$.
		
		For the derivative of $B$ we look at the derivative of the $e_i$. We consider 
		\[
			\partial_\epsilon e_2=\partial_\epsilon\left((a_1-a_2)z'+R\zeta2e^{-i\vartheta}\right)\delta\beta+\left((a_1-a_2)z'+R\zeta2e^{-i\vartheta}\right)
			\partial_\epsilon(\delta\beta)+\partial_\epsilon(h_1z')
		\]
		as similar arguments hold for $e_1$ and $e_3$. The considerations from before lead to the finiteness of the derivative of the part in the brackets for 
		the first term, which thus remains integrable. For the second term the bracket is still finite and we use the fact that $\partial_\epsilon(\delta\beta)$ 
		is integrable in the same way as is $\delta\beta$, which follows from Lemma \ref{g_asymptotics} in the same way as was shown for the kernel of the original 
		(non derived) equation in section \ref{sec_small_sec}. Finally the integrability of $\partial_\epsilon(h_1z')$ follows just as for $\partial_\epsilon(\Im(a_iz'))$ and using Lemma \ref{g_asymptotics}. Thus the kernel is of product type of order $(1,1)$. The same arguments now holds again for the 
		second derivative of $A$ and thus the kernel is of product type of order $(2,1)$. We conclude from Theorem \ref{different} that the derivative of $x$ 
		exists and is the unique solution to the following equation:
		\begin{equation}\label{derived_volterra}
			\begin{split}
				\partial_\epsilon x(t)&=-\int_{t}^{\infty}
					\partial_\epsilon\left(\begin{pmatrix}e^{\Lambda(t)-\Lambda(\tau)}&0\\0&1\end{pmatrix}
					\begin{pmatrix}e_1(\tau)&e_2(\tau)\\e_3(\tau)&-e_1(\tau)\end{pmatrix}\right)x(\tau)d\tau\\
					&-\int_{t}^{\infty}\begin{pmatrix}e^{\Lambda(t)-\Lambda(\tau)}&0\\0&1\end{pmatrix}
					\begin{pmatrix}e_1(\tau)&e_2(\tau)\\e_3(\tau)&-e_1(\tau)\end{pmatrix}\partial_\epsilon x(\tau)d\tau
			\end{split}
		\end{equation}
		Now as $x$ is bounded by some (uniform) constant and by the properties we have shown above it follows that the first integral
		is bounded by some constant $C$. The kernel of this integral equation remains the same as before, so we obtain just like before from Corollary 
		\ref{est} the estimate
		\[
			|\partial_\epsilon x(t)|\leq \frac{C}{1-Ke^{-\delta R}}.
		\]
		 Again this is bounded by a uniform constant for large enough $R$ which finishes the proof. 
	\end{proof}
	
	We have thus shown that the derivative of the remainder-terms exist and are bounded. This is all we need to obtain the behaviour of the $\wedge$-product that we need.
	
	\begin{theorem}
		Let $s^1$ and $s^2$ be the small flat section for two weakly parabolic points that are connected by a generic $\vartheta$-trajectory. Then the 
		logarithmic derivative of $s^1\wedge s^2$ has a simple pole at $\zeta=0$. Furthermore this pole is precisely the one coming from the 
		leading term.
	\end{theorem}
	
	\begin{proof}
		For the derivative of the full product we consider again the different primitives of the $\lambda_j$ and obtain
		\[
			\partial_\epsilon(s^1\wedge s^2)=\partial_\epsilon \left(e^{\Lambda_1^2-\Lambda_2^1}\right)(x\wedge y)
			+e^{\Lambda_1^2-\Lambda_2^1}\partial_\epsilon (x\wedge y)
			=\left(\partial_\epsilon (\Lambda_1^2-\Lambda_2^1)(x\wedge y)+\partial_\epsilon (x\wedge y)\right)e^{\Lambda_1^2-\Lambda_2^1}.
		\]
		Thus the logarithmic derivative has the form
		\[
			\frac{\partial_\epsilon(s^1\wedge s^2)}{s^1\wedge s^2}
			=\frac{\partial_\epsilon (\Lambda_1^2-\Lambda_2^1)(x\wedge y)+\partial_\epsilon (x\wedge y)}{x\wedge y}.
		\]
		Since $\partial_\epsilon (x\wedge y)$ is bounded and $x\wedge y=1+r$ for some bounded and exponentially suppressed $r$ it follows that there is only 
		the simple pole at $\zeta=0$ coming from the $\Lambda$ part.
	\end{proof}
	
	This boundedness already suffices to prove the necessary condition for our $\mathcal{X}$-coordinate. Nevertheless before showing that we'll proof that 
	$\partial_\epsilon (x\wedge y)$ is even exponentially suppressed as we'll need that later for the asymptotics of the metric and the corresponding 
	calculations are similar to the once here.
	
	\begin{theorem}\label{derived_wedge}
		Let $s^1$ and $s^2$ be the small flat section for two weakly parabolic points that are connected by a generic $\vartheta$-trajectory with remainder-terms $x$ 
		and $y$. Then the product $\partial_\epsilon(x\wedge y)$ is exponentially suppressed in $R$.
	\end{theorem}
	
	\begin{proof}
		We look again at equation \eqref{derived_volterra} for $\partial_\epsilon x(t)$. As $\partial_\epsilon x(t)$ is bounded by Lemma 
		\ref{derivative_bounded}, plugging it back into the right hand side and using the exponential bound of the kernel as in the proof of Proposition 
		\ref{rest_form} we 
		see that the second integral is some vector $v$ whose norm is bounded by some exponentially suppressed $K$.
		
		Now we consider again the "initial value function" of the derivative:
			\begin{align*}
				\int_{t}^{\infty}
					\partial_\epsilon\left(\begin{pmatrix}e^{\Lambda(t)-\Lambda(\tau)}&0\\0&1\end{pmatrix}\right)
					&\begin{pmatrix}e_1(\tau)&e_2(\tau)\\e_3(\tau)&-e_1(\tau)\end{pmatrix}x(\tau)d\tau\\
				&+\int_{t}^{\infty}
					\begin{pmatrix}e^{\Lambda(t)-\Lambda(\tau)}&0\\0&1\end{pmatrix}
					\partial_\epsilon\left(\begin{pmatrix}e_1(\tau)&e_2(\tau)\\e_3(\tau)&-e_1(\tau)\end{pmatrix}\right)x(\tau)d\tau\\
				=\int_{t}^{\infty}\partial_\epsilon(\Lambda(t)-\Lambda(\tau))e^{\Lambda(t)-\Lambda(\tau)}
				&\begin{pmatrix}e_1(\tau)&e_2(\tau)\\0&0\end{pmatrix}x(\tau)d\tau\\
				&+\int_{t}^{\infty}
					\begin{pmatrix}
					e^{\Lambda(t)-\Lambda(\tau)}\partial_\epsilon e_1(\tau)&e^{\Lambda(t)-\Lambda(\tau)}\partial_\epsilon e_2(\tau)\\ \partial_\epsilon e_3
					(\tau)&-\partial_\epsilon e_1(\tau)
					\end{pmatrix}x(\tau)d\tau.
			\end{align*}
		For the first term the $e_i$ imply, as mentioned above, that the whole integral is given by some function $\begin{pmatrix}g\\0\end{pmatrix}$, which is
		exponentially suppressed by some $C_1e^{-\delta_1R}$ for an appropriate constants $C_1$ and $\delta_1$.
		
		For the second term we have to take a closer look at the derivative. Remember that $x=(0,1)^t+x^r$ where $x^r$ is exponentially suppressed. We already 
		saw that $\partial_\epsilon e_i$ is finite (and linear in $R$), so the only part that may not be exponentially suppressed in $R$ is the one coming from 
		$(0,1)^t$, which is
		\[
			\int_{t}^{\infty}
					\begin{pmatrix}e^{\Lambda(t)-\Lambda(\tau)}\partial_\epsilon e_2(\tau)\\-\partial_\epsilon e_1(\tau)\end{pmatrix}d\tau.
		\]
		Here we look at the second entry and remember that in 
		\[
			e_1=2(a_1-a_2)|\beta|^2z'+R\zeta2e^{-i\vartheta}|\beta|^2=\left(2(a_1-a_2)z'+R\zeta2e^{-i\vartheta}\right)|\beta|^2+h_1z'
		\]
		the suppressing factor $|\beta|$ comes into play quadratically and so they contain the derivative
		\[
			\partial_\epsilon |\beta|^2=2|\beta|\partial_\epsilon|\beta|.
		\]
		As $\partial_\epsilon|\beta|$ is finite by the argument above the dependence on $|\beta|$ leads to the first part still being exponentially suppressed. Additionally from Lemma \ref{g_asymptotics} it follows that $\partial_\epsilon (h_1z')$ is exponentially suppressed in the right way and thus the whole term is, so
		
		\[
			\int_{t}^{\infty}
					\begin{pmatrix}e^{\Lambda(t)-\Lambda(\tau)}\partial_\epsilon e_2(\tau)\\-\partial_\epsilon e_1(\tau)\end{pmatrix}d\tau
					=\begin{pmatrix} f_1\\ f_2\end{pmatrix}
		\]
		for some bounded function $f_1$ and a function $f_2$ bounded by $C_2e^{-\delta_2 R}$ for appropriate $C_2$ and $\delta_2$. So the only term not being 
		exponentially suppressed is the first component. Luckily for the $\wedge$-product the first term is undone by the leading part of $y$ and the remainder-term 
		$y_r$ of $y$ is itself exponentially suppressed. 
		So all together we obtain 
		\[
			(\partial_\epsilon x)\wedge y=\left(\begin{pmatrix}f_1+g\\f_2\end{pmatrix}+v\right)\wedge\left(\begin{pmatrix}1\\0\end{pmatrix}+y^r\right)
			=f_2+\begin{pmatrix}f_1+g\\f_2\end{pmatrix}\wedge y_r+v\wedge\left(\begin{pmatrix}1\\0\end{pmatrix}+y^r\right).
		\]
		So with the deduced bounds we obtain
		\[
			|(\partial_\epsilon x)\wedge y|\leq |f_2|+\left|\begin{pmatrix}f_1+g\\f_2\end{pmatrix}\right||y_r|+|v|(1+|y^r|)\leq Ce^{-\delta R}.
		\]
		The corresponding calculations also hold for $x\wedge(\partial_\epsilon y)$ and thus also for $\partial_\epsilon(x\wedge y)$.
	\end{proof}
	
	As a final step we now have to transport the derivative of $x$ to the other sides of the the quadrilateral as we have done in Proposition 
	\ref{connecting_coord} for $x$ itself. Here we use the same techniques and estimates as before:
	
	\begin{proposition}
		Let $\gamma_1$ and $\gamma_2$ be two sides of the standard quadrilateral $Q$ joining in the weakly parabolic point $p$ of $Q$ and let $\tilde{p}$ be the other 
		end of $\gamma_2$. Let $x$ denote the remainder term of the small flat section along $\gamma_1$ and $y$ the remainder term of the small flat section for 
		$\tilde{p}$. Then the product $\partial_\epsilon(x\wedge y)$ on $\gamma_2$ is exponentially suppressed in $R$.
	\end{proposition}
	
	\begin{proof}
		Just as for the small flat section along the first $\vartheta$-trajectory, we obtain the following equation for the derivative along the connecting 
		circle segment described in Proposition \ref{connecting_coord}:
		\[
		\begin{split}
			x_\epsilon(t)=x_\epsilon(a)
			&+\int_a^t\partial_\epsilon\left(\begin{pmatrix}e^{\Lambda(t)-\Lambda(\tau)}&0\\0&1\end{pmatrix}\begin{pmatrix}e_1&e_2\\e_3&-e_1\end{pmatrix}\right)
			x(\tau)d\tau\\
			&+\int_a^t\begin{pmatrix}e^{\Lambda(t)-\Lambda(\tau)}&0\\0&1\end{pmatrix}\begin{pmatrix}e_1&e_2\\e_3&-e_1\end{pmatrix}x_\epsilon(\tau)d\tau.
		\end{split}
		\]
		The validity of this equation follows via the methods analogous to the ones developed in the first part of this paper. From Proposition 
		\ref{connecting_coord} we have the form $x(\tau)=(0,1)+x^r$ for some exponentially suppressed $x^r$ just as for the function along the first side. Thus 
		the same argument as in Lemma \ref{derivative_bounded} may be used to obtain an uniform bound for the first integral. But then the equation for 
		$x_\epsilon$ differs only by this uniform bound from the equation for the non derived $x$ as we transported it to the other side in 
		\ref{connecting_coord}. Thus the same arguments as in that proof can be used to obtain the limits when going into the weakly parabolic point. It 
		follows again that the values on the two sides of the quadrilateral only differ by an exponentially suppressed term plus the limit of the first 
		integral, which has the same structure as the "initial value integral" in the proof of Theorem \ref{derived_wedge}. When taking the wedge product with 
		$y$ we may then argue in the same way as in the proof of Theorem \ref{derived_wedge}, the only difference being an additional exponentially suppressed 
		term. Thus we conclude that the derivatives of the transported sections behave in the same way as the non transported ones.
	\end{proof}
	
	Now we have all the information we need for our next main result concerning the $\mathcal{X}$-coordinates:
	
	\begin{theorem}\label{pole_0}
		The logarithmic derivative of any $\mathcal{X}$-coordinate has at most a simple pole $\zeta=0$.
	\end{theorem}

	\begin{proof}
		First note that all factors in the expression for $\mathcal{X}$ are nonzero and so is $\mathcal{X}$ itself.	By the product rule we have:
		\[
			\begin{split}
				\partial_\epsilon\mathcal{X}&=\partial_\epsilon \left(\frac{\left(s^{1}\wedge s^{2}\right)\left(s^{3}\wedge s^{4}\right)}
				{\left(s^{4}\wedge s^{1}\right)\left(s^{2}\wedge s^{3}\right)}\right)\\
				&=\frac{\left(\partial_\epsilon\left(s^{1}\wedge s^{2}\right)\right)\left(s^{3}\wedge s^{4}\right)\left(s^{4}\wedge s^{1}\right)\left(s^{2}\wedge s^{3}\right)+\left(s^{1}\wedge s^{2}\right)\left(\partial_\epsilon\left(s^{3}\wedge s^{4}\right)\right)\left(s^{4}\wedge s^{1}\right)\left(s^{2}\wedge s^{3}\right)}{\left(s^{4}\wedge s^{1}\right)^2\left(s^{2}\wedge s^{3}\right)^2}\\
				&-\frac{\left(s^{1}\wedge s^{2}\right)\left(s^{3}\wedge s^{4}\right)\left(\partial_\epsilon\left(s^{4}\wedge s^{1}\right)\right)\left(s^{2}\wedge s^{3}\right)+\left(s^{1}\wedge s^{2}\right)\left(s^{3}\wedge s^{4}\right)\left(s^{4}\wedge s^{1}\right)\partial_\epsilon\left(s^{2}\wedge s^{3}\right)}{\left(s^{4}\wedge s^{1}\right)^2\left(s^{2}\wedge s^{3}\right)^2}\\
				&=\frac{\left(\partial_\epsilon\left(s^{1}\wedge s^{2}\right)\right)\left(s^{3}\wedge s^{4}\right)}{\left(s^{4}\wedge s^{1}\right)\left(s^{2}\wedge s^{3}\right)}
				+\frac{\left(s^{1}\wedge s^{2}\right)\partial_\epsilon\left(s^{3}\wedge s^{4}\right)}{\left(s^{4}\wedge s^{1}\right)\left(s^{2}\wedge s^{3}\right)}\\
				&-\frac{\left(s^{1}\wedge s^{2}\right)\left(s^{3}\wedge s^{4}\right)\partial_\epsilon\left(s^{4}\wedge s^{1}\right)}{\left(s^{4}\wedge s^{1}\right)^2\left(s^{2}\wedge s^{3}\right)}
				-\frac{\left(s^{1}\wedge s^{2}\right)\left(s^{3}\wedge s^{4}\right)\partial_\epsilon\left(s^{2}\wedge s^{3}\right)}{\left(s^{4}\wedge s^{1}\right)\left(s^{2}\wedge s^{3}\right)^2}.
			\end{split}
		\]
		Thus we obtain for the logarithmic derivative
		\[
			\begin{split}
				\partial_\epsilon\log(\mathcal{X})=\frac{\partial_\epsilon(\mathcal{X})}{\mathcal{X}}
				=\frac{\partial_\epsilon\left(s^{1}\wedge s^{2}\right)}{s^{1}\wedge s^{2}}
				+\frac{\partial_\epsilon\left(s^{3}\wedge s^{4}\right)}{s^{3}\wedge s^{4}}
				-\frac{\partial_\epsilon\left(s^{4}\wedge s^{1}\right)}{s^{4}\wedge s^{1}}
				-\frac{\partial_\epsilon\left(s^{2}\wedge s^{3}\right)}{s^{2}\wedge s^{3}}.
			\end{split}
		\]
		By the calculations above all terms 
		in the final expression have only the simple pole at $\zeta=0$ coming from the leading term which thus is also true for 
		$\partial_\epsilon\log(\mathcal{X})$.
	\end{proof}
		
	We now also need the same structure for $\zeta=\infty$, which could naturally be guessed as of the symmetry in the construction of $\nabla$. The precise 
	result could be obtained by working in an anti-holomorphic frame (i. e. a frame, where $\varphi^\ast$ is diagonal) and repeating all the steps up until 
	now, where one would have to consider some additional term in the ODE coming from the change of basis and its $z$-derivative. Though the information 
	obtained from \cite{Fredrickson.2020} is good enough to also make this work, it also follows directly from the reality condition, in the proof of which we 
	already used the 
	stated symmetry of $\nabla$. This is also the way in which \cite{Tulli.2019} derives this asymptotic, although he differs in the proof of the reality 
	condition for his case.
	
	\begin{corollary}\label{pole_infty}
		The logarithmic derivative of $\mathcal{X}$ has at most a simple pole at $\zeta=\infty$.
	\end{corollary}
	
	\begin{proof}
		This follows directly from the asymptotics for $\zeta\to0$ from Theorem \ref{pole_0} and the reality condition Theorem \ref{reality}.
	\end{proof}
	
	\newpage
	
\section{A Riemann-Hilbert problem}\label{sec_R_H}

	By now it would already be possible to construct the hyperkähler metric and even infer an exponential decay towards the semiflat metric. However we will not do this directly and instead reformulate the whole construction in terms of a Riemann-Hilbert problem in section \ref{Riem_Hill_form}. This approach was also already introduced in the general construction in \cite{Gaiotto.2010}. The idea in the general case is that even when there is no explicit moduli space but instead only some general integrable system it is still possible to formulate a Riemann-Hilbert problem for $\mathcal{X}$-functions, whose solution should naturally fit in the twistorial construction of hyperkähler metrics. The main new ingredient for this are the (generalized) Donaldson-Thomas invariants $\Omega$ which determine the jumping behaviour of $\mathcal{X}$. In our special case of the Higgs bundle moduli space these invariants are explicitly known, and we can use the known behaviour of our constructed $\mathcal{X}$-coordinates to show that they do solve the appropriate Riemann-Hilbert problem in section \ref{Riem_Hill_sol}. We then use this to show that they solve a certain integral equation, which was also formulated in \cite{Gaiotto.2010}. The upshot of this additional work is that the integral equation allows us to obtain more detailed information about the asymptotic behaviour of the $\mathcal{X}$-coordinates.
	
	\subsection{Formulation of the Riemann-Hilbert problem}\label{Riem_Hill_form}
	
	In this (and the next) section we fix an element $(A,\Phi)$ of $\mathcal{M}'_R$ such that the determinant $q$ of the Higgs field is an element of a chamber in the Hitchin base. We consider the $\mathcal{X}_\gamma^{\vartheta}$-coordinates, now for all $\gamma$ in the charge lattice $\Gamma_q$, and we eliminate the $\vartheta$-dependence of the coordinates by specializing $\mathcal{X}_\gamma^{\vartheta}$ to $\vartheta=$arg$(\zeta)$, i. e.:
\[
	\mathcal{X}_\gamma(\zeta):=\mathcal{X}^{\text{arg}\left(\zeta\right)}_\gamma(\zeta).
\]
As we work with a fixed element in $\mathcal{M}'_R$ we will suppress the dependence of $\mathcal{X}_\gamma$ on the elements of $\mathcal{M}'_R$. So we only keep the dependence on $\zeta$ and $\mathcal{X}_\gamma$ becomes a map from $\mathbb{C}^\times$ to $\mathbb{C}^\times$. The important thing to note here, is that the definition above leads to discontinuous jumps of the $\mathcal{X}_\gamma$-coordinates at certain values of $\zeta$ for which the $ca$-triangulation jumps. As this only depends on arg$(\zeta)$ we have jumps along rays in the $\zeta$-plane, which are called \textit{active BPS rays}. It is important to note that although the jumps stem from changes in the triangulations, the sides of the standard quadrilaterals which are built from the triangulation still vary continuously when the jumps occur. Thus it is possible to analytically extend each $\mathcal{X}_\gamma$ across a BPS ray by simply fixing the corresponding product of small flat sections. In the following for any BPS ray $l$ we denote the analytic continuation of any $\mathcal{X}_\gamma$ in the clockwise and counter-clockwise direction across $l$ with $\mathcal{X}_\gamma^-$ and $\mathcal{X}_\gamma^+$.

As we explained in section \ref{spec_diff}, in the case of an unsealed spectrum there are only two types of jumps, for which the coordinates actually change. They are characterized by the appearance of finite $ca$-trajectories which are saddle connections (connecting two zeroes) or periodic, i. e. closed loops. Crucially in both types of jumps the changes of the triangulation are known explicitly and GMN were able to obtain corresponding formulas for these jumps. In particular there is always an associated charge $\beta\in\Gamma_q$ corresponding to the jump. To formulate the jumps we define the quadratic refinement $\sigma:\Gamma_q\to\left\{-1,1\right\}$ and (generalized) Donaldson-Thomas invariants $\Omega:\Gamma_q\to\mathbb{Z}$ as follows. 

\begin{definition}
Let $\beta\in\Gamma_q$ be a charge such that there occurs a jump in the $ca$-triangulation corresponding to $\beta$. Then
\[
	\begin{split}
	\Omega(\beta;q)&:=\begin{cases}1,\quad \text{ if $\beta$ corresponds to a saddle connection},\\ -2,\quad \text{if $\beta$ corresponds to a closed loop}.\end{cases}
	\end{split}
\]
Additionally define $\Omega(\beta;q):=0$ for all other $\beta\in\Gamma_q$.
\end{definition}

\begin{remark}
	There is a different way of introducing the specified values of $\Omega$ (see e. g. \cite{Neitzke.2013}), where one counts for $\beta\in\Gamma_q$ the number of saddle connections that lift to $\beta$. Let this number be $SC$. Additionally one counts the set of isotopy classes of closed loops that lift to $\beta$ and denotes their number by $CL$. Now one can define
	\[
		\Omega(\beta;q)=SC-2\cdot CL.
	\]
	Generically it should be the case that this amounts to $\Omega$ being $0$, $1$ or $-2$ according to the definition above as no two different finite $ca$-trajectories may appear that lift to the same $\beta$. The results in \cite{Bridgeland.2015} support this partially at least in the case of quadratic differentials with double poles and simple zeroes considered here. Still, there may be possibilities left, where $\Omega$ thus defined gives different values from before. But we are interested in the jumping behaviour of the coordinates which was checked in \cite{Gaiotto.2013} for the values obtained from the first definition. So we keep that as the definition and postpone the question of whether this coincides with the count of finite $ca$-trajectories.
\end{remark}
 
For the quadratic refinement we proceed quite similarly 

\begin{definition}
Let $\beta$ be a generator of $\Gamma_q$ defined as in section \ref{sec_hom} (but now without demanding that they correspond to a jump in the triangulation as for $\Omega$). We then define
\[
	\begin{split}
	\sigma(\beta)&:=\begin{cases}-1,\quad \text{ if $\beta$ corresponds to a saddle connection},\\ 1,\quad \text{if $\beta$ corresponds to a closed loop}.\end{cases}
	\end{split}
\]
The rest of the values is now obtained from $\sigma(0):=1$ and the formula
\[
	\sigma(\gamma)\sigma(\beta)=\sigma(\gamma+\beta)\sigma(0)(-1)^{\left\langle\gamma,\beta\right\rangle}.
\]
\end{definition}

As GMN show in \cite{Gaiotto.2013}, section 7.7 $\sigma$ thus defined is independent of $q$ and the angle of the triangulation.

We can now define the basic building block for the following theory. We start by defining for each $\gamma,\beta\in\Gamma_q$ a function $\mathcal{K}^{\gamma}_{\beta}$, given by
\[
\begin{split}
	\mathcal{K}^{\gamma}_{\beta}:&\mathbb{C}^\times\to\mathbb{C},\quad
	\zeta\mapsto\left(1-\sigma(\beta)\mathcal{X}^-_\beta(\zeta)\right)^{\left\langle\gamma,\beta\right\rangle}.
\end{split}
\]
Note that from our analysis in the sections above for $R$ large enough it holds $|\mathcal{X}_\beta|<1$, so $\mathcal{K}^{\gamma}_{\beta}(\zeta)\in\mathbb{C}^\times$ for all $\zeta\in\mathbb{C}^{\times}$. We also regard $\mathcal{K}^{\gamma}_{\beta}$ as acting on the space of $\mathcal{X}$-functions via multiplication $\mathcal{X}^-_\gamma\mapsto\mathcal{X}^-_\gamma\mathcal{K}^{\gamma}_{\beta}$. This point of view is more in line with the general setting where the $\mathcal{K}^{\gamma}_{\beta}$ are defined as automorphisms acting on the space of twisted unitary characters of $\Gamma_q$. But for us this setting becomes very concrete as we can regard $\mathcal{K}^{\gamma}_{\beta}(\zeta)$ simpy as a complex number.

Now for each $\beta\in\Gamma_q$ define the ray $l_\beta:=Z_\beta(q^{1/2})\mathbb{R}_-$. These are the \textit{BPS rays} mentioned before. Here they arise from the observation that a jump of the triangulation is accompanied by the appearance of a finite arg$(\zeta)$-trajectory $c$. They are only defined up to orientation, but their lift to the spectral cover is canonically oriented as mentioned in section \ref{sec_hom}, from which we obtain the corresponding charge, say $\delta$. From Lemma \ref{angle_prop} we know that $Z_\delta(q^{1/2})\in\mathbb{R}_+e^{i\text{arg}(\zeta)}$. On the other hand as of the unoriented nature of the trajectory $c$, it is also a trajectory for arg$(\zeta)+\pi$, leading to the charge $-\delta$. For this charge it holds $Z_{-\delta}(q^{1/2})\in\mathbb{R}_-e^{i\text{arg}(\zeta)}$, i. e. $\zeta\in Z_{-\delta}(q^{1/2})\mathbb{R}_-$. The calculations in \cite{Gaiotto.2013} now show that it is the charge $\beta:=-\delta$ which determines the jump of a coordinate $\mathcal{X}_{\gamma}$ with $\left\langle\gamma,\beta\right\rangle\neq0$ at $l_{\beta}$ via the formula above. Thus we see that a jump can only occur when $\zeta$ crosses a BPS ray. We stress that this means that at the ray $l_\beta$ a jump of a coordinate $\mathcal{X}_\gamma$ occurs, while the coordinate $\mathcal{X}_\beta$ is continuous.
The choice of direction for $l$ will also be important later on, when we need the decaying behaviour of the coordinates along the BPS ray. 

Now one defines for each ray $l$ in the $\zeta$-plane the \textit{jump-factors} as the products of all the corresponding factors for which $l_\beta=l$:
\begin{equation}\label{jumps}
	S^\gamma_l(\zeta):=\prod_{\beta:l_\beta=l}{\left(\mathcal{K}^\gamma_\beta(\zeta)\right)^{\Omega(\beta;q)}}.
\end{equation}
This definition is taken from the general theory of GMN. As we will only consider these transformations inside the chambers of the Hitchin Base (cf. Definition \ref{chambers}) only one factor remains for each BPS ray. Therefore we don't have to worry about a particular order in the product, which could otherwise have infinite contributions.\footnote{For a way of ordering the product in the general case see section 2.3 in \cite{Gaiotto.2013}.}

A BPS ray $l_\beta$ for which $\Omega(\beta;q)\neq0$ is called \textit{active}. The values for $\sigma$ and $\Omega$ are determined from the considerations of GMN mentioned above, so that the $S_l$ describe the jumps at any BPS ray. From section \ref{sec_coord} we recall that the $\mathcal{X}$-coordinates have an explicitly known leading term given by the periods $Z_\gamma$ and $\theta_\gamma$. We denote this leading term (i. e. the function we obtain in equation \ref{exp_exp} by setting $r_q=0$) as $\mathcal{X}^{\text{sf}}$. With these definitions together with the properties of the Fock-Goncharov coordinates the following theorem holds.

\begin{theorem}\label{R_H_problem}
	For any fixed point in a chamber of $\mathcal{M}_R'$ and $\gamma\in\Gamma_q$ with unsealed spectrum the canonical coordinate $\mathcal{X}_\gamma:\mathbb{C}^\times\to \mathbb{C}^\times$ has the following properties:
	\begin{enumerate}
		\item $\mathcal{X}_\gamma$ depends piecewise-holomorphically on $\zeta\in\mathbb{C}^\times$, with discontinuities only at the rays $l_\gamma(u)$ with $\Omega(\gamma;q)\neq0$.
		\item The limits $\mathcal{X}_\gamma^{\pm}$ of $\mathcal{X}_\gamma$ as $\zeta$ approaches any ray $l$ from both sides exist and are related by
		\[
			\mathcal{X}_\gamma^+=\left(S^\gamma_l\right)^{-1}\mathcal{X}_\gamma^-.
		\]
		\item $\mathcal{X}_\gamma(\zeta)/\mathcal{X}_\gamma^{\text{sf}}(\zeta)$ is bounded near $\zeta=0$ and $\zeta=\infty$.
	\end{enumerate}
	
\end{theorem}

We note that property 2 was established by GMN, while the properties 1 and 3 have been rigorously proven in the earlier chapters of this thesis. These three properties together with the property that the value of $\mathcal{X}_\gamma$ is never $0$ are one way of describing a Riemann-Hilbert problem. Classically such a problem is given by a contour in the complex plane and a specified jumping behaviour, cf. \cite{Its.2003}, \cite{Mushkhelishvili.1953} and \cite{Fokas.2006}. One then wants to find a function that is holomorphic in the complement of the contour and jumps in the prescribed way. Here the problem is more complicated as the jumping behaviour is determined by the functions themselves. Additionally the classical case defines jumps by the action of (piecewise) constant matrices, so they are elements of a finite-dimensional space, whereas the jumps here live in the infinite-dimensional space of functions on $\mathbb{C}^\times$. Also note that the third property here differs from the usual formulation where a specific limit in $0$ or $\infty$ is specified, which leads to the uniqueness of solutions. This will be described in more detail below.

	\subsection{Integral equations and solutions}\label{Riem_Hill_sol}

Although such Riemann-Hilbert problems are interesting in themselves, here we consider it to obtain another representation for our coordinates, that allows for a better estimate of the decaying properties. To this effect GMN argued that a solution of the Riemann-Hilbert problem is given be the solution of a certain integral equation. This equation is mostly encountered in the following form:
\begin{equation}
	\mathcal{X}_\gamma(\zeta)=\mathcal{X}_\gamma^{\text{sf}}(\zeta)\text{exp}\left[-\frac{1}{4\pi i}\sum_{\beta\in\Gamma^g_q}\Omega(\beta;q)\left\langle\gamma,\beta\right\rangle\int_{l_{\beta}(q)}\frac{d\zeta'}{\zeta'}\frac{\zeta'+\zeta}{\zeta'-\zeta}\text{log}\left(1-\mathcal{X}_\beta(x,\zeta')\right)\right].
\end{equation}

Note that here an analytic continuation of the coordinates is implicit as they jump precisely at the rays of integration.

In this form the equation is considered as an actual equation for all of the $\mathcal{X}_\gamma$. The approach to solving this equation in \cite{Gaiotto.2013} is by iteration, starting with $\mathcal{X}_\gamma^{\text{sf}}$ as initial approximation. It remains to be shown that this process converges to a solution, which is then supposed to solve the Riemann-Hilbert problem. In our setting the situation is different as we already constructed a solution of the Riemann-Hilbert problem and we want to show that this solution solves the integral equation. To show this we use the fact that the equation is specifically tailored for solving the equation in the first place. This becomes clearer when writing the equation in the following form, which was the starting point in \cite{Gaiotto.2013}:

\begin{equation}
	\mathcal{X}_\gamma(\zeta)=\mathcal{X}_\gamma^{\text{sf}}(\zeta)\text{exp}\left[\frac{1}{4\pi i}\sum_l\int_{l}\frac{d\zeta'}{\zeta'}\frac{\zeta'+\zeta}{\zeta'-\zeta}\text{log}\left(\frac{\mathcal{X}_\gamma(\zeta')}{S_l\mathcal{X}_\gamma(\zeta)}\right)\right].
\end{equation}

Now the sum runs over the BPS rays $l$. Of course in all these sums it is necessary that they converge appropriately. This is clear for a finite amount of jumps, but as we have seen before, infinitely many jumps may occur, when a ring domain appears. We will deal with this below.

Although not difficult it is enlightening so see how the two formulas relate. Consider the log-part in the second formula for some BPS ray $l$. From the definition (and suppressing the $\zeta'$-dependence) it can be expanded as 
\[
	\begin{split}
	\text{log}\left(\frac{\mathcal{X}_\gamma}{S_l\mathcal{X}_\gamma}\right)
	&=\text{log}\left(\mathcal{X}_\gamma\right)-\text{log}\left(S_l\mathcal{X}_\beta\right)
	=\text{log}\left(\mathcal{X}_\gamma\right)-\text{log}\left(\prod_{\beta:l_\beta=l}{\mathcal{K}_\beta^{\Omega(\beta;q)}}\mathcal{X}_\gamma\right)\\
	&=\text{log}\left(\mathcal{X}_\gamma\right)-\text{log}\left(\mathcal{X}_\gamma\prod_{\beta:l_\beta=l}(1-\sigma(\beta)\mathcal{X}_\beta)^{\left\langle\gamma,\beta\right\rangle\Omega(\beta;q)}\right)
	=-\text{log}\left(\prod_{\beta:l_\beta=l}(1-\sigma(\beta)\mathcal{X}_\beta^{\left\langle\gamma,\beta\right\rangle\Omega(\beta;q)}\right)\\
	&=-\sum_{\beta:l_\beta=l}\left\langle\gamma,\beta\right\rangle\Omega(\beta;q)(1-\sigma(\beta)\mathcal{X}_\beta).
	\end{split}
\]
By taking the sum over all BPS rays $l$ we obtain the integral $\int_{l_\beta}$ for every BPS ray $l_\beta$ with $\Omega(\beta;q)\neq0$. Conversely the condition $\Omega(\beta;q)\neq0$ is only fulfilled for $\beta\in\Gamma^g_q$ which correspond to some BPS ray $l_\beta$. Thus the sum in the expression above for all BPS rays $l$ equals exactly the sum over all $\beta\in\Gamma^g_q$ as in the first equation.


Our method of showing that our constructed set of coordinates solves the integral equation can be summarized as follows. First we define an a priori new set of functions $\mathcal{Y}_\gamma$ via the integral equation, where we use the existence of the $\mathcal{X}_\gamma$-coordinates. We then show that these new functions are again a solution of the Riemann Hilbert problem. Finally a uniqueness argument ensures that the solutions have to coincide up to some scalar independent of $\zeta$. Analyzing this scalar will then be the last task to obtain the desired decay rates.

Let us start be defining the "new" set of functions.

\begin{definition}
Let $\mathcal{X}_\gamma$ (for all $\gamma\in\Gamma_q$) denote the constructed set of coordinates that solve the Riemann Hilbert problem \ref{R_H_problem}. The corresponding \textit{integral coordinates} are defined as
\[
\mathcal{Y}_\gamma:=\mathcal{X}_\gamma^{\text{sf}}(\zeta)\text{exp}\left[\frac{1}{4\pi i}\sum_l\int_{l}\frac{d\zeta'}{\zeta'}\frac{\zeta'+\zeta}{\zeta'-\zeta}\text{log}\left(\frac{\mathcal{X}^+_\gamma(\zeta')}{S_l\mathcal{X}^+_\gamma(\zeta)}\right)\right].
\]
\end{definition}

In our proof that these function are solutions to the Riemann Hilbert problem we will already need the main estimates which will later on give us the optimal decay rate. So we have to make a lengthy calculation now.

\begin{theorem}\label{decay_estimate}
	For every $\gamma\in\Gamma_q$ there exist constants $C_{\beta,n}$, which are independent of $\zeta,R$ and the coordinates of the moduli space, such that the following estimate holds in the large $R$ limit:
	\[
	\left\|\log\left(\frac{\mathcal{Y}_\gamma}{\mathcal{X}^{\text{sf}}_\gamma}\right)\right\|\leq\sum_\beta\Omega(\beta;u)\left\langle\gamma,\beta\right\rangle\sum_{n>0}C_{\beta,n}\frac{1}{\sqrt{4\pi nR|Z_\beta|}}e^{-2n\pi R|Z_\beta|}.
	\]
\end{theorem}

\begin{proof}
	We want to use our knowledge that for any $\beta\in\Gamma_q$ the coordinates are of the form $\mathcal{X}_\beta=\mathcal{X}^{\text{sf}}_\beta(1+\kappa)$ where $\kappa$ is exponentially bounded to estimate each integral
	\[
		I_{l_\beta}:=-\int_{l_\beta}\frac{d\zeta'}{\zeta'}\frac{\zeta'+\zeta}{\zeta'-\zeta}\log\left(1-\mathcal{X}_\beta(x,\zeta')\right).
	\] 
	Along $l_\beta$ we have $\zeta'=-Z_\beta s$ for $s\in\mathbb{R}_+$, so
	\[
	\mathcal{X}_{\beta}^{\text{sf}}(\zeta)=\exp\left(-\frac{R}{s}\pi+i\theta_\beta-R\pi|Z_\beta|^2s\right)
	=\exp\left(-|Z_\beta|R\pi\left(\frac{1}{|Z_\beta|s}+|Z_\beta|s\right)+i\theta_\beta\right).
	\]
For $R$ large enough $\mathcal{X}_\beta<1$, so we expand $-\log(1-\mathcal{X}_\beta)$ as $\sum_{n=1}^\infty\tfrac{\mathcal{X}^n_\beta}{n}$ and obtain for each $n\geq1$ the integral
\[
	I_{l_\beta,n}:=\int_0^\infty\frac{1}{s}\frac{Z_\beta s-\zeta}{Z_\beta s+\zeta}\frac{\left(1+\kappa(-Z_\beta s)\right)^n}{n}\exp\left(-n|Z_\beta|R\pi\left(\frac{1}{|Z_\beta|s}+|Z_\beta|s\right)+in\theta_\beta\right)ds
\]
Following \cite{FilippiniS.A.GarciaFernandezM..2017} a change of variable $s=e^x$ leads to
\[
\begin{split}
	I_{l_\beta,n}&=\int_{-\infty}^{\infty}\frac{Z_\beta e^x-\zeta}{Z_\beta e^x+\zeta}\frac{\left(1+\kappa(-Z_\beta e^x)\right)^n}{n}\exp\left(-n|Z_\beta|R\pi\left(e^{-x-\log(|Z_\beta|)}+ e^{x+\log|Z_\beta|}\right)+in\theta_\beta\right)dx\\
	&=\int_{-\infty}^{\infty}\frac{Z_\beta e^x-\zeta}{Z_\beta e^x+\zeta}\frac{\left(1+\kappa(-Z_\beta e^x)\right)^n}{n}\exp\left(-2n|Z_\beta|R\pi\cosh\left(x+\log|Z_\beta|\right)+in\theta_\beta\right)dx.
\end{split}
\]
Another change of variable $y=x+\log|Z_\beta|$ then leads to:
\[
\begin{split}
	I_{l_\beta,n}&=\int_{-\infty}^{\infty}\frac{Z_\beta e^{y-\log|Z_\beta|}-\zeta}{Z_\beta e^{y-\log|Z_\beta|}+\zeta}\frac{\left(1+\kappa(-Z_\beta e^{y-\log|Z_\beta|})\right)^n}{n}e^{in\theta_\beta}\exp\left(-2n|Z_\beta|R\pi\cosh\left(y\right)\right)dy\\
	&=\int_{-\infty}^{\infty}\frac{\frac{Z_\beta}{|Z_\beta|} e^{y}-\zeta}{\frac{Z_\beta}{|Z_\beta|} e^{y}+\zeta}\frac{\left(1+\kappa(-\frac{Z_\beta}{|Z_\beta|} e^{y})\right)^n}{n}e^{in\theta_\beta}\exp\left(-2n|Z_\beta|R\pi\cosh\left(y\right)\right)dy.
\end{split}
\]
Now we write $\psi_\beta:=\text{arg}(Z_\beta)$ and obtain
\[
\begin{split}
	I_{l_\beta,n}&=\int_{-\infty}^{\infty}\frac{e^{y+i\psi_\beta}-\zeta}{e^{y+i\psi_\beta}+\zeta}\frac{\left(1+\kappa(-e^{y+i\psi_\beta})\right)^n}{n}e^{in\theta_\beta}\exp\left(-2n|Z_\beta|R\pi\cosh\left(y\right)\right)dy\\
	&=\int_{0}^{\infty}\frac{e^{y+i\psi_\beta}-\zeta}{e^{y+i\psi_\beta}+\zeta}\frac{\left(1+\kappa(-e^{y+i\psi_\beta})\right)^n}{n}e^{in\theta_\beta}\exp\left(-2n|Z_\beta|R\pi\cosh\left(y\right)\right)dy\\
	&+\int_{0}^{\infty}\frac{e^{-y+i\psi_\beta}-\zeta}{e^{-y+i\psi_\beta}+\zeta}\frac{\left(1+\kappa(-e^{-y+i\psi_\beta})\right)^n}{n}e^{in\theta_\beta}\exp\left(-2n|Z_\beta|R\pi\cosh\left(y\right)\right)dy,
\end{split}
\]
where we used the symmetry of $\cosh$ in the last equality. The first fraction is a bounded continuous function that depends on the distance of $\zeta$ to $l_\beta$, and $\kappa$ is exponentially bounded. Thus there exist constants $a,b\in\mathbb{R}_>0$ (independent of $R$, $\zeta$ and $n$) so that we have the following estimate:
\[
	\left\|I_{ln}\right\|\leq a\tfrac{b^n}{n}\int_{0}^\infty \exp\left(-2n\pi R|Z_\beta|\cosh\left(y\right)\right)dy=\left(a\tfrac{b^n}{n}\right)K_{0}\left(2n\pi R|Z_\beta|\right),
\]
where
 $K_0$ is the modified Bessel function of the second kind. The main asymptotic estimate is now given by $K_0(x)\sim\sqrt{\frac{\pi}{2x}}e^{-x}(1+\mathcal{O}(1/x))$. Summing over all $n$ we obtain the asymptotic estimate for the integral along $l_\beta$ as
\[
\begin{split}
	\left\|I_{l_{\beta}}\right\|&\leq\sum_{n>0}\left\|I_{l_\beta,n}\right\|\leq\sum_{n>0}a\tfrac{b^n}{n}K_{0}\left(2n\pi R|Z_\beta|\right)
	\sim \sum_{n>0}C_{\beta,n}\frac{1}{\sqrt{nR|Z_\beta|}}e^{-2n\pi R|Z_\beta|}.
\end{split}
\]
Here we wrote $C_{\beta,n}$ for an upper bound of $\tfrac{b^n}{n}(1+f_n)$, where $f_n=\mathcal{O}(1/x)$ is the remainder for the asymptotic behaviour of $K_0$. The estimate for the full coordinate now follows by summing over all corresponding integrals.
\end{proof}

Now that we have good control over the new functions, we can continue with the plan.

\begin{proposition}
	The set of $\mathcal{Y}_\gamma$ is a solution to the Riemann Hilbert problem.
\end{proposition}

The proof uses that the standard approach to solving a Riemann Hilbert problem is to define a Cauchy-type integral where the integrand is given by the jumping behaviour. Here we follow in particular \cite{FilippiniS.A.GarciaFernandezM..2017} for the general approach. They examine an example which is quite similar to ours, so at first we simply have to interpret the definition above in this way, but some additional complexities arise as we may have infinitely many jumps.

\begin{proof}
	We start the proof with the description of of a single jump, from which the result for finitely many jumps follows immediately. Finally, we have to consider the case of infinitely many jumps which is quite a bit more intricate and uses some additional information.

	As $\mathcal{X}_\gamma$ is a solution to the Riemann Hilbert problem for all $\gamma\in\Gamma_q$ it holds
	\[
		\frac{\mathcal{X}^+_\gamma(\zeta')}{S_l\mathcal{X}^+_\gamma(\zeta)}=\frac{\mathcal{X}^+_\gamma(\zeta')}{\mathcal{X}^-_\gamma(\zeta)}
		=S_l^{-1}.
	\]
	Note that for the last equality we use that in our setting we can regard $S_l$ simply as a function with values in $\mathbb{C}^\times$.
	Thus we see that the integrand for each BPS ray $l$ is given by 
	\[
		\frac{1}{\zeta'-\zeta}\frac{\zeta'+\zeta}{\zeta'}\frac{1}{2}\text{log}(S_l^{-1})=:\frac{1}{\zeta'-\zeta}\kappa_l.
	\]
	Here we included the factor $1/2$ in $\kappa_l$, so that the integral obtains the standard coefficient $\tfrac{1}{2\pi i}$. In this form $\kappa_l$ is continuous along $l$ (for all $\zeta$) and called the density, while $\frac{1}{\zeta'-\zeta}$ encodes the pole structure and is called the kernel. For $\zeta$ outside the BPS ray $l$ standard complex analysis tells us, that $I:=\tfrac{1}{2\pi i}\int_{l}\tfrac{1}{\zeta-\zeta'}\kappa_ld\zeta'$ defines a holomorphic function on the complement of $l$.
	
	For the next step we start by observing, that the continuation $\mathcal{X}^-_\gamma$ is holomorphic in 
	$\zeta$ with bounded derivative along $l$ (cf. section \ref{sec_diff_dep}). In fact it decays exponentially for $\zeta\to0$ and $\zeta\to\infty$, so $\kappa_l$ is also holomorphic with bounded derivative along $l$. Note that this depends on the choice of direction for $l$. As such $\kappa_l$ is Hölder continuous which is sufficient for the existence of the Cauchy principal value $\mathcal{P}I$ 
	of the integral. Crucially the Plemelj-Sokhotski formula now holds (cf. \cite{Mushkhelishvili.1953}, in particular section 2). It guarantees the existence of the limits $I^+(\zeta_0)$ and $I^-(\zeta_0)$ of $I$ as $\zeta$ approaches $\zeta_0\in l$ from both sides and relates them as 
	\[
		\begin{split}
			I^+(\zeta_0)&=\frac{1}{2}\kappa(\zeta_0)+\frac{1}{2\pi i}\mathcal{P}\int_l\frac{\kappa}{\zeta'-\zeta_0}d\zeta',\\
			I^-(\zeta_0)&=-\frac{1}{2}\kappa(\zeta_0)+\frac{1}{2\pi i}\mathcal{P}\int_l\frac{\kappa}{\zeta'-\zeta_0}d\zeta'.
		\end{split}
	\]
	
	We don't really have to deal with the principal value and can just use the implication that the difference of both sides is simply
	\[
		I^+(\zeta_0)-I^-(\zeta_0)=\kappa(\zeta_0)=\log(S_l^{-1}(\zeta_0)).
	\]
	
	Accordingly $\exp(I^+(\zeta_0))$ and $\exp(I^-(\zeta_0))$ exist and are related by
	\[
		\frac{\exp(I^+(\zeta_0))}{\exp(I^-(\zeta_0))}=\exp(I^+(\zeta_0)-I^-(\zeta_0))=\exp(\log(S_l^{-1}(\zeta_0)))=S_l^{-1}(\zeta_0).
	\]
	Thus we see that $\exp(I(\zeta))$ is a holomorphic function on the complement of $l$ in $\mathbb{C}^\times$ with the desired jumping behaviour at $l$. As the corresponding function for a different BPS ray $l'$ is holomorphic along $l$ they can safely be multiplied to obtain a function that jumps at $l$ and $l'$ and is holomorphic at their complement, which can be reiterated for all of the BPS rays in a finite sum. The same holds for multiplication with $\mathcal{X}_{\text{sf}}$ which is holomorphic on all of $\mathbb{C}^\times$ so that we obtain the desired behaviour for the whole $\mathcal{Y}_\gamma$ if only finitely many jumps appear.
	
	It remains to consider the case of infinitely many jumps. From the discussion in \cite{Gaiotto.2013} we know that this happens precisely, when a ring domain appears and we know, that the corresponding charges in $\Gamma_q$ are given by $\beta+m\gamma$ for $m\in\mathbb{Z}$ where $\gamma$ is the charge of the ring domain (cf. \cite{Bridgeland.2017}). We start by considering the product of jump-functions $P_k:=\prod_{m=1}^k\exp(I_{\beta+m\gamma})$ for $k\in\mathbb{N}$. When the ring domain $\gamma$ appears at some angle $\vartheta$, these are exactly the first $k$ jumps which appear on "one side" of $\vartheta$, say for angles smaller then $\vartheta$. From Theorem \ref{decay_estimate} we infer
	\[
		\left\|I_{\beta+m\gamma}\right\|\leq \frac{A}{\sqrt{R|Z_{\beta+m\gamma}|}}e^{-\pi R|Z_{\beta+m\gamma}|}\leq Be^{-mR|Z_\gamma|}
	\]
	for some constants $A,B$ which do not depend on $m$. So we obtain the following estimate.
	\[
		\left|\sum_{m=0}^\infty I_{\beta+m\gamma}\right|\leq\sum_{m=0}^\infty\left|e^{-mR|Z_\gamma|}\right|.
	\]
	From this we infer that $\left(P_k\right)_{k\in\mathbb{N}}$ is uniformly Cauchy and therefore converges uniformly towards $P=\prod_{m=1}^{\infty}\exp(I_{\beta+m\gamma})$. As all $P_k$ are continuous away from the rays of the jumps, $P$ is also continuous at this locus, which includes the ray at $\vartheta$. Additionally all $P_k$ are holomorphic away from the ray at $\vartheta$ and the rays of the jump and so is $P$. 
	
	The same procedure works for the infinite amounts of jumps on the right side of $\vartheta$, so their product is a function that is holomorphic away from the BPS rays and the ray of angle $\vartheta$, and whose right and left hand side limits for $\vartheta$ match. Finally we can build the product with the jump function for $\vartheta$ and obtain the function with correct jumping behaviour everywhere for a ring domain. As we have only finitely many ring domains we can do this for all ring domains and obtain the correct function.
\end{proof}

\begin{remark}
	The last step of the proof above uses again the fact, that we only have finitely many ring domains, which, as of Lemma \ref{unsealed_finite} is the case for generic differentials with unsealed spectrum. The spectrum had to be unsealed there, as otherwise another type of jump of the $ca$-triangulation could occur, which is not yet fully understood. So the problems there were of a rather geometrical nature. Here on the other hand, the specific geometric meaning of the jumps is not as important. Instead the problem with another type of infinitely many jumps leads to problems in the definition of the supposed solution of the Riemann-Hilbert problem. But this is rather an analytical problem, as even for, say, an infinite amount of accumulations points of jumps, one might come up with a good way of ordering them and obtaining a well-behaved solution. In fact, the assumptions here are hardly the most general one for obtaining a solution to such a Riemann-Hilbert problem. For a more general treatment of these problems we refer to \cite{Bridgeland.2017}. So we note that this theory provides a link between two aspects of quite different nature (geometrical and analytical) of quadratic differentials.
\end{remark}

We can now complete the link of the Riemann Hilbert problem and the geometric construction from before with the following uniqueness property.

\begin{proposition}
	The solution of the Riemann Hilbert problem stated above is unique up to multiplication by a constant.
\end{proposition}

\begin{proof}
	Let $X$ and $Y$ denote two solutions of the Riemann Hilbert problem and consider $Z:=\frac{X}{Y}$. This function is holomorphic on the complement of the BPS rays and the limits when approaching any ray $l$ from both sides are related as
	\[
		Z^+=\frac{X^+}{Y^+}=\frac{S_l^{-1}X^-}{S_l^{-1}Y^-}=\frac{X^-}{Y^-}=Z^-.
	\]
	We see that $Z$ does not jump at the rays and thus defines a continuous function on all of $\mathbb{C}^\times$. We also obtain that $Z$ is holomorphic on $\mathbb{C}^\times$, where again some care is necessary for the case of infinitely many jumps. If a ring domain occurs at angle $\vartheta$ the holomorphicity away from the BPS rays follows directly. On the BPS rays away from $\vartheta$ the holomorphicity follows from Morera's theorem, as $Z$ is holomorphic on a neighborhood of the ray. Thus only the ray at $\vartheta$ remains and we can again use Morera's theorem for this ray and obtain that $Z$ is holomorphic on all of $\mathbb{C}^\times$. (Note that we used the fact that only finitely many ring domains may occur here. Otherwise different accumulation phenomena could occur.) From the boundedness assumption we obtain that $Z$ has a removable singularity at $0$, so as of Riemann's theorem on removable singularities $Z$ can be continued to a holomorphic function on $\mathbb{C}$. The boundedness at $\infty$ on the other hand now guarantees as of Liouville's Theorem that $Z$ is constant, i. e. $X=AY$ for some $A\in\mathbb{C}^\times$. 
\end{proof}

We note that similar arguments concerning the uniqueness of solutions to Riemann-Hilbert problems can be found in \cite{Bridgeland.2017} and \cite{Neitzke2017}. Crucially we obtain the following corollary.

\begin{corollary}\label{A_appears}
There exists some $A_\gamma$ independent of $\zeta\in\mathbb{C}^\times$, s. t. 
$\mathcal{X}_\gamma=A_\gamma\mathcal{Y}_\gamma$. 
\end{corollary}

The value of the constant $A_\gamma$ can be regulated by demanding the appropriate behaviour for $\zeta\to\infty$ or $\zeta\to0$. By conjecture the constant should be $1$ in our case, but at the moment it is not clear, why this should be the case. The question is rather important, as $A_\gamma$ is as of now only constant w. r. t. $\zeta$ and not the coordinates of $\mathcal{M}$, so we need more information about this as we want to talk about asymptotic decay when varying the objects on $\mathcal{M}$. Luckily we have such information, so that we can obtain our main theorem of this section.

\begin{theorem}
	The constant $A_\gamma$ is of the form $A_\gamma=\tfrac{1+K}{e^k}$, where $K$ and $k$ are exponentially bounded.
\end{theorem}

\begin{proof}
	The constant $A_\gamma$ is given by the quotient
	\[
		A_\gamma=\frac{\mathcal{X}_\gamma}{\mathcal{Y}_\gamma}=\frac{\mathcal{X}_\gamma\left(\mathcal{X}^{\text{sf}}_\gamma\right)^{-1}}{\mathcal{Y}_\gamma\left(\mathcal{X}^{\text{sf}}_\gamma\right)^{-1}},
	\]
	which is independent of $\zeta$ as of the preceding argument. We added the quotient by $\mathcal{X}^{\text{sf}}_\gamma$ to highlight that this expression appears in $\mathcal{X}_\gamma$ as well as in $\mathcal{Y}_\gamma$ and thus cancels in the expression for $A$. From the calculations in the proof of Theorem \ref{decay_estimate} it follows $\underset{\zeta\to0}{\text{lim}}\left(\mathcal{Y}_\gamma\mathcal{X}^{\text{sf}}_\gamma\right)^{-1}=e^k$ for some $k$ with the appropriate behaviour. On the other hand we have from section \ref{sec_coord} that $\underset{\zeta\to0}{\text{lim}}\left(\mathcal{X}_\gamma\mathcal{X}^{\text{sf}}_\gamma\right)^{-1}$ is of the proposed form, so the constant follows as the corresponding quotient.
\end{proof}
	
\section{Twistorial construction of the hyperkähler metric}\label{twist_metric_all}

	In this final section we construct the hyperkähler metric out of the coordinate functions we just build. For this we first summarize the general idea of 
	constructing hyperkähler metrics by twistor methods in subsection \ref{sec_basic_twistor} and show how this works for the concrete example of the semiflat metric 
	in subsection \ref{twist_sem_flat}. Finally we use the construction for 
	the true coordinate functions that we build in the last section to obtain another hyperkähler metric which is then shown to decay exponentially towards 
	the semiflat metric in subsection \ref{twist_metric}.

	\subsection{Basic twistor theory}\label{sec_basic_twistor}
	
	The foundation for all of the upcoming constructions is the twistor theorem which we therefore state at this point. The original construction is found in 
	\cite{Hitchin.1987} and \cite{Hitchin.1992}. Consider any hyperkähler manifold $M$ 
	with complex structures $I$, $J$ and $K$ and corresponding Kähler forms $\omega_I$, $\omega_J$ and $\omega_K$. One of the crucial ideas of twistor 
	theory is to make use of the fact that in case of a hyperkähler manifold there is actually a whole $\mathbb{CP}^1$ worth of complex structures on $M$. 
	They 
	may be denoted by
	\[
		I^\zeta=\frac{(1-|\zeta|^2)I+i(-\zeta+\overline{\zeta})J-(\zeta+\overline{\zeta})K}{1+|\zeta|^2},
	\]
	where $\zeta$ is a local coordinate of $\mathbb{CP}^1$. The corresponding Kähler form is then given by
	\[
		\omega^{\zeta}=\frac{(1-|\zeta|^2)\omega_I+i(-\zeta+\overline{\zeta})\omega_J-(\zeta+\overline{\zeta})\omega_K}{1+|\zeta|^2}.
	\]
	Additionally there is the holomorphic symplectic structure 
	\begin{equation}\label{decomp_form}
		\varpi(\zeta)=-\frac{i}{2\zeta}\omega_++\omega_I-\frac{i}{2}\zeta\omega_-
	\end{equation}
	for $\omega_{\pm}=\omega_J\pm i\omega_K$. This structure is used to build the twistor space of $\mathcal{Z}$ of $M$ as the product manifold $Z:=M\times\mathbb{CP}^1$ 
	equipped with the complex structure
	\[
	\underline{I}:=(I^\zeta,I_0)
	\]
	where $I_0$ is the complex structure on $\mathbb{CP}^1$ obtained by multiplication with $i$ on the tangent space $T_\zeta\mathbb{CP}^1$ of 
	$\zeta\in\mathbb{CP}^1$.
	There is then an obvious fibration of the twistor space $\mathcal{Z}$ over $\mathbb{CP}^1$ with holomorphic sections $\zeta\mapsto(m,\zeta)$ for each 
	$m\in M$ which are called the twistor lines and which each have normal bundle isomorphic to $\mathbb{C}^{2n}\otimes O(1)$. Additionally there is a real 
	structure $\tau:M\times\mathbb{CP}^1\to M\times\mathbb{CP}^1$ given by $\tau(m,\zeta)=\left(m,\-\tfrac{1}{\overline{\zeta}}\right)$ that is compatible
	with all of the aforementioned structure in a suitable sense and induces the antipodal map on $\mathbb{CP}^1$. 
	
	Crucially this collection of structures on its own is enough of to obtain a hyperkähler manifold. This is the famous twistor theorem of N. J. Hitchin, A. 
	Karlhede, U. Lindström and M. Ro$\check{\text{c}}$ek (cf. \cite{Hitchin.1987}, \cite{Hitchin.1992}).
	\begin{theorem}\label{Twistor}
		Let $\mathcal{Z}^{2r+1}$ be a manifold of complex dimension $2r+1$ with the following structure:
		\begin{enumerate}
			\item $\mathcal{Z}$ is a holomorphic fiber bundle $p:\mathcal{Z}\rightarrow\mathbb{C}\mathbb{P}^1$ over the projective line.
			\item The bundle admits a family of holomorphic sections each with normal bundle isomorphic to  $\mathbb{C}^{2r}\otimes O(1)$.
			\item There exists a holomorphic section $\varpi$ of $\Lambda^2T^\ast_F\otimes O(2)$, defining a symplectic form on each fiber, where 
							$T_F:=\text{Ker }dp:T\mathcal{Z}\rightarrow T\mathbb{C}\mathbb{P}^1$ is the tangent bundle along the fibers.
			\item $Z$ has an anti-holomorphic involution $\tau:\mathcal{Z}\rightarrow \mathcal{Z}$  
			which covers the antipodal map on 
				$\mathbb{C}\mathbb{P}^1$: $(m,\zeta)\mapsto(m,-1/\overline{\zeta})$, and preserves 
				$\varpi$, i. e. $\tau^{\ast}\varpi=\overline{\varpi}$.
		\end{enumerate}
		Then the parameter space of real sections is a manifold $M$ of real dimension $4r$ with a natural hyperkähler metric for which $\mathcal{Z}$ is 
		the twistor space. Additionally if one starts with a Hyperkähler manifold $M'$ and uses this construction on it's twistor space to obtain $M$, then $M'$ is naturally included in $M$ and the structures coincide.
	\end{theorem}
	
	Starting with \cite{Gaiotto.2010} GMN proposed a way of using this theorem on a general class of integrable systems. They stated a list of conditions 
	for a set of coordinate functions on such an integrable system which, if fulfilled, would allow for the construction of a holomorphic structure and a 
	holomorphic symplectic form fulfilling all of the conditions of theorem \ref{Twistor}. We state here the general conjecture of GMN.

		\begin{conjecture}\label{GMN twistor}
			Let $\mathcal{B}$ be a $r$-dimensional complex manifold with a $2r$-dimensional torus fibration $\mathcal{M}\to\mathcal{B}$. Additionally let there 
			be a local system of lattices $\Gamma$ of rank $2r$ with generators $\gamma_i$ equipped with a pairing with coefficients $\epsilon^{ij}$, such that  
			for every choice of local patch in $\mathcal{B}$, quadratic refinement
			and local sections $\gamma$ of $\Gamma$ there are locally defined $\mathbb{C}^\times$-valued functions 
			$\mathcal{X}_\gamma(u,\theta;\zeta)$ of $(u,\theta)\in\mathcal{M}$ and $\zeta\in\mathbb{C}^\times$. Then define 
			\[
				\varpi(\zeta):=\frac{1}{8\pi^2 R}\sum_{ij}\epsilon_{ij}\frac{d\mathcal{X}_{\gamma^{i}}}{\mathcal{X}_{\gamma^{i}}}\wedge\frac{d\mathcal{X}_{\gamma^{j}}}{\mathcal{X}_{\gamma^{j}}},
			\]
			where $d$ is the fiber wise differential, which ignores $\zeta$ and $\epsilon_{ij}$ is the inverse of $\epsilon^{ij}$.
			
			Now assume that the following properties hold:
			\begin{enumerate}
				\item The $\mathcal{X}_\gamma$ are multiplicative, i. e. $\mathcal{X}_\gamma\mathcal{X}_{\gamma'}=\mathcal{X}_{\gamma+\gamma'}$.
				\item The $\mathcal{X}_\gamma$ obey the following reality condition
					\[
						\mathcal{X}_\gamma(\zeta)=\overline{\mathcal{X}_{-\gamma}(-1/\bar{\zeta})}.
					\]
				\item All $\mathcal{X}_\gamma$ are solutions to a single set of differential equations, of the form
					\[
						\begin{split}
							\frac{\partial}{\partial u^i}\mathcal{X}_\gamma&=\left(\frac{1}{\zeta}\mathcal{A}_{u^i}^{(-1)}+\mathcal{A}_{u^i}^{(0)}\right)\mathcal{X}_\gamma,\\
							\frac{\partial}{\partial \bar{u}^{\bar{i}}}\mathcal{X}_\gamma&=\left(\mathcal{A}_{\bar{u}^{\bar{i}}}^{(0)}+\zeta\mathcal{A}_{\bar{u}^{\bar{i}}}^{(1)}\right)\mathcal{X}_\gamma,
						\end{split}
					\]
							where the $\mathcal{A}_{u^i}^{(n)}$ and $\mathcal{A}_{\bar{u}^{\bar{i}}}^{(n)}$ are complex vector fields on the torus fiber $\mathcal{M}_u$, with the $\mathcal{A}_{u^i}^{(-1)}$ 
							linearly independent at every point and similarly $\mathcal{A}_{\bar{u}^{\bar{i}}}^{(1)}$. 
				\item For each $x\in\mathcal{M}$, $\mathcal{X}_\gamma(x;\zeta)$ is holomorphic in $\zeta$ on a dense subset of $\mathbb{C}^\times$.
				\item $\varpi(\zeta)$ is globally defined (in particular the $\varpi(\zeta)$ defined over different local patches of $\mathcal{B}$ agree with one another) and holomorphic in 
							$\zeta\in\mathbb{C}^\times$.
				\item $\varpi(\zeta)$ is nondegenerate in the appropriate sense for a holomorphic symplectic form, i.e. ker$\varpi(\zeta)$ is a $2r$-dimensional subspace of the $4r$-dimensional 
							$T_{\mathbb{C}}\mathcal{M}$.
				\item $\varpi(\zeta)$ has only simple poles as $\zeta\rightarrow0$ or $\zeta\rightarrow\infty$.
			\end{enumerate}
			Then $\mathcal{Z}:=\mathcal{M}\times\mathbb{C}\mathbb{P}^1$ has all the structure necessary for the Twistor construction \ref{Twistor}. In 
			particular, a hyperkähler metric $g$ on $\mathcal{M}$ can be reconstructed from $\mathcal{Z}$.
		\end{conjecture}

	The notes \cite{Neitzke.2013} from A. Neitzke give an overview of this construction and contain 
	some examples where the construction should be successful, one of which is the case of weakly parabolic Higgs bundles we are considering here. On this 
	specific space it was conjectured in \cite{Gaiotto.2013} that a variation of this construction should work, where the necessary adjustments are stated at 
	the beginning. In particular the dependence of the triangulations on $\vartheta$ and the relationship of $\vartheta$ and $\zeta$ and the corresponding jumps have to be considered, for which the 
	half-plane $\mathbb{H}_\vartheta$ was originally introduced.
	
	Although our case therefore fits in the general picture we don't work with the whole construction here, as we are not interested in building a 
	hyperkähler metric on some 
	previously unknown integrable system, but rather on a space where a lot of structure is already given and we can simply replace the parts we want. The 
	important part we do consider is the construction of the holomorphic symplectic form, i. e. we consider the following:
	
	Let $\Gamma$ be a lattice with generators $\gamma_i$ and for each $\gamma_i$ let $\mathcal{X}_{\gamma_i}$ be a function defined of 
	$M\times\mathbb{C}^\times$, s. t. for each fixed $\zeta\in\mathbb{C}^\times$ the $\mathcal{X}_{\gamma_i}$ are coordinates of $M$. Let there also be given 
	a pairing on $\Gamma$ with coefficients $\epsilon^{ij}$ and inverse $\epsilon_{ij}$. Then define for any parameter $R\in\mathbb{R}_+$ the following $2$-
	form:
	
	\begin{equation}\label{hol_symp}
		\varpi:=\frac{1}{8\pi^2 R}\sum_{i,j}\epsilon_{ij}
		\frac{d\mathcal{X}_{\gamma_i}}{\mathcal{X}_{\gamma_i}}\wedge\frac{d\mathcal{X}_{\gamma_j}}{\mathcal{X}_{\gamma_j}}.
	\end{equation}
	
	As above $d$ is the fiberwise differential that ignores $\zeta$. In the general construction the conditions on the $\mathcal{X}$ are supposed to ensure that 
	$\varpi$ is actually a holomorphic $2$-form twisted by $O(2)$ defining a symplectic form on each fiber just as needed by theorem \ref{Twistor}. As we 
	will show below, our work in the previous chapters gives us exactly this, as was conjectured by GMN in \cite{Gaiotto.2013}.

	We use this moment to highlight the property 5 of the construction. It states that $\varpi$ is holomorphic as a function depending on $\zeta\in\mathbb{C}^\times$. But $\varpi$ as defined in \eqref{hol_symp} is build out of the $\mathcal{X}$-coordinates which jump at certain rays in the $\mathbb{C}^\times$-plane, so some cancellation has to occur. To see this we consider some part $d\log\mathcal{X}_{\gamma}\wedge d\log\mathcal{X}_{\beta}$ of $\varpi$. The important fact, following from the discussion in \cite{Gaiotto.2013}, is that if a jump of $\mathcal{X}_{\gamma}$ occurs, $\mathcal{X}_{\beta}$ is holomorphic, i. e. $\mathcal{X}_{\beta}^+=\mathcal{X}_{\beta}^-$, and the jump for $\mathcal{X}_\gamma$ is given by formula \eqref{jumps}, which in the logarithm becomes $\log\mathcal{X}^-_\gamma\mapsto\log\mathcal{X}^-_\gamma+\log\mathcal{X}_\beta$. But this now cancels if we take the $\wedge$-product of the derivatives, so
	\[
		\left(d\log\mathcal{X}_{\gamma}\wedge d\log\mathcal{X}_{\beta}\right)^-=\left(d\log\mathcal{X}_{\gamma}\wedge d\log\mathcal{X}_{\beta}\right)^+.
	\]
	In the general setting this is the reason why the $\mathcal{K}_\beta^\gamma$, considered as automorphism on the space of unitary characters of $\Gamma_q$, are called symplectomorphisms. Here we simply denote this is in the following way.
	
	\begin{proposition}\label{symp_is_cont}
		Let $q$ be a generic differential with unsealed spectrum. Then $\varpi$ as defined in \eqref{hol_symp} with the coordinate functions $\mathcal{X}_\gamma:\mathbb{C}^\times\to\mathbb{C}^\times$ is holomorphic on all of $\mathbb{C}^\times$.
	\end{proposition}
	
	Before we go on, let us recall why we go through the trouble of this construction. After all we are about to obtain a hyperkähler metric on a space 
	which we already know is hyperkähler. The upshot is simply, as already noted in \cite{Gaiotto.2010}, that we not only obtain the existence of a hyperkähler metric but we get a very concrete 
	expression of the metric which allows for a much better understanding of its behavior then the mere existence result. To see this, consider the case that 
	$\varpi$ indeed fulfills the required conditions. Then by the twistor theorem we obtain that $M$ is a hyperkähler space whose twistor space is equipped 
	with the holomorphic symplectic form $\varpi$. But we know from the general twistor theory that the holomorphic 
	symplectic form can be decomposed as in formula \eqref{decomp_form}, i. e. we may write
	\begin{equation}
		\varpi(\zeta)=-\frac{i}{2\zeta}\omega_++\omega_I-\frac{i}{2}\zeta\omega_-
	\end{equation}
	where $\omega_I$, $\omega_J$ and $\omega_K$ are the Kähler forms of the hyperkähler metric $g$. Thus we can extract the metric by specifying the $\zeta$ 
	to an appropriate value. Below we will use, that the complex structure $I$ in our case is simply given by multiplication by $i$, so we state the 
	according extraction here.
	
	\begin{lemma}\label{g_desc}
		Assume that all the conditions of the twistor theorem are fulfilled. Using the notation from above the hyperkähler metric is given as
		\[
		g(v,w)=\omega_I(v,Iw)=\Re(\left.\varpi\right|_{\zeta=-1}(v,Iw)).
		\]
	\end{lemma}
	
	\begin{proof}
		Setting $\zeta=-1$ we obtain
		\[
		\begin{split}
		\left.\varpi\right|_{\zeta=-1}=\frac{i}{2}\omega_++\omega_I+\frac{i}{2}\omega_-
		=\frac{i}{2}(\omega_J+\omega_K)+\omega_I+\frac{i}{2}(\omega_J-\omega_K)
		=\omega_I+i\omega_J.
		\end{split}
		\]
		Thus we have $\omega_I=\Re(\left.\varpi\right|_{\zeta=-1})$. As $\omega_I$ is the Kähler form for the action of $I$ we obtain
		\[
			g(v,w)=\omega_I(v,Iw)=\Re(\left.\varpi\right|_{\zeta=-1}(v,Iw)).
		\]
	\end{proof}
	
	In this way knowing $\varpi$ amounts to knowing $g$ and the better we understand $\varpi$ the better we understand $g$. In our case, as $\varpi$ is 
	constructed from the small flat sections in a concrete manner, we obtain that the better we understand the small flat sections, the better we understand 
	$\varpi$ and therefore $g$. This justifies the work we put in the detailed calculations concerning those sections.
	We conclude this introduction with a corresponding notion which will help in distinguishing different hyperkähler metrics in the following sections.
	
	\begin{definition}
		Any hyperkähler metric that is obtained in the way described above, i. e. from a holomorphic symplectic form that is build out of a set of coordinate 
		functions via formula \eqref{hol_symp} shall be called \textit{twistorial}.
	\end{definition}
		
	\subsection{The semiflat metric in the twistorial picture}\label{twist_sem_flat}
	
	Before we go on to use the method we just described for our true coordinate functions, it may be enlightening to describe the much simpler case of the 
	semiflat metric and how it fits into the picture. Basically we can just do everything we've done until now, only that we do not start with the true 
	solutions of Hitchin's equations but rather work only with the limiting configurations. Before we do this however we have to point to the fact that there 
	are two different point of views we have to consider here. We could just build all the objects via the formulas given above out of the limiting 
	configuration and obtain a symplectic $\varpi^{sf}$ form on the true moduli space $\mathcal{M}$ of solutions of Hitchin's equation. However in this 
	way we don't know whether $\varpi^{sf}$ is a holomorphic form on $\mathcal{M}$. Our argument for the holomorphicity won't work as the corresponding small 
	flat sections are not solutions to the true ODE and thus won't necessarily vary holomorphically with the ODE data. We do know however that there is 
	semiflat metric $g_{\text{sf}}$ on $\mathcal{M}$ from the usual theory of parabolic Higgs bundles.
	
	On the other hand we can consider the moduli space of limiting configurations $\mathcal{M}_\infty$ (cf. \cite{Mazzeo.2016} for the regular case and 
	\cite{Fredrickson.2020} and references therein for the irregular one) with corresponding regular locus $\mathcal{M}'_\infty$. In this case the 
	construction does work in full and we really obtain the 
	twistorial metric $g^{\text{lim}}_{\text{twist}}$.
	The relevant fact here is that these two sides actually coincide, i. e. with the appropriate identifications we have 
	$g_{\text{sf}}=g^{\text{lim}}_{\text{twist}}$. This fact can be obtained from the explicit expression of $g^{\text{lim}}_{\text{twist}}$ which we'll now 
	describe on $\mathcal{M}'_\infty$. As our goal in this section is to describe how the explicitly known semi-flat metric can be obtained in the twistorial picture, we refrain from giving any detailed proofs here. The fact that the finally resulting form is actually the Kähler form of the semi-flat metric simply follows from the appearance of the explicitly known structure. For the rigorous argument we simply note that all calculations are simpler versions of the ones we already did and the remaining arguments concerning the twistor theory will be given for the more complicated version of the twistorial metric on $\mathcal{M}$ in the next section.
	
	In the space of limiting configurations all of the calculations become way more easy as the Higgs field as well as the connection form are diagonal in unitary frame for the corresponding Hermitian metric. We 
	may thus write the flat connection $\nabla^L$ of the limiting case in the following way, using the third Pauli matrix $\sigma^3$ for convenience:
	
	\[
		\begin{split}
				\nabla^L(\zeta):=\frac{R}{\zeta}\Phi_\infty+A_\infty+R\zeta\Phi_\infty^\ast
				=\partial+\frac{R}{\zeta}q^{1/2}\sigma^3dz+A_zdz+\overline{\partial}+A_{\overline{z}}d\overline{z}+R\zeta \overline{q^{1/2}}\sigma^3d
				\overline{z}.
		\end{split}
	\]
	
	Here $\Phi_\infty=q^{1/2}\sigma_3$, $A_z$ is diagonal with entries $-a_1$ and $-a_2$, and $A_{\overline{z}}=-\overline{A_z}$.\footnote{Here the signs are 
	chosen to match the formulas in section \ref{local_desc}.} We also consider the same conventions for the quadrilateral, annulus and frame as 
	constructed in section \ref{local_desc}. 
	As the connection is in "diagonal gauge" and the description 
	is in local coordinates the sections can be computed in each entry of $s=(s_1,s_2)$ independently. For $s_i$ the differential equation 
	becomes:
	\[
		0=\partial s_i-a_1s_1dz+\frac{R}{\zeta}q^{1/2}s_idz+\overline{\partial}s_i+\overline{a_1}s_id\overline{z}+R\zeta \overline{q^{1/2}}s_id\overline{z}.
	\]
	
	As before the $dz$ and $d\overline{z}$ part have to become $0$ independently, so the equation splits into the set of two PDEs:
	\[
		\frac{\partial s_i}{\partial z}=\left(-\frac{R}{\zeta}q^{1/2}+a_1\right)s_i\quad\wedge\quad 
		\frac{\partial s_i}{\partial\overline{z}}=(-R\zeta\overline{q^{1/2}}-\overline{a_1})s_i.
	\]
	Denote by $P$ a primitive of $q^{1/2}$. From $\overline{\partial}\overline{P}=\overline{\partial P}=\overline{q^{1/2}}$ it then follows, that 
	$\overline{P}$ is a primitive of $\overline{q^{1/2}}$. Furthermore let $\Delta_i$ be a primitive of $a_i-\overline{a_i}$.
	With these notations the solution $s_1$ of the set of PDEs becomes
	\[
		s_1=C_1e^{-\frac{R}{\zeta}P+\Delta_1-R\zeta\overline{P}},
	\]
	for any $C_1\in\mathbb{C}$. For the other component the different sign in $\sigma^{3}$ leads to the corresponding solution
	\[
		s_2=C_2e^{\frac{R}{\zeta}P-\Delta_2+R\zeta\overline{P}},
	\]
	for any $C_2\in\mathbb{C}$. So we have the general solution 
	\[
		s_L:=\begin{pmatrix}C_1e^{-\frac{R}{\zeta}P+\Delta-R\zeta\overline{P}}\\C_2e^{\frac{R}{\zeta}P-\Delta+R\zeta\overline{P}}\end{pmatrix},
	\]
	where the superscript $L$ indicates that these are the flat sections for the limiting configurations. 
	
	Just as in section \ref{sec_small_sec} we want to choose a small flat section for each vertex, but now this is easily obtained as no remainder term exists and we 
	see directly which section to choose for each vertex. I. e. if we label the vertices of a quadrilateral by $1,2,3,4$ we obtain for vertex $1$ the 
	section
	\[
		s^1_L(\zeta,z):=\begin{pmatrix}0\\e^{\frac{R}{\zeta}P_1(z)-\Delta(z)_2+R\zeta\overline{P_1}(z)}\end{pmatrix}.
	\]
	As before we need to consider the primitives in a contractible subset of the standard annulus, but as we know the solutions completely, we don't have to 
	transport the solutions to the $ca$-trajectories in order to evaluate them.
	
	Then we may again build the $\mathcal{X}_\gamma$ coordinates and obtain the corresponding period integrals.
	\[
		\begin{split}
				\mathcal{X}^L_\gamma&=-\frac{(s^1_L\wedge s^2_L)(s^3_L\wedge s^4_L)}{(s^2_L\wedge s^3_L)(s^4_L\wedge s^1_L)}\\
													&=-\exp\left(\frac{R}{\zeta}\oint_\gamma{q^{1/2}}+\left(\oint_\gamma a_1dz-\overline{a_1}d\overline{z}\right)
													+R\zeta\oint_\gamma{\overline{q}^{1/2}}\right).
		\end{split}
	\]
	We can now define just as before $i\theta:=\left(\oint_\gamma a_1dz-\overline{a_1}d\overline{z}\right)$. Together with the periods $Z_\gamma:=\tfrac{\pi}\oint_\gamma{q^{1/2}}$ we obtain
	the action angle coordinates on the corresponding integrable system, where the $Z_\gamma$ are coordinates on the Hitchin base and the $\theta$ are 
	coordinates on the torus fibers. It follows from the theory of limiting configurations that these are in fact periodic, so they give appropriate 
	coordinates on the torus fibers of $\mathcal{M}'_\infty$. With these the final form for the $\mathcal{X}_\gamma$ coordinates for the 
	limiting configuration is obtained:
	\begin{equation}\label{x_lim}
		\mathcal{X}^L_\gamma=-\exp\left(\frac{R}{\zeta}\pi Z_\gamma+i\theta_\gamma+R\zeta\pi\overline{Z}_\gamma\right).
	\end{equation}
	In light of the goal to describe the asymptotic behaviour of the true coordinate system we note here that this is exactly the term we obtained from the 
	leading term for the true $\mathcal{X}$ coordinates in Theorem \ref{x_int_form}. Thus the difference amounts to the addition of the remainder term $r_q$ for the true coordinates. We will come 
	back to this at the end.
	
	Again these coordinate functions can be defined for each charge $\gamma$ in the charge lattice $\Gamma_q$, although in the following we only use the 
	generators of the gauge lattice $\Gamma^g_q$. We then obtain the holomorphic symplectic form on $\mathcal{M}'_\infty$ as in the description of equation \ref{hol_symp} above:
	
		\[
			\varpi^{sf}(\zeta):=\frac{1}{8\pi^2R}\sum_{i,j}\epsilon_{ij}
			\frac{d\mathcal{X}_{\gamma_i}^{L}}{\mathcal{X}_{\gamma_i}^{L}}\wedge \frac{d\mathcal{X}_{\gamma_j}^{L}}{\mathcal{X}_{\gamma_j}^{L}}.
		\]
		Here the coefficients $\epsilon_{ij}$ arise as the inverse of the intersection pairing on $\Gamma^g$, i. e. we regard $dZ$ and $d\theta$ as elements in $\Omega^1((\Gamma^g)^\ast)$ and use the corresponding dual pairing of the intersection pairing on $(\Gamma^g)^\ast$.
	
	We then obtain the metric $g^{\text{lim}}_{\text{twist}}$ by simply plugging in the known expression for $\mathcal{X}_\gamma^L$. Here it is important to 
	note the fact that
	\begin{equation}\label{Griffiths_trans}
		dZ_{\gamma_i}\wedge dZ_{\gamma_j}=0.
	\end{equation}
	This is clear if the dimension of the Hitchin Base is $1$, which is the case for the four punctured sphere. A proof for higher dimensions can be found in the script \cite{Neitzke2017} where it is referred to as ""Griffiths transversality" for spectral curves" and interpreted as an embedding of the Hitchin base as a complex Lagrangian submanifold of some complex symplectic vector space. The main idea of the proof is to use a "Riemann bilinear identity" to identify the intersection pairing in the the homology with the integral over the wedge product of two holomorphic 1-forms on the spectral curve with, which then has to vanish.
	
	Using this cancellation and the fact that the $\theta$ coordinates are real-valued one obtains for $\zeta=-1$ the Kähler form
	\[
		\omega_I=\Re(\left.\varpi^{\text{sf}}\right|_{\zeta=-1})=\sum_{i,j}\epsilon_{ij}
		\left(\frac{R}{4}dZ_{\gamma_{i}}\wedge d\overline{Z}_{\gamma_{j}}-\frac{1}{8\pi^2 R}d\theta_{\gamma_i}\wedge d\theta_{\gamma_j}\right).
	\]
	
	Now this is in fact the Kähler form for the semiflat metric, so we obtain as stated above the equality $g^{\text{lim}}_{\text{twist}}=g_{\text{sf}}$. This concludes the twistorial construction of the semiflat metric. It will appear again at the end of the last section, as it is our goal to compare 
	$g_{\text{sf}}$ and $g_{\text{twist}}$ on (a subset of ) $\mathcal{M}'$.

\subsection{The twistorial hyperkähler metric on \texorpdfstring{$\mathcal{M}$}{M}}\label{twist_metric}
	
	Finally we use the construction of section \ref{sec_basic_twistor} on the space of true solutions to Hitchin's equations $\mathcal{M}$. At this point we 
	have to specialize the construction of the $\mathcal{X}$-coordinates which up until now depended on the (almost) arbitrary angle parameter $\vartheta$ of 
	the triangulation, so that it is well defined and all the results from the sections before hold. This is done by setting $\vartheta:=\text{arg}\zeta$ and 
	building the triangulation and coordinates accordingly. Obviously in this way $\zeta\in\mathbb{H}_{\vartheta}$ which was the sufficient condition for the 
	construction to work. But no we can't control $\vartheta$ to be generic anymore. Indeed there is now the 
	real co-dimension $1$ locus in the $\zeta$-plane where the coordinates jump. Away from this locus the coordinates are fine, but it is one of the crucial 
	parts of the main Theorem \ref{omega_symp} in this section that these jumps do not pose a problem for us, so all the considerations of section \ref{spec_diff} come into play.
	
	We start by investigating the complex structures on $\mathcal{M}$. Fore more details on this we refer to \cite{Swoboda.2021}, \cite{Mazzeo.2017} and \cite{Fredrickson.2020}. In order to define them we have to consider the tangent space at $(A,\Phi)\in\mathcal{M}$. The space of unitary connections on $E$ which induce the fixed connection on det$E$ is an affine space modeled on $\Omega^1(\mathfrak{su}(E))$. With the identification $\Omega^1(\mathfrak{su}(E))\cong \Omega^{0,1}(\text{End}_0(E)), D\mapsto D^{0,1}:=\alpha$
	we may regard the deformations of the connection part of the Higgs pair $(A,\Phi)$ as an element of $\Omega^{0,1}(\text{End}_0(E))$. We observe that this is just the complex conjugate of the space of traceless Higgs fields $\Omega^{1,0}(\text{End}_0(E))$, so deformations of Higgs pairs can be regarded as pairs
	$(\alpha,\phi)\in\Omega^{0,1}(\text{End}_0(E))\times \Omega^{1,0}(\text{End}_0(E))$. With this identification we can understand that the following maps are complex structures on $\mathcal{M}$:
	
	\[
		I(\alpha,\phi)=(i\alpha,i\phi),\enspace J(\alpha,\Phi)=\left(-R\phi^{\ast_h},\tfrac{1}{R}\alpha^{\ast_h}\right),
		\enspace K(\alpha,\phi)=\left(-iR\phi^{\ast_h},\tfrac{i}{R}\alpha^{\ast_h}\right).
	\]
	Note that the factor $R$ appears in the definition of the complex structures $J$ and $K$. This is an artifact of the the moduli space of weakly parabolic Higgs bundle, where the $\mathbb{C}^\ast$ action of the regular and strongly parabolic moduli spaces does not exist. In those cases the $\mathbb{C}^\ast$ action allows to regard solutions of the $R$-rescaled Hitchin equation as solutions of the original Hitchin equation (i. e. with $R=1$). But in our case the fixed residues at the poles of det$(\Phi)$ prevent us from using this identification. Thus each $R\in\mathbb{R}_>0$ corresponds to it's own moduli space and the corresponding scaling has to be considered in the complex structures and hyperkähler metrics.
	
	\begin{remark}\label{hol_quotient}
		The description of the the holomorphic structures above might seem odd, as they are not defined on the elements of the moduli space but on their representatives. This basically boils down to the same considerations we explained in remark \ref{diff_quotient}. The fact that the holomorphic structure is preserved when taking the equivalence class is a direct consequence of the construction of the moduli space as a Hyperkähler quotient. The construction actually starts with complex structures on the set of all the corresponding objects (i. e. before taking the quotient by gauge transformations) and then proceeds by building the structure on the moduli space in exactly such a way that everything commutes. Therefore it suffices to talk about the action of the complex structure on representatives and not on the elements of the moduli space as we just did above. Again we refer to \cite{Mayrand.2016} for more details.
		
		Additionally the other considerations of remark \ref{diff_quotient} are also important in this section, i. e. it is important to note that when working in a local frame we can chose this frame in such a way, that all coefficients of the local expression vary holomorphically in a holomorphic coordinate of the moduli space. Again, as explained before, in choosing the appropriate frame the exact expressions might differ slightly from the one we use in our calculations, which is solely the case because the structure of the equations is easier to understand in this way. The difference only amounts to some different expression of the leading term and slightly different terms of the error matrix, which do not change the analytic properties which are important here.
		\end{remark}
	
	We may now consider again the nonabelian Hodge correspondence that we used to identify the moduli space of solutions of Hitchins equations $\mathcal{M}$ and the moduli space of flat connections $\mathcal{M}_{\text{flat}}$ in section \ref{sec_small_sec}. There the emphasis lay on the parameter $\zeta\in\mathbb{C}^\ast$ while we now, in light of the discussion above, add $R$ to the discussion. Concretely for every $R\in\mathbb{R}_{>0}$ and $\zeta\in\mathbb{C}^\ast$ we obtain the map
	\[
		\text{NAHC}_{R,\zeta}:\mathcal{M}\to\mathcal{M}_{\text{flat}},\quad (A,\Phi)\mapsto d_A+\frac{R}{\zeta}\Phi+R\zeta\Phi^{\ast_h}=\nabla(R,\zeta).
	\]
	The space of flat connections is equipped with the trivial complex structure $\hat{I}$, that is given by multiplication with $i$. On the other hand the three complex structures on $\mathcal{M}$ lead to the general complex structure $I^{R,\zeta}$, which we described in the beginning of this chapter and where we added $R$ to emphasize the dependence we explained above. By a simple calculation one can now show that $\text{NAHC}_{R,\zeta}$ preserves the complex structure in the following way.
	
	\begin{lemma}\label{NAHC_hol}
		Fix $R\in\mathbb{R}_{>0}$ and $\zeta\in\mathbb{C}^\ast$. Then $\text{NAHC}_{R,\zeta}:(\mathcal{M},I^{R,\zeta})\to(\mathcal{M}_{\text{flat}},\hat{I})$ is holomorphic. 
	\end{lemma}
	
	In this formulation we exclude the complex structures $I$ and $-I$ on $\mathcal{M}$ which corresponds to the moduli space of Higgs bundles. Indeed, there is no $\zeta\in\mathbb{C}^\ast$ for which the nonabelian Hodge correspondence becomes holomorphic w. r. t. $I$. The construction below will still give us a form $\varpi$ that is also holomorphic in these complex structures as it will be twisted by the bundle $O(2)$.
	
	
	We now recall the restrictions we have to take to make sure that the construction works.
	\begin{definition}
		Let $\mathcal{M}$ denote the moduli space of solutions of weakly parabolic Higgs bundles as described in section \ref{parabolic_higgs} for Higgs bundles with parabolic degree $0$. In this space $\mathcal{M}'$ denotes the regular locus, in which the quadratic differential $q$ given by the determinant of the Higgs field has only simple zeroes and $\widetilde{\mathcal{M}}'$ the subset of $\mathcal{M}'$, in which the determinants $q$ have unsealed spectrum. Finally denote by $\widetilde{\mathcal{M}}'_c$ the subset corresponding to the chambers in the Hitchin base $\mathcal{B}$.
	\end{definition}
	
	Note that as explained in section \ref{spec_diff} in the case of weakly parabolic Higgs bundles on the sphere with at least 4 punctures it holds $\widetilde{\mathcal{M}}'=\mathcal{M}'$ and the chambers form an open dense subset of $\widetilde{\mathcal{M}}'$. Thus there is a large class of moduli spaces where the following construction works without further deformations of ring domains or other adjustments.
	
	We now use the NAHC$_{R,\zeta}$ to pull back the $\mathcal{X}(\zeta)$ (and $\mathcal{Y}(\zeta)$) functions on $\mathcal{M}^{\text{dR}}$ to obtain corresponding functions on $\widetilde{\mathcal{M}}'_{R,c}$ which we will denote with the same symbols. Thus $\mathcal{X}(\zeta)$ is now a holomorphic function in complex structure $I_\zeta$, and we can use this to show that the phase factor $A_\gamma$ appearing in corollary \ref{A_appears} is constant w. r. t. to coordinates on the moduli space.
	
	\begin{lemma}
	Let $A_\gamma$ be the factor appearing in corollary \ref{A_appears} s. t. $\mathcal{X}(\zeta)=A_\gamma\mathcal{Y}(\zeta)$ and let $d$ denote the differential w. r. t. $\widetilde{\mathcal{M}}'_{R,c}$. It holds $dA_\gamma=0$.
\end{lemma}

\begin{proof}

We write the holomorphicity $I\circ d\mathcal{X}(\zeta)=d\mathcal{X}(\zeta)\circ I_{\zeta}$ of $\mathcal{X}_\gamma(\zeta)=A_\gamma\mathcal{X}^{\text{sf}}_\gamma(\zeta)E_\gamma(\zeta)$ in the complex structure $I_{\zeta}$ to obtain the formula
\[
\begin{split}
	I\circ \frac{dA_\gamma}{A_\gamma}-\frac{dc_\gamma}{A_\gamma} \circ I_\zeta&=\left(\frac{R}{\zeta}dZ_\gamma+id\Theta_\gamma+R\zeta d\overline{Z_{\gamma}}\right)\circ I_{\zeta}+dE_\gamma(\zeta)\circ I_{\zeta}\\
	&-I\circ\left(\frac{R}{\zeta}dZ_\gamma+id\Theta_\gamma+R\zeta d\overline{Z_{\gamma}}\right)
	-I\circ dE_\gamma(\zeta).
\end{split}
\]
The correction term $E_\gamma$ is constant for $\zeta=0$ and the terms with first order pole in $\zeta=0$ vanish as they are given by the $I$-holomorphic function $Z_\gamma$. Thus in the limit $\zeta\to 0$ we obtain
\[
I\circ\frac{dA_\gamma}{A_\gamma}-\frac{dA_\gamma}{A_\gamma}\circ I=i\left(d\Theta_\gamma\circ I-I\circ d\Theta_\gamma\right)-2RdZ_\gamma\circ K.
\]
Analogously we obtain from the limit $\zeta\to\infty$ the expression
\[
-I\circ\frac{dA_\gamma}{c_\gamma}-\frac{dA_\gamma}{A_\gamma}\circ I=i\left(-d\Theta_\gamma\circ I-I\circ d\Theta_\gamma\right)-2Rd\overline{Z_\gamma}\circ K.
\]
Adding both equations leads to the formula
\begin{equation}\label{dc_formula}
	\frac{dA_\gamma}{A}\circ I=-d\Theta+R(dZ_\gamma+d\overline{Z_\gamma})\circ K.
\end{equation}
To get a hold on this we first consider the expression on the right hand side for the semiflat metric, i. e. for $K^{\text{sf}}\left(\alpha,\phi\right)=i\left(-R\overline{\phi}^t,\tfrac{1}{R}\overline{\alpha}^t\right)$. Here $(\alpha,\phi)\in\Omega^{0,1}(\text{End}_0(E))\times \Omega^{1,0}(\text{End}_0(E))$ is a tangent vector at a point $(A,\Phi)\in\mathcal{M}$ . 

As of \cite{Fredrickson.2020} we can take $\alpha$ and $\phi$ in a diagonal gauge in our standard frame. Thus a curve $s$ representing $\left(\alpha,\phi\right)$ is given by
\[
	s(t)=\left(A+t\alpha+\mathcal{O}(t^2),\Phi+t\phi+\mathcal{O}(t^2)\right).
\]
We now remember that $Z$ and $\Theta$ are defined in terms of the eigenvalues of the Higgs pair\footnote{A priori the eigenvalues of the connection matrix depend on the local frame. But after fixing the frame as an eigenframe of the Higgs field as described in section \ref{frame_sec} the only freedom that remains is a scaling in the eigenspaces and the final expression for $\Theta$ is independent of this ambiguity.} and write $\alpha=\begin{pmatrix}a&0\\0&-a\end{pmatrix}$ for $a\in\Omega^{0,1}(\mathcal{C})$. Then the diagonality of $A$ and $\alpha$ implies 
\[
	d\Theta_\gamma(\alpha,\phi)=\left.\frac{d}{dt}\right|_{t=0}\Theta(s(t))
	=-i\left(\oint_\gamma\overline{a}d\overline{z}-\oint_\gamma a dz\right).
\]
On the other hand we can write the curve $\tilde{s}$ corresponding to $K(\alpha,\phi)$ as $\tilde{s}(t)=(A,\Phi)+ti\left(-R\overline{\phi}^t,\tfrac{1}{R}\overline{\alpha}^t\right)+\mathcal{O}(t^2)$ and obtain
\[
dZ_\gamma(K^{\text{sf}}(\alpha,\phi))=\left.\frac{d}{dt}\right|_{t=0}Z_\gamma(\tilde{s}(t))
=\frac{i}{R}\oint_\gamma a dz.
\]
Thus it holds
\[
	d\Theta_\gamma=R(dZ_\gamma+d\overline{Z_\gamma})\circ K^{\text{sf}}.
\]
Note that this equation holds in the linear subspace of $\Omega^{0,1}(\text{End}_0(E))\times \Omega^{1,0}(\text{End}_0(E))$ for which the period integrals are well defined. 
We can extend $dc$ linearly to this subspace via the formula \ref{dc_formula} and plug the last equation into equation \ref{dc_formula} to obtain
\[
	\frac{dA_\gamma}{c_\gamma}\circ I=-R(dZ_\gamma+d\overline{Z_\gamma})\circ K^{\text{sf}}+R(dZ_\gamma+d\overline{Z_\gamma})\circ K.
\]
Concatenation with $K$ and $K^{\text{sf}}$ leads to the equations
\begin{equation}
\begin{split}
	\frac{dA_\gamma}{A_\gamma}\circ I\circ K&=-R(dZ_\gamma+d\overline{Z_\gamma})\circ K^{\text{sf}}\circ K-R(dZ_\gamma+d\overline{Z_\gamma})\\
	\frac{dA_\gamma}{A_\gamma}\circ I\circ K^{\text{sf}}&=R(dZ_\gamma+d\overline{Z_\gamma})+R(dZ_\gamma+d\overline{Z_\gamma})\circ K\circ K^{\text{sf}}.
\end{split}
\end{equation}
Taking the sum of these we obtain
\[
	\frac{dA_\gamma}{A_\gamma}\circ I\circ K+\frac{dA_\gamma}{A_\gamma}\circ I\circ K^{\text{sf}}=R(dZ_\gamma+d\overline{Z_\gamma})\circ K\circ K^{\text{sf}}-R(dZ_\gamma+d\overline{Z_\gamma})\circ K^{\text{sf}}\circ K.
\]
The expression on the right hand side can now be evaluated in a local frame similar to the calculation above: Locally the hermitian adjoint is given by $\alpha^\ast=H^{-1}\overline{\alpha}^tH$ for $H$ the matrix representing the hermitian metric. Using the properties described in section \ref{Hitchin_pair} one obtains that the right hand side vanishes. We thus obtain
\[
	dA_\gamma\circ I\left(\alpha,\phi\right)=-dA_\gamma\circ I\left(H^{-1}\alpha H,H^{-1}\phi H\right).
\]
But $H$ is exponentially closed to the identity for large $R$, which can only be true for $dA_\gamma=0$.

\end{proof}

For this we start with the expression
	\[
		\mathcal{X}_\gamma=A_{\gamma}\mathcal{X}^{\text{sf}}_\gamma\exp\left(\frac{1}{4\pi i}\sum_\beta\Omega(\beta;q)\left\langle\gamma,\beta\right\rangle I_{l_\beta}\right),
	\]
and the corresponding expression for the derivative
	
	\begin{equation}\label{log_diff_expand}
	d\log{\mathcal{X}_\gamma}=\frac{R}{\zeta}\pi dZ_\gamma+id\theta_\gamma+R\zeta\pi d\overline{Z_\gamma}+\frac{1}{4\pi i}\sum_\beta\Omega(\beta;q)\left\langle\gamma,\beta\right\rangle dI_{l_\beta}.
	\end{equation}
	We first observe that the first part corresponds to the semiflat part. For the second part we consider again the definition of $I_{l_\beta}$ and obtain
	\begin{equation}\label{diff_int_eq}
	\begin{split}
	dI_{l_\beta}&=-\int_0^\infty\frac{1}{s}d\left(\frac{Z_\beta s-\zeta}{Z_\beta s+\zeta}\right)\log(1-\mathcal{X}_\beta)ds\\
	&-\int_0^\infty\frac{1}{s}\left(\frac{Z_\beta s-\zeta}{Z_\beta s+\zeta}\right)d(1+\kappa(-Z_\beta s))\frac{\mathcal{X}^{\text{sf}}_\beta}{1-\mathcal{X}_{\beta}}ds\\
		&-\int_0^\infty\frac{1}{s}\left(\frac{Z_\beta s-\zeta}{Z_\beta s+\zeta}\right)(1+\kappa(-Z_\beta s))\frac{d\mathcal{X}^{\text{sf}}_{\beta}}{1-\mathcal{X}_{\beta}}ds
		:=I^1_{l_\beta}+I^2_{l_\beta}+I^3_{l_\beta}.
	\end{split}
	\end{equation}
	The first integral becomes
	\[
	I^1_{l_\beta}=-\int_0^\infty\frac{2\zeta}{\left(Z_\beta s+\zeta\right)^2}\log(1-\mathcal{X}_\beta)dsdZ_\beta.
	\]
	Expanding the logarithm and changes of variables just like in the proof of Theorem \ref{decay_estimate} now leads to
	\[
		\begin{split}
	I^1_{l_\beta,n}
	&=\int_{-\infty}^{\infty}\frac{2\zeta}{\left(e^{y+i\psi}+\zeta\right)^2}e^{y-\log|Z_\beta|}\frac{\left(1+\kappa(-e^{y+i\psi_\beta})\right)^n}{n}e^{in\theta_\beta}\\
	&\cdot\exp\left(-2n|Z_\beta|R\pi\cosh\left(y\right)\right)dydZ_\beta\\
	&=\int_{0}^{\infty}\frac{2\zeta}{\left(e^{y+i\psi}+\zeta\right)^2}e^{y-\log|Z_\beta|}\frac{\left(1+\kappa(-e^{y+i\psi_\beta})\right)^n}{n}e^{in\theta_\beta}\\
	&\cdot\exp\left(-2n|Z_\beta|R\pi\cosh\left(y\right)\right)dydZ_\beta\\
	+&\int_{0}^{\infty}\frac{2\zeta}{\left(e^{y+i\psi}+\zeta\right)^2}e^{-y-\log|Z_\beta|}\frac{\left(1+\kappa(-e^{-y+i\psi_\beta})\right)^n}{n}e^{in\theta_\beta}\\
	&\cdot\exp\left(-2n|Z_\beta|R\pi\cosh\left(y\right)\right)dydZ_\beta.
		\end{split}
	\]
	The main difference to the integrals before is the additional integrating factor $e^y$. Still the rest in front of the exponential is bounded, so we may estimate (the function part of) $I^1_{l_\beta,n}$ as
	\[
	\begin{split}
		\left\|I^1_{l_\beta,n}\right\|&\leq\frac{1}{|Z_\beta|}a\frac{b^n}{n}\int_0^\infty\cosh(y)\exp\left(-2n|Z_\beta|R\pi\cosh\left(y\right)\right)\\
		&=\frac{1}{|Z_\beta|}a\frac{b^n}{n}K_1(2n\pi R|Z_\beta|)
	\end{split}
	\]
	for some scalars $a,b$ as in Theorem \ref{decay_estimate}. We see that the structure is mostly the same as before, only now with $K_1$ instead of $K_0$.
	
	For the second integral in \eqref{diff_int_eq} we obtain
	\[
		\begin{split}
		I^2_{l_\beta}
		=\int_0^\infty\frac{1}{s}\left(\frac{Z_\beta s-\zeta}{Z_\beta s+\zeta}\right)(d\kappa)(-Z_\beta s))\frac{\mathcal{X}^{\text{sf}}_\beta}{1-\mathcal{X}_{\beta}}ds.
	\end{split}
	\]
	From section \ref{der_coord} we know that the derivatives of $\kappa$ (w. r. t. some $R$-independent coordinates like $Z_\beta$ and $\theta_\beta$) are bounded, and for $R$ large enough $|1-\mathcal{X}_\beta|^{-1}$ is also bounded. Thus the same arguments as in the estimates before can be used to obtain upper bounds in terms of $K_0$ for this integral 
	
	Finally the last part of \eqref{diff_int_eq} is		
	\[
		\begin{split}
		I^3_{l_\beta}&=-\int_0^\infty\frac{1}{s}\left(\frac{Z_\beta s-\zeta}{Z_\beta s+\zeta}\right)(1+\kappa(-Z_\beta s))\\
		&d\left(-\frac{R}{s}\pi +i\theta_\beta-R\pi|Z_\beta|^2s\right)\frac{\mathcal{X}^{\text{sf}}_{\beta}}{1-\mathcal{X}_{\beta}}ds\\
		&=-\int_0^\infty\frac{1}{s}\left(\frac{Z_\beta s-\zeta}{Z_\beta s+\zeta}\right)(1+\kappa(-Z_\beta s))
		i\frac{\mathcal{X}^{\text{sf}}_{\beta}}{1-\mathcal{X}_{\beta}}dsd\theta_\beta\\
		&+\int_0^\infty\left(\frac{Z_\beta s-\zeta}{Z_\beta s+\zeta}\right)(1+\kappa(-Z_\beta s))
		R\pi \overline{Z_\beta}\frac{\mathcal{X}^{\text{sf}}_{\beta}}{1-\mathcal{X}_{\beta}}dsdZ_\beta\\
		&+\int_0^\infty\left(\frac{Z_\beta s-\zeta}{Z_\beta s+\zeta}\right)(1+\kappa(-Z_\beta s))
		R\pi Z_\beta\frac{\mathcal{X}^{\text{sf}}_{\beta}}{1-\mathcal{X}_{\beta}}dsd\overline{Z_\beta}.
		\end{split}
	\]
	Now the same arguments as before apply and we obtain bounds in terms of the Bessel functions, which thus give us bounds for all of $dI_{l_\beta}$.
	
	\begin{definition}
		For Det$(\Phi)=q$ let $\Gamma_q^g=H_1(\Sigma_q;\mathbb{Z})$ denote the gauge lattice and let $\mathcal{X}_{\gamma_i}$ for each generator $\gamma_i\in\Gamma^g$ be the 
		corresponding true coordinate function on $\mathcal{M}'$ constructed with triangulation angle $\vartheta:=\text{arg}\zeta$. With $\epsilon^{ij}$ 
		denoting the coefficients of the symplectic form on $\Gamma^g_q$ and $\epsilon_{ij}$ its inverse on $(\Gamma^g_q)^\ast$ define
		\[
			\varpi:=\frac{1}{8\pi^2 R}\sum_{i,j}\epsilon_{ij}
			\frac{d\mathcal{X}_{\gamma_i}}{\mathcal{X}_{\gamma_i}}\wedge\frac{d\mathcal{X}_{\gamma_j}}{\mathcal{X}_{\gamma_j}}.
		\]
	\end{definition}
	
	All of our calculations in section \ref{const_on_mod} and the arguments in \cite{Gaiotto.2010} and \cite{Gaiotto.2013} now come together in the following theorem.
	
	\begin{theorem}\label{omega_symp}
		Let $\mathcal{Z}:=\widetilde{\mathcal{M}}'_c\times\mathbb{CP}^1$ be the twistor space for $\widetilde{\mathcal{M}}'_c$ with projection $p:\mathcal{Z}\to\mathbb{CP}^1$ and let $T_F=\text{ker}dp$ denote the tangent bundle along the fibers of the twistor space. If $R$ is large enough $\varpi$ is a holomorphic section of $\Lambda^2T_F^\ast\otimes O(2)$, defining a symplectic form on each fiber.
	\end{theorem}
	
	\begin{proof}
		The results of Theorem \ref{pole_0} and Corollary \ref{pole_infty} concerning the pole structure together with the Lagrangian property of equation \ref{Griffiths_trans} ensure that $\varpi$ is a well defined section of $\Lambda^2T_F^\ast\otimes O(2)$ where the $O(2)$ twist matches the simple poles at 
		$\zeta=0$ and $\zeta=\infty$.
		
		As of Lemma \ref{NAHC_hol} any holomorphic coordinate $x_{R,\zeta}$ in the complex structure $I^{R,\zeta}$ of $\widetilde{\mathcal{M}}'_c$ becomes a holomorphic map into $\mathcal{M}_{\text{flat}}$ in complex structure $\hat{I}$. As this is the trivial complex structure on $\mathcal{M}_{\text{flat}}$ it commutes with the evaluation of $\nabla(R,\zeta)$ at a point on the Riemann surface $\mathcal{C}$. Thus the entries of the connection matrix of $\nabla(R,\zeta)$ are holomorphic in any $x_{R,\zeta}$. As explained in Remark \ref{hol_quotient} above we can take a frame in which the kernel of the corresponding integral obtains the product structure of \ref{hol_dep} and in which each factor as well as the initial value function depend holomorphically on the entries of $\nabla(R,\zeta)$. Thus
		we can infer from theorem \ref{hol_dep} that the small flat sections depend holomorphically on the coordinates $x_{R,\zeta}$. The construction of $\mathcal{X}$ then proceeds by operations which preserve the holomorphicity and so does the construction of $\varpi$. Thus the fiberwise part of $\varpi$ is holomorphic in complex structure $I^{R,\zeta}$ for all $\zeta\in\mathbb{C}^\ast$.
	
		A similar argument as above holds for the $\zeta$-dependence for any $\zeta\in\mathbb{C}^\times$ as long as the $ca$-triangulation with angle 
		$\vartheta=\text{arg}(\zeta)$ does not obtain values for which finite $ca$-trajectories appear. Thus it is piecewise holomorphic with jumps on a real 
		co-dimension-1 domain. As of proposition \ref{symp_is_cont} the jumps in the coordinates are by symplectomorphisms, so $\varpi$ is still well defined 
		and continuous on all of $\mathbb{C}^\times$. By an application of Morera's theorem it is thus holomorphic on all of $\mathbb{C}^\times$.
		Now $\varpi$ has only simple poles for $\zeta\to0$ and $\zeta\to\infty$ and is considered as a section of the twisted bundle
		a section of $\Lambda^2T_F^\ast\otimes O(2)$ it is holomorphic everywhere on $\mathbb{CP}^1$ by the Riemann removable singularity theorem.\footnote{The holomorphicity also follows from the fact, that the $\mathcal{X}$-coordinates are Fock-Goncharov coordinates, which are holomorphic on the space of flat connections. The proof here is included, to show how the concrete construction delivers this independently.}
		
		For the complex structures $I$ and $-I$, corresponding to $\zeta=0$ and $\zeta=\infty$ on $\mathbb{CP}^1$, note that we obtain the explicit form of their holomorphic symplectic structures by considering $\left.\zeta\varpi(\zeta)\right|_{\zeta=0}=-\frac{i}{2}\omega_+$ (and $\zeta=\infty$ accordingly) and plugging in the known expression for $\varpi$. For the complex structure $I$ we obtain 
		\[
			\left.\zeta\varpi(\zeta)\right|_{\zeta=0}=\frac{i}{8\pi}\sum_{i,j}\epsilon_{ij}dZ_{\gamma_i}\wedge d\theta_{\gamma_j}.
		\]
		Considering that as described in Remark \ref{hol_quotient} the explicit expression for $\theta$ may have to be adapted for an appropriate frame, we see that this is a $(2,0)$-form for the complex structure $I$. The complex structure $-I$ at $\zeta=\infty$ can be considered analogously.
		
		Lastly the form on each fiber is symplectic as of the intersection pairing which extracts the Darboux coordinate system out of the set of the 
		constructed canonical functions $\mathcal{X}_\gamma$. As the semiflat part leads to a nondegenerate form and the correction terms are exponentially suppressed in $R$  it follows that $\varpi$ is also non-degenerate for large enough $R$.
	\end{proof}
	
	%
	
	Thus the twistor theorem applies and we obtain a twistorial hyperkähler metric $g_{\text{twist}}$ on $\widetilde{\mathcal{M}}'_c$. We use this point as a reminder that 
	we use the fact that $\widetilde{\mathcal{M}}'_c$ is a hyperkähler space which thus already carries all of the structure needed in the twistor theorem. For our metric 
	we simply forget the given existing holomorphic symplectic form and replace it by the $\varpi$ from our construction, while keeping the rest of the structure, i. e. the complex 
	structure, fibration of the twistor space $\mathcal{Z}$ over $\mathbb{CP}^1$, the family of holomorphic sections with normal bundle isomorphic to 
	$\mathbb{C}^{2n}\otimes O(1)$ and the compatible real structure $\tau$. Thus we obtain the following theorem out of the twistor construction.
	
	\begin{theorem}
		For large enough $R$ there exists a hyperkähler metric $g_{\text{twist}}$ on $\widetilde{\mathcal{M}}'_c$ whose twistorial data matches the natural one except for the holomorphic symplectic form which is given by $\varpi$.
	\end{theorem}
	
	We can now use Lemma \ref{g_desc} to obtain the twistorial metric $g_{\text{twist}}$ on $\widetilde{\mathcal{M}}'_c$ as
		\[
			g_{\text{twist}}(v,w)=\omega_I(v,Iw)=\Re(\varpi^{\zeta=-1}(v,Iw)).
		\]
	
	
	The benefit of this construction becomes perfectly clear as we consider the asymptotics of this metric for large $R$. From the construction of the 
	$\mathcal{X}_{\gamma_i}$ it is clear that they can be expressed as the limiting coordinates $\mathcal{X}^{L}_{\gamma_i}$ modified by some term that decays 
	exponentially in $R$. Tanks to the explicit formulation we can easily see that this decaying behavior descends through the construction of the metric, 
	leading to the final result.
	
	\begin{theorem}
	For large enough $R$ the coefficients of the difference $g_{\text{twis}}-g_{\text{sf}}$ on $\widetilde{\mathcal{M}}'_c$ are bounded by terms of the form $\sum_n C_n(K_0(2n\pi R|Z_\beta|)+K_1(2n\pi R|Z_\beta|))$ for some constants $C_n$.
\end{theorem}

\begin{proof}
	The necessary calculations have been carried out above. With them we may now write the holomorphic symplectic form as
	\[
	\begin{split}
		\varpi(\zeta)&=\frac{1}{8\pi^2 R}\sum_{i,j}\epsilon_{ij}d\log\mathcal{X_{\gamma_i}}(\zeta)\wedge d\log\mathcal{X_{\gamma_j}}(\zeta)\\
		&=\sum_{i,j}\epsilon_{ij}\left(\frac{R}{8}dZ_{\gamma_i}\wedge d\overline{Z}_{\gamma_j}-\frac{1}{8\pi^2R}d\theta_{\gamma_i}\wedge d\theta_{\gamma_j}
		+\frac{i}{8\pi}\left(\frac{1}{\zeta}dZ_{\gamma_i}\wedge d\theta_{\gamma_j}+\zeta d\theta_{\gamma_i}\wedge d\overline{Z}_{\gamma_j}\right)\right)
		+\varpi_e(\zeta).
	\end{split}
	\]
	Here $\varpi_e(\zeta)$ is a $2$-form whose coefficients, as of the calculations above have upper bounds of the form
	\[
	\sum_n C_n(K_0(2n\pi R|Z_\beta|)+K_1(2n\pi R|Z_\beta|)),
	\]
	for some $\beta\in\Gamma_q^g$ and constants $C_n$, which are independent of $\zeta$ and $R$. The first part of this expression corresponds to the semiflat part and $\varpi_e$ denotes the so-called instanton-corrections. From this expression we obtain as explained above the corresponding expression for the instanton-correct metric
	\[
		g_{\text{twist}}=g_{\text{sf}}+g_e,
	\]
	where the coefficients of $g_e$ are bounded just like the ones from $\varpi_e(\zeta)$.
\end{proof}

From the asymptotic expression for the Bessel function one obtains different orders of correctional terms. All of the different terms can be combined to obtain a simple expression for the exponential decay of the difference of the metrics.

\begin{corollary}
	For large enough $R$ the difference $g_{\text{twist}}-g_{\text{sf}}$ on $\widetilde{\mathcal{M}}'_c$ is exponentially suppressed as $e^{-2\pi R|Z_\mu|}$, where $|Z_\mu|$ is the minimal period for which $\Omega(\mu)\neq0$.
\end{corollary}

This is precisely the decay rate that was conjectured by GMN. We stress that this decay holds in any direction in the moduli space.

\begin{remark}
	We have now constructed a Hyperkähler metric that decays exponentially towards the semiflat metric with a certain optimal decay rate. This is exactly the behaviour that is expected from the natural $L^2$ metric on $\mathcal{M}$ and it was accordingly assumed in the work of GMN that the constructed twistorial metric is actually the $L^2$-metric. At this point we are not able to rigorously prove this, but there is an outline of a possible proof. The author is grateful to Andrew Neitzke, who communicated the following ideas to him:
	
	The moduli space of flat connections is equipped with a holomorphic version of the Atiyah-Bott-Goldman symplectic form. Using the twisted nonabelian Hodge correspondence one can compare this form with the holomorphic symplectic form on the moduli space of solutions to Hitchin's equations guaranteed by the twistor theorem for the $L^2$-metric and these two forms should coincide. The Fock-Goncharov coordinates on the moduli space of flat connections on other hand should be Darboux coordinates for the Atiyah-Bott-Goldman symplectic form. Combining these arguments it would follow that the coordinates we considered here are actually Darboux coordinates for the holomorphic symplectic form that corresponds to the $L^2$-metric. As of the twistor theorem the holomorphic symplectic form determines the Hyperkähler metric, thus inferring that the constructed metric is actually the $L^2$-metric. Working out the details of this argument is subject of future research.
\end{remark}
	
	\newpage
	
\printbibliography[title={References}]

\end{document}